\documentclass[11pt]{article}
\usepackage{amssymb}
\usepackage{mathrsfs}
\usepackage{latexsym,amsmath,amssymb,amsfonts,epsfig,graphicx,cite,psfrag}
\allowdisplaybreaks
\usepackage{eepic,color,colordvi,amscd}
\definecolor{blue}{rgb}{0,0,0.9} 
\definecolor{red}{rgb}{0.9,0,0} 
\definecolor{green}{rgb}{0,0.9,0} 
\newcommand{\blue}[1]{\begin{color}{blue}#1\end{color}} 
 
\newcommand{\red}[1]{\begin{color}{red}#1\end{color}} 
 
\usepackage{ebezier} 
\usepackage{verbatim}
\usepackage{subfigure}
\usepackage{pifont}
\usepackage{array}
\usepackage{enumitem}
\usepackage{longtable}
\usepackage{algorithm}
\usepackage{algorithmic}
\usepackage{amsthm}
\usepackage{comment}
\usepackage{blkarray}
\usepackage{hyperref}
\usepackage{color}
\theoremstyle{plain}
\newtheorem{theo}{Theorem}[section]

\newtheorem{ques}[theo]{Question}
\newtheorem{lem}[theo]{Lemma}

\newtheorem{prop}[theo]{Proposition}

\theoremstyle{definition}

\theoremstyle{remark}
\newtheorem{rem}[theo]{Remark}
\def\Tr{\mathrm{Tr}}
\def\Trl{\mathrm{Tr_l}}
\def\Trr{\mathrm{Tr_r}}
\def\H{\mathcal{H}}
\def\hH{\widehat{\mathcal{H}}}
\def\<{\left\langle}
\def\>{\right\rangle}
\def\E{\mathcal{E}}

\def\A{\mathcal{A}}
\def\Hr{\mathbb{H}}
\def\B{\mathcal{B}}
\def\hA{\widehat{\mathcal{A}}}
\def\hB{\widehat{\mathcal{B}}}

\def\Lm{\mathcal{L}}
\def\F{\mathcal{F}}
\def\hF{\widehat{\mathcal{F}}}
\def\hG{\widehat{\mathcal{G}}}
\def\Q{{\bf Q}}
\def\supp{{\rm supp}}

\def\I{\mathcal{I}}

\def\R{\mathbb{R}}

\def\U{\mathcal{U}}

\def\X{\mathcal{X}}
\def\Y{\mathcal{Y}}
\newcommand{\ri}[1]{\operatorname{ri}(#1)}

\def\G{\mathcal{G}}
\def\N{\mathbb{N}}

\def\dd{{\rm diag}}

\def\P{\mathcal{P}}
\def\Ph{\hat{\mathcal{P}}}
\def\PX{\mathcal{P}_{\mathcal{X}}}
\def\PY{\mathcal{P}_{\mathcal{Y}}}

\def\O{\mathcal{O}}

\def\({\left(}
\def\){\right)}

\def\hX{\widehat{\mathcal{X}}}
\def\hY{\widehat{\mathcal{Y}}}

\def\bC{{\bf C}}
\def\hC{\widehat{C}}
\def\hc{\hat{c}}
\def\hbC{\widehat{{\bf C}}}
\def\hbc{\widehat{{\bf c}}}
\def\hbc{\hat{{\bf c}}}
\def\bQ{{\bf Q}}
\def\bq{{\bf q}}
\def\Pi{\displaystyle\prod\limits}
\def\logm{{\rm logm}}
\def\expm{{\rm expm}}

\let\svthefootnote\thefootnote
\newcommand\blankfootnote[1]{%
	\let\thefootnote\relax\footnotetext{#1}%
	\let\thefootnote\svthefootnote%
}

\begin{document}
	\title{A Bregman ADMM for Bethe variational problem}
	\author{Yuehaw Khoo
	\thanks{Department of Statistics, University of Chicago, 
	({\tt ykhoo@uchicago.edu}). The research of this author is partially funded by NSF DMS-2339439, DOE DE-SC0022232, DARPA The Right Space HR0011-25-9-0031, and a Sloan research fellowship.}, \quad
	Tianyun Tang
\thanks{Department of Statistics, University of Chicago,  ({\tt ttang@u.nus.edu}).
         }, \quad 
	 Kim-Chuan Toh\thanks{Department of Mathematics, and Institute of 
Operations Research and Analytics, National
         University of Singapore, ({\tt mattohkc@nus.edu.sg}).  The research of this author is supported
by the Ministry of Education, Singapore, under its Academic Research Fund Tier 3 grant call (MOE-2019-T3-1-010).}
	 }
	\date{\today} 
	\maketitle

\begin{abstract}
In this work, we propose a novel Bregman ADMM with nonlinear dual update to solve the Bethe variational problem (BVP), a key optimization formulation in graphical models and statistical physics. Our algorithm provides rigorous convergence guarantees, even if the objective function of BVP is non-convex and non-Lipschitz continuous on the boundary. A central result of our analysis is proving that the entries in local minima of BVP are strictly positive, effectively resolving non-smoothness issues caused by zero entries. Beyond theoretical guarantees, the algorithm possesses high level of separability and parallelizability
to achieve highly efficient subproblem computation. Our Bregman ADMM can be easily extended to solve the quantum Bethe variational problem. Numerical experiments are conducted to validate the effectiveness and robustness of the proposed method. Based on this research, we have released an open-source package of the proposed method at \url{https://github.com/TTYmath/BADMM-BVP}.
\end{abstract}

%%%%%%%%%%%%%%%%%%%%%%%%%Introduction%%%%%%%%%%%%%%%%%%%%%%%%%%%%%%%%%%%%%%%%%%%%%%%%%%%%%%%%%%%%%%

\section{Introduction}\label{Sec:intr}

%%%%%%%%%%%%%%%%%%%%%%%%%Bethe free energy and Belief propagation%%%%%%%%%%%%%%%%%%%%%%%%%%%%%%%%%%%%%%%%%%%%%%%%%%%%%%%%%%%%%%%%%%%%%%%%%%%%%%%%%%%%%%%%%%%%%%%%%%%%%

\subsection{Graphical model and belief propagation}\label{Subsec:beth} 

Graphical model \cite{lauritzen1996graphical} is a powerful framework for representing multivariate distributions, with a graph encoding the dependency structure among variables. It has applications in many fields including scientific computing \cite{perry2018message}, machine learning \cite{bach2004learning}  and statistical physics \cite{pelizzola2005cluster}. A key challenge in graphical models is probabilistic inference, which we describe in the next paragraph.

For a simple undirected connected graph $G = ([n], \E)$, consider a pairwise Markov random field (MRF) with the joint probability distribution function given by
\begin{equation}\label{MPDF}
P(x_1, x_2, \ldots, x_n) = \frac{1}{Z} \prod_{ij \in \E} \Psi_{ij}(x_i, x_j) \prod_{k \in [n]} \Psi_k(x_k),
\end{equation}
where $x_k$ is the state of node $k$ with $r$ possible values, $\Psi_{ij}(x_i, x_j)$ and $\Psi_k(x_k)$ are potential functions, and $Z$ is a normalization constant. The goal of probabilistic inference is to determine the marginal distribution of a single random variable $x_k$. While the marginal distribution can be computed directly by summing over all other variables, this approach is generally NP-hard \cite{dagum1993approximating} due to the exponential growth in the state space as $n$ increases.

\medskip

%%%%%%%%%%%%%%%%%%%%%%%belief propagation%%%%%%%%%%%%%%%%%%%%%%%%%%%%%%%%%%%%%%%%%%%%%%%%%%%%%%%%%%%%%%%%%%%%%%%%%%%%%%%%%%%%%%%%%%%%%%%%%%%%

\noindent{\bf Belief propagation.} In 1988, Pearl introduced belief propagation (BP) \cite{pearl2014probabilistic} as an approximate method to solve the probabilistic inference problem. For each edge $ij$ (In this paper, we assume $i < j$ when referring to $ij \in \E$.) in $\E$, BP defines message variables $m_{i \rightarrow j}(x_j)$ and $m_{j \rightarrow i}(x_i)$, which are updated as follows:
\begin{align}
&m_{i \rightarrow j}(x_j) \leftarrow \frac{1}{z_{i,j}} \sum_{x_i \in [r]} \Psi_{ij}(x_i, x_j) \Psi_i(x_i) \prod_{k \in N(i) \setminus \{j\}} m_{k \rightarrow i}(x_i), \label{BPij}\\
&m_{j \rightarrow i}(x_i) \leftarrow \frac{1}{z_{j,i}} \sum_{x_j \in [r]} \Psi_{ij}(x_i, x_j) \Psi_j(x_j) \prod_{k \in N(j) \setminus \{i\}} m_{k \rightarrow j}(x_j), \label{BPji}
\end{align}
where $z_{i,j}$ and $z_{j,i}$ are normalization constants that prevent the messages from diverging. This iterative process can be interpreted as passing messages along the network until convergence to a fixed point. Using these messages, the marginal distributions are approximated as:
\begin{align}
&b_k(x_k) := \frac{1}{\hat{z}_k} \Psi_k(x_k) \prod_{i \in N(k)} m_{i \rightarrow k}(x_k), \label{bk}\\
&b_{i,j}(x_i, x_j) := \frac{1}{\bar{z}_{i,j}} \Psi_{ij}(x_i, x_j) \prod_{k \in N(i) \setminus \{j\}} m_{k \rightarrow i}(x_i) \prod_{k \in N(j) \setminus \{i\}} m_{k \rightarrow j}(x_j), \label{bij}
\end{align}
where $\hat{z}_k$ and $\bar{z}_{i,j}$ are normalization constants ensuring that $b_k$ and $b_{i,j}$ are valid probability distributions. The approximations $b_k$ and $b_{i,j}$, often referred to as beliefs, provide estimates of the marginal distributions of the joint probability $P$. For tree-structured graphs, BP is guaranteed to converge to the exact marginal distributions in a finite number of iterations. However, for graphs with cycles, convergence is not guaranteed \cite{murphy2013loopy}. Nevertheless, BP often performs remarkably well as a heuristic in practice and has been successfully applied in various domains \cite{abbe2018proof,makur2024robustness,mridula2011combining}.

\subsection{Bethe variational problem}

The practical success of BP has inspired extensive theoretical investigations. Yedidia, Freeman, and Weiss achieved a key breakthrough in their seminal works \cite{yedidia2001bethe,yedidia2003understanding,yedidia2000generalized}, where they demonstrated that the fixed points of BP correspond to the critical points of the following Bethe variational problem (BVP)\cite{wainwright2008graphical}:
\begin{align}\label{BFEM}
\min\Bigg\{&\sum_{ij \in \E} \sum_{x_i, x_j \in [r]} b_{ij}(x_i, x_j) \ln \frac{b_{ij}(x_i, x_j)}{\Psi_{ij}(x_i, x_j)b_i(x_i)b_j(x_j)} 
+ \sum_{k \in V} \sum_{x_k \in [r]} b_k(x_k) \ln \frac{b_k(x_k)}{\Psi_k(x_k)} : \\
&\forall k \in [n], \sum_{x_k \in [r]} b_k(x_k) = 1, \ \forall ij \in \E, \ \sum_{x_j \in [r]} b_{ij}(x_i, x_j) = b_i(x_i), \ \sum_{x_i \in [r]} b_{ij}(x_i, x_j) = b_j(x_j) \notag
\Bigg\},
\end{align}
where $b_k$ and $b_{ij}$ are the beliefs satisfying normalization and marginalization constraints. The objective function in (\ref{BFEM}), known as the Bethe free energy (BFE), was introduced by Hans Bethe in 1935 \cite{bethe1935statistical}. It serves as a powerful approximation tool in statistical physics \cite{mackay2003information}, simplifying the computation of free energy in systems where exact calculations are computationally infeasible. Alternatively, the BFE can be interpreted as the Kullback-Leibler (KL) divergence \cite{kullback1951information} between a tree-like probability distribution 
\begin{equation}\label{treelike}
\frac{\prod_{ij \in \E} b_{ij}(x_i, x_j)}{\prod_{k \in [n]} b_k(x_k)^{d_k - 1}},
\end{equation}
where $d_k := |N(k)|$ is the degree of vertex $k$, and the joint distribution in (\ref{MPDF}). Since general probability distributions may not adhere to the form in (\ref{treelike}), the minimizer of (\ref{BFEM})—and by extension, the output of the BP algorithm—provides only an approximation to the original probabilistic inference problem.

BP can be viewed as a fixed-point method applied to the Karush-Kuhn-Tucker (KKT) conditions of (\ref{BFEM}), with the message variables corresponding to the Lagrange multipliers for the marginal constraints. The connection between BP and (\ref{BFEM}) offers an alternative approach to finding a fixed point of BP by directly solving the optimization problem (\ref{BFEM}) and obtaining its stationary point. \iffalse Various optimization methods \cite{welling2013belief,yuille2002cccp,leisenberger2024convexity,le2019nonconvex} have been proposed to tackle this problem. \fi

However, one important limitation of the earlier works by Yedidia et al.~\cite{yedidia2001bethe,yedidia2003understanding,yedidia2000generalized} is that they considered only the variables $b_{ij}$ and $b_k$ with strictly positive entries, thereby ignoring boundary points in the feasible set of (\ref{BFEM}). If the minimizer of (\ref{BFEM}) contains zero entries, it cannot serve as a fixed point for BP. However, boundary points cannot be disregarded when solving (\ref{BFEM}), as the non-convexity makes it challenging to ensure that local minima will not have zero entries. Furthermore, zero entries pose significant challenges for optimization algorithms because the objective function is not Lipschitz-continuous on the boundary. This observation motivates the following question:

\begin{ques}\label{Q1}
Do the local minima of (\ref{BFEM}) have only positive entries?
\end{ques}

In Section~\ref{Sec:land}, we provide a positive answer to Question~\ref{Q1} by demonstrating that the entries of all local minima are bounded below by a positive constant, which depends only on the coefficients in (\ref{BFEM}). Our result not only addresses a gap in the theoretical analysis of the aforementioned works but is also critical for algorithmic design and convergence analysis, as it helps overcome the non-smoothness issue on the boundary of (\ref{BFEM}). In the next subsection, we summarize several classical algorithms for solving (\ref{BVP}) and introduce our new algorithm, which offers several advantageous properties.

%%%%%%%%%%%%%%%%Boundary issue%%%%%%%%%%%%%%%%%%%%%%%%%%%%%%%%%%%%%%%%%%%%%%%%%%%%%%%%%%%%%%%%%%%%%%%%%%%%%%%%%%%%%%%%%%%%%%%%%%%%%%%%%%%%%%%%%%%%%%%%%%

%%%%%%%%%%%%%%%%Algorithm for BFEM%%%%%%%%%%%%%%%%%%%%%%%%%%%%%%%%%%%%%%%%%%%%%%%%%%%%%%%%%%%%%%%%%%%%%%%%%%%%%%%%%%%%%%%%%%%%%%%%%%%%%%%%%%%%%%%%%%%%%%%

\subsection{Algorithms for Bethe variational problem}\label{Subsec:algo}

In this subsection, we present algorithms for solving (\ref{BVP}), starting with a discussion of several classical approaches.

\medskip

\noindent{\bf Gradient descent method.} Welling and Teh \cite{welling2013belief} proposed a gradient descent (GD) method to find a stationary point of (\ref{BVP}) for $binary$ MRFs. The key insight is that for $r = 2$, problem (\ref{BVP}) can be reformulated as an unconstrained optimization problem with $n$ variables, allowing GD to be applied. Shin refined this approach in \cite{shin2012complexity,shin2014complexity} to get a polynomial time complexity bound for graphs with maximum degree $\O(\log n)$. Additionally, Leisenberger et al. \cite{leisenberger2024convexity} extended this idea using a projected quasi-Newton method to exploit second-order information. For binary pairwise MRFs with $submodular$ cost functions, global minimizers of (\ref{BVP}) have been explored using GD \cite{weller2013bethe}. Despite its theoretical soundness, the GD method is limited to binary networks. For $r > 2$, it is unclear how to reformulate the problem as unconstrained optimization, rendering GD inapplicable in such cases.

\medskip

\noindent{\bf Double-loop algorithm.} Yuille \cite{yuille2001double,yuille2002cccp} introduced a double-loop algorithm called the concave-convex procedure (CCCP) to solve (\ref{BVP}) for any $r \geq 2$. The key observation is that the objective function of (\ref{BVP}) is a difference of convex (DC) functions, enabling the application of DC programming techniques \cite{horst1999dc}. In each outer iteration of CCCP, the concave part is linearized to create a majorization. The obtained convex problem is then solved using a block coordinate minimization (BCM) method applied to its dual problem. Although CCCP can handle any $r \geq 2$, its efficiency is hindered by the iterative nature of BCM, as solving subproblems to high accuracy can be computationally expensive. Furthermore, while CCCP guarantees monotonic decreases in the objective function, the convergence of the variables remains unresolved, motivating the need for algorithms with more rigorous theoretical guarantees. The strengths and weaknesses of classical algorithms for (\ref{BVP}) are summarized in Table~\ref{tab:alg:BEM}.

\begin{table}[h]
\centering
\begin{tabular}{|>{\raggedright\arraybackslash}m{2cm}|>{\raggedright\arraybackslash}m{5cm}|>{\raggedright\arraybackslash}m{5cm}|}
\hline
\textbf{Algorithm} & \textbf{Strengths} & \textbf{Weaknesses} \\ \hline
BP & 
Converges in finite time for tree-structured graphs. &
May diverge for general graphs. \\ \hline

GD & 
Ensures convergence for binary MRFs. &
Not applicable to (\ref{BVP}) for $r > 2$. \\ \hline

CCCP & 
Applicable for any $r \geq 2$. Monotonically decreasing objective function. &
BCM is computationally expensive. No convergence guarantee of variables. \\ \hline

\end{tabular}
\caption{Comparison of classical algorithms for (\ref{BVP}).}
\label{tab:alg:BEM}
\end{table}

The limitations of classical algorithms for (\ref{BVP}) motivate the following question:

\begin{ques}\label{Q2}
Can we design an algorithm for (\ref{BVP}) that satisfies the following properties?
\begin{itemize}
    \item[1.] Applicability to (\ref{BVP}) for any $r \geq 2$.
    \item[2.] Convergence guarantee.
    \item[3.] Easy-to-solve subproblems.
\end{itemize}
\end{ques}
Because the variables in (\ref{BVP}) can be separated into two blocks, $b_{ij}$ and $b_k$, a natural approach is to apply the Alternating Direction Method of Multipliers (ADMM). However, the non-convexity and non-smoothness of the objective function in (\ref{BVP}) present significant challenges for both the implementation and convergence analysis of ADMM. In Section~\ref{Sec:BADMM}, we introduce a variant of the ADMM algorithm, called Bregman ADMM, which possesses all the desirable properties mentioned in Question~\ref{Q2}. The main innovations include the use of a KL-divergence as the penalty function and a nonlinear update of the Lagrange multiplier. These choices not only enable efficient updates but also ensure rigorous convergence guarantees in the non-convex setting.

Beyond its convergence guarantees, Bregman ADMM also features easily solvable subproblems. In Section~\ref{Sec:impl}, we show that these subproblems are highly separable, comprising many independent, small-dimensional, strongly convex, and smooth optimization problems. This separability allows for the use of parallelizable methods, offering a significant advantage over CCCP, whose subproblems rely on non-parallelizable block coordinate descent. Furthermore, in order to obtain a closed-form solution,  we propose incorporating a novel proximal term into the second subproblem.

%%%%%%%%%%%%%%%%%%%Quantum%%%%%%%%%%%%%%%%%%%%%%%%%%%%%%%%%%%%%%%%%%%%%%%%%%%%%%%%%%%%%%%%%%%%%%%%%%%%%%%%%%%%%%%%%%%%%%%%%%%%%%%%%%%%%%%%%%%%%%%%
\subsection{Quantum Bethe variational problem}
Akin to classical graphical models, Leifer and Poulin introduced quantum graphical models in \cite{leifer2008quantum}, with applications in quantum error correction \cite{poulin2008belief}, simulations of many-body quantum systems \cite{mezard2007constraint}, and approximations of ground-state energy \cite{poulin2011markov}. Hastings \cite{hastings2007quantum} further proposed quantum belief propagation (QBP) as a quantum counterpart to classical belief propagation for solving quantum graphical models. Recently, Zhao et al. \cite{zhao2024quantum} explored the connection between QBP and the quantum Bethe variational problem (\ref{QBVP}).

A related class of optimization problems involves semidefinite programming (SDP) relaxations of quantum many-body problems via quantum embedding theories \cite{sun2016quantum,georges1996dynamical,kotliar2006electronic,knizia2012density,knizia2013density}. These models are similar to (\ref{QBVP}) in that they approximate a many-body quantum system by clustering accurate local densities and combining them using local consistency constraints. In particular, Lin and Lindsey \cite{lin2022variational} proposed tightening convex relaxations by intersecting local consistency constraints with the global semidefinite constraint. More recently, Khoo and Lindsey \cite{khoo2024scalable} developed an ADMM-type method to efficiently solve these SDP problems by exploiting the parallelizability of the subproblems. However, the method proposed there does not readily work with entropy, and also there is no convergence guarantee. Compared to the model in \cite{lin2022variational}, (\ref{QBVP}) does not impose a global semidefinite constraint but instead incorporates an additional entropy function. In this work, we focus on (\ref{QBVP}), though we believe extending our approach to solving SDP problems in quantum embedding theories is a promising direction for future research.

A separate class of convex optimization problems stems from two-electron reduced density matrix (2-RDM) theory \cite{mazziotti1998contracted, mazziotti2004realization, cances2006electronic, mazziotti2012structure, zhao2004reduced, li2018semismooth, anderson2013second, nakata2001variational, deprince2010exploiting}. This approach approximates the many-body density function using low-order moments \cite{lasserre2001global}. The formulation of 2-RDM optimization problems differs significantly from both (\ref{QBVP}) and quantum embedding theories, so we do not pursue it in this work.

In Section~\ref{Sec:QT}, we demonstrate how our algorithm, Bregman ADMM, can be extended to solve (\ref{QBVP}). Since our primary focus is on solving classical Bethe variational problems, we defer the detailed formulation of (\ref{QBVP}) and the adaptation of our algorithm to Section~\ref{Sec:QT}. Numerical experiments in Section~\ref{Sec:nume} will highlight the efficiency and robustness of our method compared to various state-of-the-art algorithms for both (\ref{BVP}) and (\ref{QBVP}).

\subsection{Summary of our contributions}
We summarize our contributions as follows:
\begin{itemize} 
\item We prove that the entries of all local minima of (\ref{BVP}) are strictly positive. This result not only addresses a gap in the theoretical analysis found in the literature but also plays a critical role in algorithmic design and convergence analysis.
\item We propose a novel algorithm, Bregman ADMM, for solving (\ref{BVP}) with $r \geq 2$. The algorithm offers convergence guarantees and features simple, easily solvable subproblems. Numerical experiments are conducted to verify the efficiency and robustness of Bregman ADMM.
\item We extend the Bregman ADMM to address the Quantum Bethe Variational Problem, with numerical experiments demonstrating its strong performance.
\end{itemize}

%%%%%%%%%%%%%%%%%%%outline%%%%%%%%%%%%%%%%%%%%%%%%%%%%%%%%%%%%%%%%%%%%%%%%%%%%%%%%%%%%%%%%%%%%%%%%%%%%%%%%%%%%%%%%%%%%%%%%%%%%%%%%%%%%%%%%%%%%%%%%
\subsection{Outline of the rest of the paper}\label{Subsec:outl}
In Subsection~\ref{Subsec:nota}, we will provide useful notations and definitions that will be used throughout the paper. In Section~\ref{Sec:land}, we prove that the entries of local minima of (\ref{BVP}) are strictly positive. In Section~\ref{Sec:BADMM}, we propose our algorithm Bregman ADMM and conduct convergence analysis. In Section~\ref{Sec:impl}, we discuss some implementation details of Bregman ADMM. In Section~\ref{Sec:QT}, we will state the quantum Bethe variational problem and the Bregman ADMM for solving it. In Section~\ref{Sec:nume}, we demonstrate the efficiency of our algorithm through numerical experiments. In Section~\ref{Sec:conc}, we end this paper with a conclusion.

%%%%%%%%%%%%%%%%%%notations%%%%%%%%%%%%%%%%%%%%%%%%%%%%%%%%%%%%%%%%%%%%%%%%%%%%%%%%%%%%%%%%%%%%%%%%%%%%%%%%%%%%%%%%%%%%%%%%%%%%%%%%%%%%%%%%%%%%%%%%%%%%

\subsection{Notations and definitions}\label{Subsec:nota}

In this subsection, we introduce some notations and definitions. 
\begin{itemize}
\item For any $k\in [n],$ define $q_k:=[b_k(1);\ldots;b_k(r)],$ $c_k:=[-\ln \Psi_k(1);\ldots ;-\ln \Psi_k(r)].$ For any $ij\in \E,$ define $C_{ij},Q_{ij}\in \R^{r\times r}$ such that $C_{ij}(s,t)=-\ln \Psi_{ij}(s,t),\; Q_{ij}(s,t)=b_{ij}(s,t), \; \forall\; s,t\in [r].$ Define 
\begin{equation}\label{defiM}
M:=\max\left\{ \max_{ij\in \E}\{ \|C_{ij}\|_\infty \},\ \max_{k\in [n]}\{ 
\|c_k\|_\infty \} \right\}.
\end{equation}
We assume that $M<\infty,$ i.e., the potential functions in (\ref{MPDF}) are strictly positive.

\item Let $\H_1,\H_2,\H_3$ be the following linear spaces:
\begin{equation}\label{defiH}
\H_1:=\Pi_{k\in [n]}\R^r,\ \H_2:=\Pi_{ij\in \E}\R^{r\times r},\ \H_3:=\Pi_{ij\in \E} (\R^r\times \R^r).
\end{equation}

\item Define ${\bf c}, \bq\in \H_1,$ and ${\bf C}, \bQ\in \H_2$ such that
\begin{equation}\label{deficqCQ}
{\bf c}:= \(c_k\)_{k\in [n]},\ \bq:=\(q_k\)_{k\in [n]},\ {\bf C}:=\(C_{ij}\)_{ij\in \E},\ \bQ:=\(Q_{ij}\)_{ij\in \E},
\end{equation}
and the compact convex sets $\X\subset \H_1,$ $\Y\subset \H_2$ such that
\begin{eqnarray}
\label{defX}
\X &:=&\left\{ \bq\in \H_1:\ \forall k\in [n],\ q_k\in \R^r_+,\ {\bf 1}_r^\top q_k=1 \right\},
\\
\label{defY}
\Y &:=& \left\{ \Q\in \H_2:\ \forall ij\in \E,\ Q_{ij}\in \R^{r\times r}_+,\ \<{\bf 1}_{r\times r},Q_{ij}\>=1 \right\}.
\end{eqnarray}

We use $\ri{\cdot}$ to denote the relative interior of a set. 

\item Define the linear operators $\A:\H_1\rightarrow \H_3$ and $\B:\H_2\rightarrow \H_3$ such that for any $\bq\in \H_1$ and $\bQ\in \H_2,$ 
\begin{equation}\label{defAB}
\A(\bq):=\(-q_i,-q_j\)_{ij\in \E},\ \B(\bQ):=\(Q_{ij} {\bf 1}_r,Q_{ij}^\top {\bf 1}_r\)_{ij\in \E},
\end{equation}
and the functions $\F:\H_1\rightarrow \R$ and $\G:\H_2\rightarrow \R$ such that 
\begin{eqnarray}\label{defF}
&& \hspace{-7mm} 
\F(\bq) :=\<{\bf c},{\bf q}\>-\sum_{k\in [n]} (d_k-1)\<q_k,\log q_k\>,\G(\bQ) := \< {\bf C},{\bf Q} \>+\sum_{ij\in \E}\<\log Q_{ij},Q_{ij}\>.
\end{eqnarray}
It is easy to see that $\F(\bq)$ and $\G(\bQ)$ are smooth inside $\ri{\X}$ and $\ri{\Y}$ respectively. In addition, $\F(\bq)$ is concave in $\X$ and $\G(\bQ)$ is convex in $\Y$. 

\item For any strictly convex function $\phi,$ the Bregman divergence $B_{\phi}(\cdot,\cdot)$ \cite{bregman1967relaxation} is defined as follows:
\begin{equation}\label{Bdis}
B_\phi(x,y)=\phi(x)-\phi(y)-\<\nabla \phi(y),x-y\>,
\end{equation}
which is a nonnegative function.  We will apply the following three-point property frequently
\begin{equation}\label{3p}
\<x-y,\nabla \phi(y)-\nabla \phi(z)\>=B_\phi(x,z)-B_\phi(x,y)-B_\phi(y,z).
\end{equation}

\item

Define affine hulls
\begin{equation}\label{affhull}
{\rm aff}(\X):=\left\{ (x_k)_{k\in [n]}:\ \<x_k,{\bf 1}_r\>=1 \right\},\ {\rm aff}(\Y):=\left\{ (X_{ij})_{ij\in \E}:\ \<X_{ij},{\bf 1}_{r\times r}\>=1 \right\}.
\end{equation}

\item

Define the linear map $\PX:\ \H_1\rightarrow \H_1$ to be the orthogonal projection onto the linear space ${\rm aff}(\X-\X)=\left\{ (x_{k})_{k\in [n]}:\ \<x_k,{\bf 1}_r\>=0 \right\}.$ 
\begin{equation}\label{defiPX}
\forall x\in \H_1,\ \PX(x):={\rm proj}_{{\rm aff}(\X-\X)}(x).
\end{equation}

\item

Define the linear map $\PY:\ \H_2\rightarrow \H_2$ to be the orthogonal projection onto the linear space ${\rm aff}(\Y-\Y)=\left\{ (X_{ij})_{ij\in \E}:\ \<X_{ij},{\bf 1}_{r\times r}\>=0 \right\}:$ 
\begin{equation}\label{defiP}
\forall x\in \H_2,\ \PY(x):={\rm proj}_{{\rm aff}(\Y-\Y)}(x).
\end{equation}

\item

Define the linear map $\Ph: \H_3\rightarrow \H_3$ be the orthogonal projection map of the linear space $\B({\rm aff}(\Y-\Y))$ such that
\begin{equation}\label{defiPh}
\forall x\in \H_3,\ \Ph(x):={\rm proj}_{\B({\rm aff}(\Y-\Y))}(x).
\end{equation}

\begin{comment}
\item
Define the functions $\F:\H_1\rightarrow \R$ and $\G:\H_2\rightarrow \R$ such that 
\begin{eqnarray}\label{defF}
\F(\bq) &:=&\<{\bf c},{\bf q}\>-\sum_{k\in [n]} (d_k-1)\<q_k,\log q_k\>,
\\
\label{defG}
\G(\bQ)&:=&\< {\bf C},{\bf Q} \>+\sum_{ij\in \E}\<\log Q_{ij},Q_{ij}\>.
\end{eqnarray}
It is easy to see that $\F(\bq)$ and $\G(\bQ)$ are smooth inside $\ri{\X}$ and $\ri{\Y}$ respectively. In addition, $\F(\bq)$ is concave in $\X$ and $\G(\bQ)$ is convex in $\Y$. 

\item
Define the linear operators $\A:\H_1\rightarrow \H_3$ and $\B:\H_2\rightarrow \H_3$ such that for any $\bq\in \H_1$ and $\bQ\in \H_2,$ 
\begin{equation}\label{defAB}
\A(\bq):=\(-q_i,-q_j\)_{ij\in \E},\ \B(\bQ):=\(Q_{ij} {\bf 1}_r,Q_{ij}^\top {\bf 1}_r\)_{ij\in \E}.
\end{equation}
\end{comment}

\item Define $\gamma_{r}>0$ such that 
\begin{equation}\label{defigamma}
\gamma_{r}=\min\left\{\| \PY\circ \B^*(x) \|^2:\ x\in {\rm Im}(\B\circ \PY),\ \|x\|=1 \right\}.
\end{equation}
We have that for any $x\in \H_3$
\begin{equation}\label{frineq}
\|\Ph(x)\|\leq \frac{\|\PY\circ \B^*(\Ph(x))\|}{\sqrt{\gamma_{r}}} =\frac{\|\PY\circ \B^*(x+(\Ph(x)-x))\|}{\sqrt{\gamma_{r}}} =\frac{\|\PY\circ \B^*(x)\|}{\sqrt{\gamma_{r}}}.
\end{equation}
For any ${\bf x}=(x_k)_{k\in [n]}\in {\rm aff}(\X-\X),$ define ${\bf X}:=(X_{ij})_{ij\in \E}$ such that $X_{ij}=x_i x_j^\top.$ It is easy to see that ${\bf X}\in {\rm aff}(\Y-\Y)$ and $\A({\bf x})+\B({\bf X})={\bf 0}.$ Therefore, we have that
\begin{equation}\label{Nov_29_14}
{\rm Im}(\A\circ \PX)\subset {\rm Im}(\B\circ \PY)={\rm Im}(\Ph).
\end{equation}

\item

For any $\tau>0,$ define $\xi_{r,\tau},\hat{\xi}_{r,\tau},\tilde{\xi}_{r,\tau}>0$ as follows:
\begin{eqnarray}
\xi_{r,\tau} &:=& \frac{\sqrt{|\E|}r\( -2\log \tau+2\log 2r+5M \)}{\sqrt{\gamma_{r}}}, \label{defidelta}
\\
\hat{\xi}_{r,\tau} &:=& 
2nM+2(2|\E|-n)\log r+2\xi_{r,\tau}\sqrt{\xi_{r,\tau}^2+4|\E|(M+\log r)}
 \nonumber \\
&&+\,\frac{4\(\xi_{r,\tau}^2+4|\E|(M+\log r)\) }{\tau},
\label{defideltah}
\\[3pt]
\tilde{\xi}_{r,\tau} &:=& 4\hat{\xi}_{r,\tau}+2\sqrt{ \xi_{r,\tau}^2+4|\E|(M+\log r)}.
\label{defideltat}
\end{eqnarray}

\item 

Define 
\begin{equation}\label{defispx}
\triangle_r:=\left\{x\in \R^r:\ {\bf 1}_r^\top x=1,\ x\geq 0\right\}.
\end{equation}

\item Consider linear space ${\bf X},$ a function $f:{\bf X}\rightarrow \R\cup\{+\infty\}$ is said to have the Kurdyka-{\L}ojasiewicz (KL) property \cite{attouch2010proximal} at $x^*\in {\rm dom}( \partial f)$ if there exists $\eta\in (0,+\infty],$ a neighborhood $U$ of $x^*$ and a continuous concave function $\varphi: [0,\eta)\rightarrow \R_+$ such that:
\begin{itemize}
\item[(i)] $\varphi(0)=0$,
\item[(ii)] $\varphi$ is $C^1$ on $(0,\eta)$,
\item[(iii)] for all $s\in (0,\eta),$ $\varphi'(s)>0$,
\item[(iv)] for all $x\in U\cap \{x\in {\bf X}:\ f(x^*)<f(x)<f(x^*)+\eta\},$ the Kurdyka-{\L}ojasiewicz (KL) inequality holds
\begin{equation}\label{KL-ineq}
\varphi'(f(x)-f(x^*)){\rm dist}(0,\partial f(x))\geq 1.
\end{equation}
\end{itemize}
Functions satisfying the KL property including semi-algebraic, subanalytic and log-exp functions \cite{attouch2009convergence,attouch2013convergence}. 
\end{itemize}

%%%%%%%%%%%%%%%%%%%BVPt and BVP%%%%%%%%%%%%%%%%%%%%%%%%%%%%%%%%%%%%%%%%%%%%%%%%%%%%%%%%%%%%%%%%%%%%%%%%%%%%%%%%%%%%%%%%%%%%%%%%%%%%%%%%%%%%%%%%%%%%%%%%

\section{Local minima of (\ref{BVP})}\label{Sec:land}

\subsection{(\ref{BVP}) has no boundary local minima}
\label{Subsec:bdis}

Using the notation and definitions in Subsection~\ref{Subsec:nota}, we rewrite problem (\ref{BFEM}) as:
\begin{equation}\label{BVP}
\min\left\{ \F({\bf q}) + \G({\bf Q}) : \ \A(\bq) + \B(\bQ) = {\bf 0}, \ \bq \in \X, \ \bQ \in \Y \right\}. \tag{BVP}
\end{equation}
Yedidia et al. provided a negative answer to Question~\ref{Q1} in \cite[Theorem 9]{yedidia2005constructing}. However, their proof only addressed the case $r=2$ and claimed that the argument could be easily extended to the general case. While their contribution is significant, we argue that the case $r > 2$ is much more intricate than $r = 2$. For $r = 2$, problem (\ref{BVP}) can be reformulated as an unconstrained optimization problem with $n$ variables \cite{welling2013belief}, whose gradient has an explicit form, which significantly simplifies the perturbation analysis. In contrast, for $r \geq 3$, the constraints cannot be eliminated, making the corresponding perturbation analysis considerably more complex.

Here, we provide a complete negative answer to Question~\ref{Q1}. Furthermore, instead of merely proving the positivity of local minima, we establish a quantitative lower bound for the entries of the local minima of (\ref{BVP}). 

\begin{theo}\label{intthm}
Suppose $(\bq^*, \bQ^*)$ is a local minimizer of (\ref{BVP}). Then $\bq^* \geq \sigma$ and $\bQ^* \geq \sigma^2 / (r^2 \exp[4M])$, where 
\begin{equation}\label{defisig}
\sigma := \exp\left[ -2d(G)r(\log r + 1.5M) - 2Mr - 2\log r \right],
\end{equation}
and $M$ and $d(G)$ are as defined in (\ref{defiM}) and the maximum degree of 
$G$, respectively.
\end{theo}

Theorem~\ref{intthm} implies that problem (\ref{BVP}) shares the same set of minimizers as the following modified problem, provided that $\tau \in (0, \sigma)$:
\begin{equation}\label{BVPt}
\min\left\{ \F({\bf q}) + \G({\bf Q}) : \ \A(\bq) + \B(\bQ) = {\bf 0}, \ \bq \in \X_\tau, \ \bQ \in \Y \right\}, \tag{BVP$_\tau$}
\end{equation}
where $\X_\tau$ is defined as 
\begin{equation}\label{defiXt}
\X_\tau := \left\{\bq \in \X : \bq \geq \tau \right\}.
\end{equation}
In \eqref{defiXt} and other places throughout this paper, the 
notation such as $\bq \geq \tau$ means that all the entries of $\bq$ are greater than or equal to $\tau$.

The main advantage of solving (\ref{BVPt}) lies in the constraint $\bq \geq \tau$, which prevents algorithmic iterations from producing zero entries. This constraint plays a crucial role in the convergence analysis of our algorithms. Since finding local minima for non-convex optimization problems is generally NP-hard \cite{quadNPhard}, we establish the following equivalence result regarding the stationarity of (\ref{BVP}) and (\ref{BVPt}):

\begin{prop}\label{eqprop}
Suppose $\tau \in (0, \sigma)$, where $\sigma$ is defined in (\ref{defisig}). For any stationary point $(\bq^*, \bQ^*)$ of (\ref{BVPt}) such that $\bQ^* \in \ri{\Y}$, we have $\bq^* > \tau$, and $(\bq^*, \bQ^*)$ is also a stationary point of (\ref{BVP}).
\end{prop}

Proposition~\ref{eqprop} enables us to find stationary points of (\ref{BVP}) by solving the modified problem (\ref{BVPt}) using first-order methods. In the following several subsections, we will prove Theorem~\ref{intthm} and Proposition~\ref{eqprop}.

\subsection{Properties of entropy regularized optimal transport}

Before we prove Theorem~\ref{intthm} and Proposition~\ref{eqprop}, we first prove several lemmas about entropy regularized optimal transport (OT) problem \cite{cuturi2013sinkhorn}.

\begin{lem}\label{lemenOT}
Consider the following entropy regularized OT problem

\begin{equation}\label{enOT}
\min_{X\in \R^{r\times r}_+}\left\{ \<C,X\>+\<\log X,X\>:\ X{\bf 1}_r=u,\ X^\top {\bf 1}_r=v \right\},
\end{equation}
where $u,v>0$ and ${\bf 1}_r^\top u={\bf 1}_r^\top v=1.$ Then the optimal solution $X^*$ can be written as:

\begin{equation}\label{optXenOT}
X^*:=\exp \left[ -C+\lambda {\bf 1}_r^\top+{\bf 1}_r\mu^\top  \right],
\end{equation}
where $\lambda,\mu\in \R^r$ satisfy that
\begin{eqnarray}
\|\lambda-\log u\|_\infty &\leq & 1.5\|C\|_\infty+\log r,
\label{ranl}
\\
\|\mu-\log v\|_\infty &\leq & 1.5\|C\|_\infty+\log r.
\label{ranm}
\end{eqnarray}
In addition, we have that
\begin{equation}\label{minX}
\min_{1\leq i,j\leq r}\{X^*(i,j)\}\geq \frac{\big(\min_{1\leq i,j\leq r}\{u(i),v(j)\}\big)^2}{r^2 \exp[4\|C\|_\infty]}.
\end{equation}
\end{lem}

\begin{proof}
From Lemma 2 in \cite{cuturi2013sinkhorn},we know that the optimal solution $X^*$ of (\ref{enOT}) can be written in the form (\ref{optXenOT}) for some $\lambda,\mu\in \R^r.$ Note that the Lagrangian multipliers $\lambda,\mu$ are not unique because for any $\alpha\in \R,$ $\lambda':=\lambda-\alpha {\bf 1}_r,$ $\mu':=\mu+\alpha {\bf 1}_r$ also satisfy the equation (\ref{optXenOT}). Thus, without loss of generality, we assume $\max_{i\in [r]}\{\lambda(i)\}=\max_{i\in [r]}\{\mu(i)\}.$ Let $a,b\in [r]$ such that $\lambda(a):=\max_{i\in [r]}\{\lambda(i)\},$ $\mu(b):=\max_{i\in [r]}\{\mu(i)\}.$ From (\ref{optXenOT}), we have that
\begin{equation*}\label{Nov_24_1}
1\geq X(a,b)=\exp[-C(a,b)+\lambda(a)+\mu(b)]\geq \exp[-\|C\|_\infty+\lambda(a)+\mu(b)],
\end{equation*}
from which we have
\begin{equation}\label{Nov_24_2}
\max_{i\in [r]}\{\lambda(i)\}=\max_{i\in [r]}\{\mu(i)\}\leq \|C\|_\infty/2.
\end{equation}
Also, we have that 
\begin{equation*}\label{Nov_24_2.5}
1/r^2\leq \max_{1\leq i,j\leq r}\{X(i,j)\}\leq \exp\left[\|C\|_\infty+\max_{i\in [r]}\{\lambda(i)\}+\max_{i\in [r]}\{\mu(i)\}\right],
\end{equation*}
from which we have
\begin{equation}\label{Nov_24_2.6}
\max_{i\in [r]}\{\lambda(i)\}=\max_{i\in [r]}\{\mu(i)\}\geq -\log r-\|C\|_\infty/2.
\end{equation}
Now, from (\ref{optXenOT}) and (\ref{Nov_24_2}), we have that
\begin{align*}\label{Nov_24_3}
u=\exp\left[ -C+\lambda {\bf 1}_r^\top+{\bf 1}_r \mu^\top \right]{\bf 1}_r &\leq \exp\left[ \| C \|_\infty {\bf 1}_{r\times r}+\lambda {\bf 1}_r^\top+ \|C\|_\infty  {\bf 1}_{r\times r}/2 \right]{\bf 1}_r
\\ \notag
&=\exp(\lambda)\circ \(r\exp(1.5\|C\|_\infty) {\bf 1}_r\).
\end{align*}
Computing $\log(\cdot)$ of the above inequality, we get
\begin{equation}\label{Nov_24_3.5}
\log u-(\log r+1.5\|C\|_\infty){\bf 1}_r \leq \lambda.
\end{equation}
From (\ref{optXenOT}) and (\ref{Nov_24_2.6}), we have that 
\begin{align*}\label{Nov_24_3.6}
&u=\exp\left[ -C+\lambda {\bf 1}_r^\top+{\bf 1}_r \mu^\top \right]{\bf 1}_r=\exp(\lambda)\circ  \(\exp\left[ -C \right]\exp[\mu]\)\\ \notag
&\geq \exp(\lambda)\circ  \(\exp\left[ -\|C\|_\infty {\bf 1}_{r\times r} \right]\exp[\mu]\)\geq \exp(\lambda)\circ \( \exp\left[ -\|C\|_\infty {\bf 1}_r \right]\cdot\exp\left[\max_{i\in [r]}\{\mu(i)\}\right] \)\notag \\
&\geq \exp(\lambda)\circ \exp \left[ -(1.5\|C\|_\infty+\log r){\bf 1}_r \right] \notag.
\end{align*}
Computing $\log(\cdot)$ of the above inequality, we get
\begin{equation}\label{Nov_24_3.7}
\lambda\leq \log u+(1.5\|C\|_\infty+\log r){\bf 1}_r.
\end{equation}
Combining (\ref{Nov_24_3.5}) and (\ref{Nov_24_3.7}), we get (\ref{ranl}). Similarly, we can prove (\ref{ranm}). 

From (\ref{optXenOT}), (\ref{ranl}) and (\ref{ranm}), we have that
\begin{equation}
\min_{1\leq i,j\leq r}\{X^*(i,j)\}\geq \exp\left[-\|C\|_\infty+2\( -\log r +\log\(\min_{1\leq i,j\leq r}\{u(i),v(j)\}\)-1.5\cdot \|C\|_\infty \)\right], \notag
\end{equation}
which implies (\ref{minX}). 
\end{proof}

\subsection{Equivalent formulation of (\ref{BVP})} 

Before we prove Theorem~\ref{intthm}, we first consider an equivalent reformulation of (\ref{BVP}). When $\bq=(q_k)_{k\in [n]} \in \X$ is fixed, the problem (\ref{BVP}) becomes several independent entropy regularized OT problems as follows:
\begin{equation}\label{ijenOT}
\forall ij\in \E,\ {\rm L}_{ij}(q_i,q_j):=\min\left\{ \<C_{ij},Q_{ij}\>+\<\log Q_{ij},Q_{ij}\>:\ Q_{ij}{\bf 1}_r=q_i,\ Q_{ij}^\top {\bf 1}_r=q_j \right\}.
\end{equation}
From Proposition 4.6 of \cite{peyre2019computational}, we know that $L_{ij}:\triangle_r\times \triangle_r\rightarrow \R$ is a convex function and for any $(q_i,q_j)\in \ri{\triangle_r}\times \ri{\triangle_r},$ it is differentiable with gradient given as follows:
\begin{equation}\label{gradFij}
\nabla {\rm L}_{ij}(q_i,q_j)=\begin{bmatrix} (I_r-{\bf 1}_{r\times r}/r)\lambda_{ij}\\ (I_r-{\bf 1}_{r\times r}/r)\mu_{ij} \end{bmatrix},
\end{equation}
where $(\lambda_{ij},\mu_{ij})\in \R^r\times \R^r$ is any optimal dual variable of (\ref{ijenOT}) and $(I_r-{\bf 1}_{r\times r}/r)$ is the projection map on the linear space $\{x\in \R^r:\ e^\top x=0\}.$ (\ref{BVP}) is equivalent to the following problem:
\begin{equation}\label{eqBVP}
\min\left\{ \Lm(\bq)+\F(\bq):\ \bq\in \X\right\},
\end{equation}
where $\Lm(\cdot):\X\rightarrow \R$ is defined by $\Lm(\bq):=\sum_{ij\in \E}{\rm L}_{ij}(q_i,q_j).$ With (\ref{eqBVP}), we now prove the following lemma, which is about the descent direction for a point $\bq$ with small entries.
\begin{lem}\label{desclem}
Suppose $\bq^*\in \X$ such that $\tau:=\min\{\bq^*\}<\sigma$, where $\sigma$ is defined in (\ref{defisig}). Then there exists a descent feasible direction ${\bf v}\in \H_1$ of (\ref{eqBVP}) and $\delta>0$ such that
\begin{itemize}
\item[(i)] For any $t\in (0,\delta),$ $\bq^*+t{\bf v}\geq \tau,$
\item[(ii)] For any $t\in (0,\delta),$ $\Lm(\bq^*+t{\bf v})+\F(\bq^*+t{\bf v})<\Lm(\bq^*)+\F(\bq^*),$
\item[(iii)] If $\bq^*>0,$ then the directional derivative $D_{{\bf v}}\(\Lm(\bq^*)+\F(\bq^*)\)<0.$ 
\end{itemize}
\end{lem}

\begin{proof}
Suppose $\bQ^*=(Q^*_{ij})_{ij\in \E}$ are global minimizers of the convex problems (\ref{ijenOT}) with given marginals $\bq^*.$ Define $\I\subset [n]$ such that
\begin{equation*}\label{defiI}
\I:=\left\{ k\in [n]:\ \exists\ i\in [r]\ {\rm s.t.}\ q^*_k(i)<\sigma \right\}.
\end{equation*}
For any $k\in \I,$ define $J_k\subset [r]$ and $h_k\in [r]$ such that
\begin{equation*}\label{defiJ}
J_k:=\left\{ i\in [r]:\ q^*_k(i)<\sigma \right\},\ q^*_k(h_k)> 1/r;
\end{equation*}
Because $\sigma\in (0,1/r),$ for any $k\in \I,$ $1\leq |J_k|\leq r-1.$ Define the direction ${\bf v}=(v_k)_{k\in [n]}\in \H_1$ such that 
\begin{equation}\label{defid}
\forall k\in [n]\setminus \I,\ v_k:={\bf 0}_r,\ \forall k\in \I,\ v_k(i):=\begin{cases} 
\frac{1}{|J_k|+1} & i\in J_k \\[3pt] 
\frac{-|J_k|}{|J_k|+1} & i= h_k \\[3pt]
0 & i\in [r]\setminus (J_k\sqcup\{h_k\}).
\end{cases}.
\end{equation}
It is easy to see that for any $k\in [n],$ ${\bf 1}_r^\top v_k=0.$ Moreover, there exists $\delta>0$ such that for any $t\in (0,\delta),k\in \I, i\in J_k$ 
\begin{equation}\label{stsmall}
\bq^*+t {\bf v}\in \ri{\X},\ \bq^*+t {\bf v}\geq \tau,\ q^*_k(i)+tv_k(i)<\sigma,\ q^*_k(h_k)+t v_k(h_k)> 1/r. 
\end{equation}
Therefore, ${\bf v}$ is a feasible direction of (\ref{eqBVP}) that satisfies (i). Next, we proceed to prove (ii). For any $t\in (0,\delta),\ \ij\in \E,$ consider the following entropy OT problems:
\begin{equation}\label{ijenOTt}
\min\left\{ \<C_{ij},Q_{ij}\>+\<\log Q_{ij},Q_{ij}\>:\ Q_{ij}{\bf 1}_r=q^*_i+tv_i,\ Q_{ij}^\top {\bf 1}_r=q^*_j+t v_j \right\}.
\end{equation}
From Lemma~\ref{lemenOT}, we can choose the optimal dual variables $\lambda_{ij}(t),\mu_{ij}(t)\in \R^r$ such that
\begin{eqnarray}
\|\lambda_{ij}(t)-\log (q^*_i+tv_i)\|_\infty &\leq& \log r+1.5M,
\label{ranlij}
\\
\|\mu_{ij}(t)-\log (q^*_j+t v_j)\|_\infty &\leq & \log r+1.5M,
\label{ranmij}
\end{eqnarray}
where $M$ is defined in (\ref{defiM}). From (\ref{defF}), (\ref{gradFij}) and (\ref{stsmall}), we have that for any $t\in (0,\delta),$
\begin{align}\label{multineq}
&\frac{{\rm d}}{{\rm dt}}\( \Lm(\bq^*+t{\bf v})+\F(\bq^*+t{\bf v}) \)\\ \notag
&=\sum_{k\in [n]} \Big(\<c_k,v_k\> -(d_k-1)\<\log (q^*_k+t v_k)+{\bf 1}_r,v_k\>+\sum_{kj\in \E}\<\lambda_{kj}(t),v_k\>+\sum_{ik\in \E}\<\mu_{ik}(t),v_k\>\Big) \\ \notag
&=\sum_{k\in \I} \Big(\<c_k,v_k\> -(d_k-1)\<\log (q^*_k+t v_k),v_k\>+\sum_{kj\in \E}\<\lambda_{kj}(t),v_k\>+\sum_{ik\in \E}\<\mu_{ik}(t),v_k\>\Big) \\ \notag
&\leq \sum_{k\in \I} \Big(\<c_k,v_k\> -(d_k-1)\<\log (q^*_k+t v_k),v_k\> +d_k\<\log (q^*_k+tv_k),v_k\>+d_k r(\log r+1.5M) \Big) \\ \notag
&=\sum_{k\in \I} \big(\<c_k,v_k\> +\<\log (q^*_k+t v_k),v_k\>+d_k r(\log r+1.5M) \big)\\ \notag 
&<\sum_{k\in \I} \big(\<c_k,v_k\>+d_k r(\log r+1.5M) \big)+\sum_{k\in \I} \( \frac{\(\log \sigma\)|J_k|}{|J_k|+1}+\frac{-|J_k|\log \(1/r\)}{|J_k|+1}
\) \\ \notag
&\leq |\I|d(G)r(\log r+1.5M)+|\I|Mr+(|\I|/2)\log \sigma+|\I| \log r\leq 0\notag,
\end{align}
where the second equality comes from that $v_k={\bf 0}_r$ for any $k\in [n]\setminus \I$ and ${\bf 1}_r^\top v_k=0$ for any $k\in [n].$ The first inequality comes from (\ref{ranlij}) and (\ref{ranmij}), the second inequality comes from (\ref{defid}) and (\ref{stsmall}) and the last inequality comes from (\ref{defisig}). Because $\frac{{\rm d}}{{\rm dt}}\( \Lm(\bq^*+t{\bf v})+\F(\bq^*+t{\bf v}) \)<0$ for any $t\in (0,\delta),$ we have that $\Lm(\bq^*+t{\bf v})+\F(\bq^*+t{\bf v})$ is monotonically decreasing on $(0,\delta),$ which implies (ii). 

If $\bq^*>0,$ then (\ref{multineq}) also holds for $t=0.$ This implies (iii).
\end{proof}

With all the preparations, we now prove Theorem~\ref{intthm} and Proposition~\ref{eqprop} in the next two subsections.

%%%%%%%%%%%%%%%%%%%Proof of Theorem intthm%%%%%%%%%%%%%%%%%%%%%%%%%%%%%%%%%%%%%%%%%%%%%%%%%%%%%%%%%%%%%%%%%%%%%%%%%%%%%%%%%%%%%%%%%%%%%%%%%%%%%%%%%
\subsection{Proof of Theorem~\ref{intthm}}

\begin{proof}
Because  $(\bq^*,\bQ^*)$ is a local minimizer of (\ref{BVP}), we have that $\bQ^*=(Q^*_{ij})_{ij\in \E}$ are global minimizers of the convex problems (\ref{ijenOT}) with given marginals $\bq^*.$ Assume on the contrary that $\bq^*\geq \sigma$ doesn't hold. From Lemma~\ref{desclem}, there exists a descent feasible direction ${\bf v}$ at $\bq^*$ for (\ref{eqBVP}). Therefore, $\bq^*$ is not a local minimizer of (\ref{eqBVP}). From the continuity of the optimal solution of the entropy OT problem (\ref{ijenOTt}) (see \cite[Theorem 3.11]{eckstein2022quantitative}), we know that $(\bq^*,\bQ^*)$ is also not a local minimizer of (\ref{BVP}). We get a contradiction. Therefore, $\bq^* \geq \sigma.$ From (\ref{minX}) in Lemma~\ref{lemenOT}, we have that $\bQ^*\geq \sigma^2/(r^2\exp[4M]).$
\end{proof}

\subsection{Proof of Proposition~\ref{eqprop}}

\begin{proof}
Because $(\bq^*,\bQ^*)$ is a stationary point of (\ref{BVPt}) such that $\bQ^*\in \ri{\Y},$ we know that $\bq^*\geq \tau.$ Moreover, when $\bq$ is fixed as $\bq^*,$ (\ref{BVPt}) will be a convex optimization problem with respect to $\bQ$. From the stationarity, we have that $\bQ^*$ is the optimal solution of (\ref{BVPt}) for given $\bq^*.$ From the definition of the problem (\ref{eqBVP}), we know that $\bq^*$ is a stationary point of the following optimization problem:
\begin{equation}\label{eqBVP1}
\min\left\{ \Lm(\bq)+\F(\bq):\ \bq\in \X_\tau \right\}.
\end{equation}
We assume on the contrary that $\min\{\bq\}=\tau<\sigma$. Then from Lemma~\ref{desclem}, there exists descent feasible direction ${\bf v}$ of (\ref{eqBVP}) at $\bq^*$ such that (i), (ii), (iii) hold. From (i), ${\bf v}$ is also a descent feasible direction of (\ref{eqBVP1}) at $\bq^*.$ From (iii), we know that $\bq^*$ is not a stationary point of (\ref{eqBVP1}). We get a contradiction. Therefore, $\bq^*> \tau$ and it is obvious that $\bq^*$ is a stationary point of (\ref{eqBVP}). This implies that $(\bq^*,\bQ^*)$ is a stationary point of (\ref{BVP}). 
\end{proof}

%%%%%%%%%%%%%%%%%%%Convergence analysis%%%%%%%%%%%%%%%%%%%%%%%%%%%%%%%%%%%%%%%%%%%%%%%%%%%%%%%%%%%%%%%%%%%%%%%%%%%%%%%%%%%%%%%%%%%%%%%%%%%%%%%%%%%%%%%%%%%%%

\section{Bregman ADMM}\label{Sec:BADMM}

\subsection{Algorithmic framework}

In this subsection, we present our algorithm for solving (\ref{BVPt}). The alternating direction method of multipliers (ADMM) \cite{gabay1976dual,glowinski1975approximation} is widely used in optimization due to its simplicity and broad applications \cite{boyd2011distributed,li2016schur,tang2024self,sun2024accelerating}. The traditional ADMM algorithm for solving (\ref{BVPt}) is outlined in Algorithm~\ref{alg:ADMM}.

\begin{algorithm}
	\renewcommand{\algorithmicrequire}{\textbf{Input:}}
	\renewcommand{\algorithmicensure}{\textbf{Output:}}
	\caption{ADMM for (\ref{BVPt})}
	\label{alg:ADMM}
	\begin{algorithmic}
		\STATE {\bf Initialization}: Choose $\rho > 0$, $\bq^1 \in \ri{\X}, \bQ^1 \in \ri{\Y}, \Lambda^1 \in \H_3$.
		\FOR{$t = 1, 2, \dots$}
		\STATE 1. $\bq^{t+1} \in \arg\min_{\bq \in \X_\tau} \left\{ \F(\bq) - \< \Lambda^t, \A(\bq) \> + \frac{\rho}{2} \| \A(\bq) + \B(\bQ^t) \|^2 \right\}$
		\STATE 2. $\bQ^{t+1} \in \arg\min_{\bQ \in \Y} \left\{ \G(\bQ) - \< \Lambda^t, \B(\bQ) \> + \frac{\rho}{2} \| \A(\bq^{t+1}) + \B(\bQ) \|^2 \right\}$
		\STATE 3. $\Lambda^{t+1} := \Lambda^t - \rho ( \A(\bq^{t+1}) + \B(\bQ^{t+1}) )$
		\ENDFOR
	\end{algorithmic}
\end{algorithm}

While straightforward, Algorithm~\ref{alg:ADMM} uses quadratic penalty terms, which can be problematic because $\F(\bq)$ is strongly concave with its Hessian's smallest eigenvalue having the order of $-\Omega(1/\tau)$ in $\X_\tau$. When $\tau$ is small, $\rho$ must be chosen very large, $\rho = \Omega(1/\tau)$, to ensure that the subproblem in Step 1 remains convex, which makes the algorithm computationally inefficient. When the subproblem is non-convex, finding its global minimizer can be challenging. Even if specialized algorithms could succeed in identifying the global minimizer, the solution may lie on the boundary, which is not a critical point of (\ref{BVPt}).

In \cite{wang2014bregman}, Wang and Banerjee proposed a Bregman ADMM, which uses the Bregman divergences $B_\phi(\cdot,\cdot)$ and $B_\psi(\cdot,\cdot)$ as penalty function and proximal term, as shown in Algorithm~\ref{alg:BADMM}.

\begin{algorithm}
	\renewcommand{\algorithmicrequire}{\textbf{Input:}}
	\renewcommand{\algorithmicensure}{\textbf{Output:}}
	\caption{Bregman ADMM with linear dual update for (\ref{BVPt}) \cite{wang2014bregman}}
	\label{alg:BADMM}
	\begin{algorithmic}
		\STATE {\bf Initialization}: Choose $\rho > 0$, $\bq^1 \in \ri{\X}, \bQ^1 \in \ri{\Y}, \Lambda^1 ={\bf 0}\in \H_3$.
		\FOR{$t = 1, 2, \dots$}
		\STATE 1. $\bq^{t+1} \in \arg\min_{\X_\tau} \left\{ \F(\bq) - \< \Lambda^t, \A(\bq) \> + \rho B_\phi(-\A(\bq), \B(\bQ^t))+B_\psi(\bq,\bq^t) \right\}$
		\STATE 2. $\bQ^{t+1} \in \arg\min_{\Y} \left\{ \G(\bQ) - \< \Lambda^t, \B(\bQ) \> + \rho B_\phi(\B(\bQ), -\A(\bq^{t+1})) \right\}$
		\STATE 3. $\Lambda^{t+1} := \Lambda^t - \rho ( \A(\bq^{t+1}) + \B(\bQ^{t+1}) )$
		\ENDFOR
	\end{algorithmic}
\end{algorithm}
In this paper, because of the special structure of (\ref{BVP}), we choose the function $\phi:\H_3\rightarrow \R,$ and $\psi:\H_1\rightarrow \R$ ($\H_1,\H_3$ are defined in (\ref{defiH})) such that
\begin{equation}\label{defiphipsi}
\forall {\bf x}\in \H_3,\ \phi({\bf x}):=\<{\bf x},\log {\bf x}\>,\quad \forall y\in \H_1,\ \psi({\bf y}):=\sum_{k\in [n]}(d_k-1)\<y_k,\log y_k\>,
\end{equation}
where $d_k$ is the degree of vertex $k.$

\begin{rem}\label{notstricrem}
From the definition of $\psi$ in (\ref{defiphipsi}), it may not be strictly convex if $d_k=1$ for some $k\in [n].$ However, this won't affect our theoretical analysis in this paper. The reason is that the Bregman divergence $B_{\phi}(-\A(\bq),\B(\Q^t))$ is strongly convex with respect to $\bq$.
\end{rem}

In Algorithm~\ref{alg:BADMM}, the subproblem in Step 1 is convex as long as $\rho >0$, contrasting with the stricter requirements in Algorithm~\ref{alg:ADMM}. The convergence of Bregman-type methods such as the Bregman proximal gradient method \cite{hanzely2021accelerated}, stochastic 
Bregman proximal gradient method \cite{ding2023nonconvex} and Bregman proximal point method \cite{yang2022bregman} has been well-studied in the literature. However, the theoretical analysis of Bregman ADMM is still in its infancy. Although Wang and Banerjee proved the convergence of their Bregman ADMM for convex problems, its extension to non-convex problems like (\ref{BVPt}) lacks a theoretical guarantee. While nonconvex ADMM variants have been studied \cite{wang2018convergence,pham2024bregman,liu2023bregman}, they often use quadratic penalties, with the Bregman divergence appearing only as a proximal term. Here, we refer to a Bregman ADMM as an ADMM that uses Bregman divergence as the penalty function, and the convergence of such a method for non-convex problems remains an open question.

\medskip

\noindent {\bf Theoretical challenges:} Analyzing the convergence of Algorithm~\ref{alg:BADMM} involves two key difficulties:
\begin{itemize}
    \item The asymmetry of the Bregman divergence penalty function
    complicates the convergence proofs. In contrast, the convergence proofs of the usual ADMM that employs a quadratic penalty function do not encounter such difficulties. This fact alone hinders a direct adaptation of the proof in Wang and Banerjee \cite{wang2014bregman} for the convergence of convex Bregman ADMM.
    \item The marginal $\G(\bQ)$ defined in (\ref{defF}) is not Lipschitz-continuous on the boundary of $\Y$ defined in (\ref{defY}). However, in existing analyses of non-convex ADMM \cite{wang2019global,yang2017alternating,hong2016convergence,liu2019linearized}, the objective function (whose variable is optimized) in the second subproblem is typically assumed to have a Lipschitz-continuous gradient (When there are more than two blocks, this assumption is applied to the last function.), which ensures that the Lagrangian multiplier can be bounded in terms of the primal variables.
\end{itemize}

To address these challenges, we introduce a Bregman ADMM with nonlinear dual update as stated in Algorithm~\ref{alg:BADMMN}.

\begin{algorithm}
	\renewcommand{\algorithmicrequire}{\textbf{Input:}}
	\renewcommand{\algorithmicensure}{\textbf{Output:}}
	\caption{Bregman ADMM with nonlinear dual update for (\ref{BVPt})}
	\label{alg:BADMMN}
	\begin{algorithmic}
		\STATE {\bf Initialization}: Choose $\rho > 0$, $\bq^1 \in \ri{\X}, \bQ^1 \in \ri{\Y}, \Lambda^1 = {\bf 0}\in \H_3$.
		\FOR{$t = 1, 2, \dots$}
		\STATE 1. $\bq^{t+1} \in \arg\min_{\X_\tau} \left\{ \F(\bq) - \< \Lambda^t, \A(\bq) \> + \rho B_\phi(-\A(\bq), \B(\bQ^t))+B_\psi(\bq,\bq^t) \right\}$
		\STATE 2. $\bQ^{t+1} \in \arg\min_{\Y} \left\{ \G(\bQ) - \< \Lambda^t, \B(\bQ) \> + \rho B_\phi(\B(\bQ), -\A(\bq^{t+1})) \right\}$
		\STATE 3. $\Lambda^{t+1} := \Lambda^t - \rho \( \nabla \phi(\B(\bQ^{t+1})) - \nabla \phi(-\A(\bq^{t+1})) \)$
		\ENDFOR
	\end{algorithmic}
\end{algorithm}

Algorithm~\ref{alg:BADMMN} modifies the dual update to a nonlinear form. This nonlinear dual update is more natural than the linear update in Algorithm~\ref{alg:BADMM} because it aligns with the optimality conditions of the second subproblem, i.e., the gradient of the function in the second subproblem is zero after replacing $\Lambda^t$ by $\Lambda^{t+1}.$ As far as we are aware of, such dual 
update in ADMM has not been considered in the past.
The nonlinear dual update in Algorithm~\ref{alg:BADMMN} plays a crucial role both theoretically and practically. Theoretically, it simplifies the task of bounding the dual variable $\Lambda$. Practically, as observed from our numerical experiments, Algorithm~\ref{alg:BADMMN} demonstrates strong robustness in solving (\ref{BVPt}), whereas Algorithm~\ref{alg:BADMM} rarely achieves convergence. The convergence result of Algorithm~\ref{alg:BADMMN} is stated in the following theorem.

\begin{theo}\label{convthm}
Suppose $\bq^1 \in \X_\tau$, $\bQ^1 \in \ri{\Y}$, $\Lambda^1 = {\bf 0}$, and 
\begin{equation}\label{realranbeta}
\rho \geq \max\left\{ \frac{16(\xi_{r,\tau}^2 + 4|\E|(M + \log r))}{\tau^2}, \frac{128r^4 \exp(8M)}{\gamma_{r}\tau^4}, \frac{64\tilde{\xi}_{r,\tau}^2}{\tau^4} \right\},
\end{equation}
where $\gamma_{r}, \xi_{r,\tau}, \hat{\xi}_{r,\tau}, \tilde{\xi}_{r,\tau}$ are defined in (\ref{defigamma}), (\ref{defidelta}), (\ref{defideltah}), (\ref{defideltat}). Then, the sequence $(\bq^t, \bQ^t, \Ph(\Lambda^t))_{t \in \N^+}$ generated by Algorithm~\ref{alg:BADMMN} converges to a KKT point of (\ref{BVPt}).
\end{theo}

The key insights for proving convergence are:
\begin{itemize}
    \item The asymmetry of the Bregman divergence introduces only a higher-order error term, which can be bounded.
    \item The non-smoothness of $\G(\bQ)$ is mitigated by ensuring that the iterates $\{\bQ^k\}$ remain bounded away from zero when $\rho$ is sufficiently large. In this case, $\G(\bQ)$ has a Lipschitz-continuous gradient in the convex hull of the iterates.
\end{itemize}
In the following subsections, we will prove Theorem~\ref{convthm} on the convergence
of Algorithm~\ref{alg:BADMMN}.

\subsection{Lemmas for the convergence analysis of Algorithm~\ref{alg:BADMMN}}
\label{Sec:conv}

We first study the optimality conditions of the subproblems in Algorithm~\ref{alg:BADMMN}. Because of the constraint $\bq\geq \tau,$ $\bq^t\in \ri{\X}.$ Because of the entropy function in $\G(\bQ),$ we know that $\bQ^{t}\in \ri{\Y}.$ Therefore, Step 3 in Algorithm~\ref{alg:BADMMN} is well-defined. The optimality conditions of the subproblems are:
\begin{eqnarray}
{\bf 0} &\in & \nabla \F\(\bq^{t+1}\)+\partial\delta_{\X_\tau}(\bq^{t+1})-\A^*(\Lambda^t)+\rho\A^*\( -\nabla \phi(-\A(\bq^{t+1}))+\nabla \phi(\B(\bQ^{t}))  \)\nonumber 
\\
&&+\nabla \psi(\bq^{t+1})-\nabla \psi(\bq^t), 
\label{KKTq} 
\\[5pt]
{\bf 0} &\in& \nabla \G(\bQ^{t+1})+\partial \delta_{ \Y }(\bQ^{t+1})-\B^*(\Lambda^t)+\rho\B^*\( \nabla \phi(\B(\bQ^{t+1}))-\nabla \phi(-\A(\bq^{t+1})) \).
\label{KKTQ}
\end{eqnarray}
From Step 3 in Algorithm~\ref{alg:BADMMN} and (\ref{KKTQ}), we have the following condition:
\begin{equation}\label{conlm}
{\bf 0}\in \nabla\G(\bQ^{t+1})+\partial \delta_{ \Y }(\bQ^{t+1})-\B^*(\Lambda^{t+1}).
\end{equation} 
Because $\bQ^{t}\in \ri{\Y},$ $\B(\bQ^{t})\in \B(\ri{\Y}).$ In addition,  let $X\in \Y$ be such that $X_{ij}:=q_i^{t}{q_j^{t}}^\top.$ From $\bq^{t}\in \ri{\X}$ and (\ref{defAB}), we have that $X\in \ri{\Y}$ and $\A(\bq^{t})+\B(X)=0.$ This implies that $-\A(\bq^{t})\in \B(\ri{\Y}).$ We summarize the results as follows:
\begin{equation}\label{Nov_24_22}
\forall t\in \N^+,\quad \B(\bQ^{t})\in \B(\ri{\Y}),\ -\A(\bq^{t})\in \B(\ri{\Y}).
\end{equation}

%%%%%%%%%%%%%%%%%%%%subsection useful lemmas%%%%%%%%%%%%%%%%%%%%%%%%%%%%%%%%%%%%%%%%%%%%%%%%%%%%%%%%%%%%%%%%%%%%%%%%%%%%%%%%%%%%%%%%%%%%%%%%%%%%%%%%%%%%%%%%%%%%%%

Recall the definition of the linear map $\PY$ in (\ref{defiP}), we have the following lemma.

\begin{lem}\label{afflem}
Suppose ${\bf X}\in \ri{\Y}.$ Then ${\bf 0}\in \partial \delta_{\Y}({\bf X})+{\bf Y}\Leftrightarrow \PY\( {\bf Y} \)={\bf 0}.$ 
\end{lem}
\begin{proof}
Because ${\bf X}\in \ri{\Y},$ we have that $\partial \delta_{\Y}({\bf X})={\rm aff}(\Y-\Y)^\perp.$ This implies ${\bf 0}\in \partial \delta_{\Y}({\bf X})+{\bf Y}\Leftrightarrow \PY\( {\bf Y} \)={\bf 0}.$ 
\end{proof}

The next Lemma will be used to bound the Lagrangian multiplier sequence 
$\{\Lambda^t\}$.

\begin{lem}\label{boundlm}
Suppose $\rho>0$ and $\tau\in (0,1/r),$ then for any $t\geq 2,$ the iterate $(\bq^t,\bQ^t,\Lambda^t)$ from Algorithm~\ref{alg:BADMMN} satisfies the following inequality:
\begin{equation}\label{boundd}
\<\Lambda^t-\Lambda^{t+1},\A(\bq^{t+1})+\B(\bQ^{t+1})\>\leq \frac{\| \nabla \G(\bQ^{t+1})-\nabla\G(\bQ^t) \|^2}{\gamma_{r}\rho},
\end{equation}
where $\gamma_{r}>0$ is defined in (\ref{defigamma}).
\end{lem}

\begin{proof}
From Step 3 in Algorithm~\ref{alg:BADMMN}, we have that 
\begin{align}\label{Nov_24_21.5}
&\<\Lambda^t-\Lambda^{t+1},\A(\bq^{t+1})+\B(\bQ^{t+1})\>\notag \\
&=\rho\< \nabla \phi(\B(\bQ^{t+1}))-\nabla \phi(-\A(\bq^{t+1})),\B(\bQ^{t+1})-(-\A(\bq^{t+1})) \>.
\end{align}
Recall the definition of $\hat{\P}$ in (\ref{defiPh}). Define the function $\hat{\phi}:\ \B({\rm ri}(\Y))\rightarrow \R$ such that
\begin{equation}\label{Nov_24_23}
\forall x\in \B({\rm ri}(\Y)),\ \hat{\phi}(x):=\phi(x)+\delta_{[0,1]}(x).
\end{equation}
Because $\phi (x)+\delta_{[0,1]}(x)=\<x,\log x\>+\delta_{[0,1]}(x)$ is strongly convex with parameter $1,$ $\hat{\phi}$ is also strongly convex with parameter 1. For any $x\in \B(\ri{\Y}),$ because $0<x<1,$ $\hat{\phi}(x)$ is differentiable with gradient $\hat{\P}\(\nabla \phi(x)\).$ 
From the property of a strongly convex function in Proposition 12.6 of \cite{rockafellar2009variational}, we have that for any $x,y\in \B(\ri{\Y}),$
\begin{equation}\label{Nov_24_24}
\<x-y,\hat{\P}(\nabla \phi(x))-\hat{\P}(\nabla \phi(y))\>\leq \| \hat{\P}(\nabla \phi(x))-\hat{\P}(\nabla \phi(y)) \|^2
\end{equation}
From (\ref{Nov_24_22}),we can apply inequality (\ref{Nov_24_24}) with $x=\B(\bQ^{t+1})$ and $y=-\A(\bq^{t+1})$ and get the following inequality
\begin{align}\label{Nov_24_25}
&\<\B(\bQ^{t+1})-(-\A(\bq^{t+1})),\hat{\P}(\nabla \phi(\B(\bQ^{t+1})))-\hat{\P}(\nabla \phi(-\A(\bq^{t+1})))\>\notag \\
&\leq \| \hat{\P}(\nabla \phi(\B(\bQ^{t+1})))-\hat{\P}(\nabla \phi(-\A(\bq^{t+1}))) \|^2.
\end{align}
After applying (\ref{frineq}) with $x=\nabla \phi(\B(\bQ^{t+1}))-\nabla \phi(-\A(\bq^{t+1}))$, we get
\begin{align}\label{Nov_24_28} 
&\| \hat{\P}(\nabla \phi(\B(\bQ^{t+1})))-\hat{\P}(\nabla \phi(-\A(\bq^{t+1}))) \|^2 \notag \\
&\leq \frac{\| \PY\circ  \B^* \(\nabla \phi(\B(\bQ^{t+1}))-\nabla \phi(-\A(\bq^{t+1}))\)\|^2}{\gamma_{r}}=\frac{\| \PY\circ  \B^* \(\Lambda^t-\Lambda^{t+1}\)\|^2}{\gamma_{r}\rho^2}.
\end{align}
Combining (\ref{Nov_24_22}), (\ref{Nov_24_25}) and (\ref{Nov_24_28}), we get
\begin{align}\label{Nov_24_29} 
&\rho\< \nabla \phi(\B(\bQ^{t+1}))-\nabla \phi(-\A(\bq^{t+1})),\B(\bQ^{t+1})-(-\A(\bq^{t+1})) \>\notag \\
&=\rho\< \nabla \phi(\B(\bQ^{t+1}))-\nabla \phi(-\A(\bq^{t+1})),\hat{\P}(\B(\bQ^{t+1})-(-\A(\bq^{t+1}))) \> \notag \\
&=\rho\<\B(\bQ^{t+1})-(-\A(\bq^{t+1})),\hat{\P}(\nabla \phi(\B(\bQ^{t+1})))-\hat{\P}(\nabla \phi(-\A(\bq^{t+1})))\> \notag \\
&\leq \frac{\| \PY\circ  \B^* \(\Lambda^t-\Lambda^{t+1}\)\|^2}{\gamma_{r}\rho}. 
\end{align}
From (\ref{conlm}) and Lemma~\ref{afflem}, we have that 
\begin{equation}\label{Nov_24_30}
\PY(\nabla \G(\bQ^{t+1}))=\PY(\B^*(\Lambda^{t+1})).
\end{equation}
Note that the above equality also holds for all $t\geq 1.$ Substitute (\ref{Nov_24_30}) into (\ref{Nov_24_29}), we get that for any $t\geq 2$
\begin{align}\label{Nov_24_31}
&\rho\< \nabla \phi(\B(\bQ^{t+1}))-\nabla \phi(-\A(\bq^{t+1})),\B(\bQ^{t+1})-(-\A(\bq^{t+1})) \>\notag\\
&\leq \frac{\| \PY(\nabla \G(\bQ^{t+1})-\nabla \G(\bQ^{t})) \|^2}{\gamma_{r}\rho}\leq \frac{\| \nabla \G(\bQ^{t+1})-\nabla \G(\bQ^{t})\|^2}{\gamma_{r}\rho}.
\end{align}
This together with (\ref{Nov_24_21.5}) implies (\ref{boundd}).
\end{proof}

The next lemma will be used to bound various terms in the sequence $
\{(\bq^t,\bQ^t,\Lambda^t)\}$ generated from Algorithm~\ref{alg:BADMMN}.

\begin{lem}\label{lemlbQ}
Suppose $\tau\in (0,1/r)$ and $\xi_{r,\tau},$ $\hat{\xi}_{r,\tau}$ are defined in (\ref{defidelta}), (\ref{defideltah}). If
\begin{equation}\label{ranbeta}
\rho\geq \frac{16\( \xi_{r,\tau}^2+4|\E|(M+\log r) \)}{\tau^2} ,
\end{equation}
then the sequence $(\bq^t,\bQ^t,\Lambda^t)_{t\in \N^+}$ generated by Algorithm~\ref{alg:BADMMN} will satisfy that for any $t'\in \N^+,$ 

\begin{itemize}
\item[(i)] $\|\Ph(\Lambda^{t'})\|\leq \xi_{r,\tau}$\,,
\item[(ii)] $\B(\bQ^{t'+1})\geq \frac{\tau}{2}, 
\quad \bQ^{t'+1}\geq \frac{\tau^2}{4r^2\exp[4M]}$,
\item[(iii)] $\| \A(\bq^{t'+1})+\B(\bQ^{t'+1}) \|\leq \frac{2\sqrt{ \xi_{r,\tau}^2+4|\E|(M+2\log r)}}{\sqrt{\rho}}$,
\item[(iv)] $\| \A(\bq^{t'+2})+\B(\bQ^{t'+1}) \|\leq \frac{4\hat{\xi}_{r,\tau}}{\sqrt{\rho}}.$
\end{itemize}
\end{lem}

\begin{proof}
We rewrite the subproblem in Step 2 of Algorithm~\ref{alg:BADMMN} as follows:
\begin{equation}\label{Nov_28_1}
\bQ^{t+1}:= \arg\min_{\bQ\in \Y}\left\{ \G(\bQ)-\< \Lambda^t,\B(\bQ)+\A(\bq^{t+1})\>+\rho B_\phi\( \B(\bQ),-\A(\bq^{t+1}) \)\right\}.
\end{equation}
From the definitions (\ref{defiM}) and (\ref{defF}), we have that 
\begin{equation}\label{ranG}
\forall \bQ\in \Y,\ -|\E|\cdot (M+2\log r)\leq  \G(\bQ)\leq |\E|\cdot M.
\end{equation}
where we have used the fact that 
$$-|\E|M\leq \<\bC,\bQ\>\leq |\E|M,\ -2|\E|\log r \leq \sum_{ij\in \E}\<Q_{ij},\log Q_{ij}\>\leq 0$$ for any $\bQ\in \Y.$ Define $\widehat{\bQ}^{t+1}\in \Y$ such that for any $ij\in \E,$ $\widehat{Q}^{t+1}_{ij}:=q^{t+1}_i ({q^{t+1}_j})^\top.$ We have that $\A(\bq^{t+1})+\B(\widehat{\bQ}^{t+1})={\bf 0}$ and $B_\phi\( \B(\widehat{\bQ}^{t+1}),-\A(\bq^{t+1}) \)=0.$ From (\ref{Nov_28_1}), we have that 
\begin{equation}\label{Nov_28_2}
\G(\bQ^{t+1})-\< \Lambda^t,\B(\bQ^{t+1})+\A(\bq^{t+1})\>+\rho B_\phi\( \B(\bQ^{t+1}),-\A(\bq^{t+1}) \)\leq \G(\widehat{\bQ}^{t+1}).
\end{equation}
From (\ref{defiPh}), (\ref{defigamma}) and (\ref{Nov_24_22}), we have that
\begin{align}\label{Nov_28_3}
&-\< \Lambda^t,\B(\bQ^{t+1})+\A(\bq^{t+1})\>=-\< \Ph(\Lambda^t),\B(\bQ^{t+1})+\A(\bq^{t+1})\> 
\notag\\
&\geq -\| \Ph(\Lambda^t)\|\cdot \| \B(\bQ^{t+1})+\A(\bq^{t+1}) \|.
\end{align}
Substituting (\ref{Nov_28_3}) into (\ref{Nov_28_2}), and using (\ref{ranG}) and the property $B_{\phi}(x,y)\geq \frac{1}{2}\|x-y\|^2$ for $x,y$ in the relative interior of the unit simplex \cite[Prop. 5.1]{beck2003mirror}, we get
\begin{equation}\label{Nov_29_1}
-\| \Ph(\Lambda^t) \|\cdot \| \B(\bQ^{t+1})+\A(\bq^{t+1}) \|+\frac{\rho}{2}\| \A(\bq^{t+1})+\B(\bQ^{t+1}) \|^2\leq 2|\E|\cdot \( M+\log r \).
\end{equation}
Now, we use mathematical induction to prove ($i$), ($ii$), ($iii$) and ($iv$). When $t'=1,$ $(i)$ holds because $\Lambda^1={\bf 0}.$ Suppose $(i)$ holds for $t'=t\in \N^+.$ We will proceed with the following 4 steps.

\medskip

\noindent{\bf Step 1.} We prove that $(iii)$ holds for $t'=t$.

\medskip

Substituting ($i$) into (\ref{Nov_29_1}), we get
\begin{equation}\label{Nov_29_2}
-\xi_{r,\tau}\cdot \| \B(\bQ^{t+1})+\A(\bq^{t+1}) \|+\frac{\rho}{2}\| \A(\bq^{t+1})+\B(\bQ^{t+1}) \|^2\leq 2|\E|\cdot \( M+\log r \),
\end{equation}
which implies that 
\begin{align}\label{Nov_29_3}
&\| \A(\bq^{t+1})+\B(\bQ^{t+1}) \|\leq \frac{ \xi_{r,\tau} +\sqrt{\xi_{r,\tau}^2+4\rho |\E|\cdot (M+\log r) }}{\rho} \notag \\
&\leq \frac{2\sqrt{ \xi_{r,\tau}^2+4|\E|(M+\log r)}}{\sqrt{\rho}}\leq \frac{\tau}{2},
\end{align}
where the last inequality comes from (\ref{ranbeta}). (\ref{Nov_29_3}) implies that ($iii$) holds for $t+1$.

\medskip

\noindent{\bf Step 2.} We prove that $(ii)$ holds for $t'=t.$

\medskip

From (\ref{Nov_29_3}), we have the following condition
\begin{equation}\label{bdmarg}
\forall ij\in \E,\ \|Q^{t+1}_{ij}{\bf 1}_r-q_i^{t+1}\|\leq \frac{\tau}{2},
\quad {\|Q^{t+1}_{ij}}^\top {\bf 1}_r-q_j^{t+1}\|\leq \frac{\tau}{2}.
\end{equation}
Because $q^{t+1}_k\geq \tau {\bf 1}_r$ for all $k\in[n]$,  from (\ref{bdmarg}), we have that 
\begin{equation}\label{bdmargQ}
\forall ij\in \E,\ Q^{t+1}_{ij}{\bf 1}_r \geq \frac{\tau}{2},
\quad {Q^{t+1}_{ij}}^\top {\bf 1}_r\geq \frac{\tau}{2}.
\end{equation}
Because $\bQ^{t+1}$ is the minimizer of (\ref{Nov_28_1}), we have that for any $ij\in \E,$ $Q_{ij}^{t+1}$ is the minimizer of the following entropy OT problem:
\begin{equation*}\label{inexenOT}
\min\left\{ \<C_{ij},X\>+\<\log X,X\>:\ X{\bf 1}_r = Q^{t+1}_{ij}{\bf 1}_r,\ X^\top{\bf 1}_r = {Q^{t+1}_{ij}}^\top{\bf 1}_r \right\}.
\end{equation*}
From Lemma~\ref{lemenOT} and (\ref{bdmargQ}), we have that 
\begin{equation}\label{Nov_29_4}
\forall ij\in \E,\ Q_{ij}^{t+1}\geq \frac{\tau^2}{4r^2\exp[4\|C_{ij}\|_\infty]}.
\end{equation} 
which together with (\ref{bdmargQ}) implies that ($ii$) holds for $t'=t$. 

\medskip

\noindent{\bf Step 3.} We prove that $(i)$ holds for $t'=t+1.$  

\medskip

From (\ref{defiPh}), (\ref{defF}), (\ref{defigamma}), (\ref{frineq}) and (\ref{Nov_24_30}), we have that 
\begin{align}\label{seqineq}
&\| \Ph(\Lambda^{t+1}) \|\leq \frac{\| \PY\circ \B^*(\Lambda^{t+1}) \|}{\sqrt{\gamma_{r}}}= \frac{\| \PY\( \nabla \G(\bQ^{t+1}) \) \|}{\sqrt{\gamma_{r}}}\leq \frac{\|{\bf C}\|+\| \log (\bQ^{t+1}) \|}{\sqrt{\gamma_{r}}}\notag \\
&\leq \frac{\sqrt{|\E|}r\(M+\( -2\log \tau+2\log 2r+4M \) \)}{\sqrt{\gamma_{r}}}\notag \\
&=\frac{\sqrt{|\E|}r\( -2\log \tau+2\log 2r+5M \)}{\sqrt{\gamma_{r}}}=\xi_{r,\tau},
\end{align}
where the third inequality comes from (\ref{Nov_29_4}). Thus, (i) holds for $t'=t+1.$ 

\medskip

\noindent{\bf Step 4.} We prove that $(iv)$ holds for $t'=t.$

\medskip

From the subproblem in Step 1 of Algorithm~\ref{alg:BADMMN}, we get 
\begin{align}\label{Dec_2_1}
&\F(\bq^{t+2})-\<\Lambda^{t+1},\A(\bq^{t+2})+\B(\bQ^{t+1})\>+\rho B_\phi \( -\A(\bq^{t+2}),\B(\bQ^{t+1}) \)+B_\psi(\bq^{t+2},\bq^{t+1})\notag \\
&\leq \F(\bq^{t+1})-\< \Lambda^{t+1},\A(\bq^{t+1})+\B(\bQ^{t+1}) \>+\rho B_\phi \( -\A(\bq^{t+1}),\B(\bQ^{t+1}) \).
\end{align}
From Step 2, we have that 
\begin{equation*}\label{Dec_2_2}
\bq^{t+1}\geq \tau ,\quad 
\bq^{t+2}\geq \tau ,\quad 
\B(\bQ^{t+1})\geq \frac{\tau}{2} .
\end{equation*}
Combining this and the property of KL-divergence, we have that
\begin{align}\label{Dec_2_3}
&B_\phi \( -\A(\bq^{t+1}),\B(\bQ^{t+1}) \)\leq \frac{\| \A(\bq^{t+1})+\B(\bQ^{t+1}) \|^2}{\tau}, \notag\\
&B_\phi \( -\A(\bq^{t+2}),\B(\bQ^{t+1}) \)\geq \frac{\| \A(\bq^{t+2})+\B(\bQ^{t+1}) \|^2}{2}. 
\end{align}
Substituting (\ref{Dec_2_3}) into (\ref{Dec_2_1}) and using (\ref{defiPh}) and (\ref{Nov_24_22}), we have that
\begin{align}\label{Dec_2_4}
&\F(\bq^{t+2})-\|\Ph(\Lambda^{t+1})\|\cdot \|\A(\bq^{t+2})+\B(\bQ^{t+1})\|+\frac{\rho\| \A(\bq^{t+2})+\B(\bQ^{t+1}) \|^2}{2} \notag\\
&\leq \F(\bq^{t+1})+\|\Ph(\Lambda^{t+1})\|\cdot \|\A(\bq^{t+1})+\B(\bQ^{t+1})\|+\frac{\rho\| \A(\bq^{t+1})+\B(\bQ^{t+1}) \|^2}{\tau}.
\end{align}
Using $(i)$ for $t'=t+1,$ $(iii)$ for $t'=t$ and (\ref{defF}) in (\ref{Dec_2_4}), we get
\begin{align}\label{Dec_2_5}
&-\xi_{r,\tau}\|\A(\bq^{t+2})+\B(\bQ^{t+1})\|+\frac{\rho\| \A(\bq^{t+2})+\B(\bQ^{t+1}) \|^2}{2} \notag\\
&\leq 2nM+2(2|\E|-n)\log r+\xi_{r,\tau}\|\A(\bq^{t+1})+\B(\bQ^{t+1})\|+\frac{\rho\| \A(\bq^{t+1})+\B(\bQ^{t+1}) \|^2}{\tau} \notag\\
&\leq 2nM+2(2|\E|-n)\log r+2\xi_{r,\tau}\sqrt{\xi_{r,\tau}^2+4|\E|(M+\log r)}
 \notag\\
&\quad +\frac{4\(\xi_{r,\tau}^2+4|\E|(M+\log r)\) }{\tau}=\hat{\xi}_{r,\tau}.
\end{align}
From (\ref{Dec_2_5}), we have that
\begin{equation}\label{Dec_2_6}
\|\A(\bq^{t+2})+\B(\bQ^{t+1})\|\leq \frac{\xi_{r,\tau}+\sqrt{\xi_{r,\tau}^2+2\rho \hat{\xi}_{r,\tau}}}{\rho}\leq \frac{4\hat{\xi}_{r,\tau}}{\sqrt{\rho}}.
\end{equation}
This implies that $(iv)$ holds for $t'=t$.

Because $(i)$ holds for $t'=t+1,$ by continuing the mathematical induction, we have that $(i),(ii),(iii),(iv)$ hold for all $t'\in \N^+.$ 
\end{proof}

The next lemma will be used to handle the non-symmetric issue of the KL-divergence.

\begin{lem}\label{nonsymlem}
Suppose $x,y,z\in \H_3$ such that $x,y,z\geq \epsilon$ for some $\epsilon>0.$ Then we have the following inequality
\begin{align}\label{nonsymineq}
&|\< \nabla \phi(x)-\nabla \phi(y),y-z \>-\< x-y,\nabla \phi(y)-\nabla \phi(z) \>| \notag \\
&\leq \frac{\|x-y\|^2 \|y-z\|_\infty}{\epsilon^2}+\frac{\|y-z\|^2 \|x-y\|_\infty}{\epsilon^2}.
\end{align}
\end{lem}

\begin{proof}
For any index $i$ of $x,$ we have that 
\begin{equation}\label{Dec_2_7}
[\nabla \phi(x)](i)-[\nabla \phi(y)](i)=\log x(i)-\log y(i)=\int_{y(i)}^{x(i)}\frac{{\rm d}t}{t}.
\end{equation}
From (\ref{Dec_2_7}), we have that
\begin{equation}\label{Dec_2_8}
\left|[\nabla \phi(x)](i)-[\nabla \phi(y)](i)-\frac{x(i)-y(i)}{y(i)}\right|\leq \left|\int_{y(i)}^{x(i)}\left|\frac{1}{t}-\frac{1}{y(i)}\right|{\rm d}t\right|\leq \frac{(x(i)-y(i))^2}{\epsilon^2}.
\end{equation}
From (\ref{Dec_2_8}), we have that
\begin{equation}\label{Dec_2_9}
\left|\< \nabla \phi(x)-\nabla \phi(y),y-z \>-\<(x-y)./y,y-z\>\right|\leq \frac{\|x-y\|^2 \|y-z\|_\infty}{\epsilon^2},
\end{equation}
where ``$./$" is element-wise division. Similarly, we have that
\begin{equation}\label{Dec_2_10}
\left| \< x-y,\nabla \phi(y)-\nabla \phi(z)\>-\<x-y,(y-z)./y\> \right| \leq \frac{\|y-z\|^2 \|x-y\|_\infty}{\epsilon^2}.
\end{equation}
Note that $\<(x-y)./y,y-z\>=\<x-y,(y-z)./y\>.$ After applying the triangle inequality to (\ref{Dec_2_9}) and (\ref{Dec_2_10}), we get (\ref{nonsymineq}).
\end{proof}

Now we state the following two descent lemmas.

\begin{lem}\label{ostlemq}
Suppose $\rho>0$ and $\tau\in (0,1/r),$ then the sequence $\{(\bq^t,\bQ^t,\Lambda^t)\}$ generated from Algorithm~\ref{alg:BADMMN} satisfies the following inequality for any $t\geq 1,$
\begin{align}\label{Nov_24_16}
&\F(\bq^{t+1})+B_\psi(\bq^{t+1},\bq^{t})+\< \Lambda^t,\A(\bq^t)-\A(\bq^{t+1})\>\\ \notag
&-\rho B_\phi(-\A(\bq^t),\B(\bQ^{t}))+\rho B_\phi(-\A(\bq^{t+1}),\B(\bQ^{t}))+\rho B_\phi(-\A(\bq^t),-\A(\bq^{t+1}))\leq \F(\bq^{t}). 
\end{align}
\end{lem}

\begin{proof}
From (\ref{KKTq}), we have that
\begin{align}\label{Nov_24_13}
&\A^*(\Lambda^t)+\nabla \psi(\bq^{t})-\rho\A^*\( -\nabla \phi(-\A(\bq^{t+1}))+\nabla \phi(\B(\bQ^{t}))  \)\notag\\
&\in \nabla \F\(\bq^{t+1}\)+\partial\delta_{\X_\tau}(\bq^{t+1})+\nabla \psi(\bq^{t+1}).
\end{align}
From (\ref{Nov_24_13}) and the convexity of $\F(\bq)+\delta_{\X_\tau}(\bq)+\psi(\bq)$, we have that
\begin{align}\label{Nov_24_14}
&\F(\bq^{t+1})+\delta_{\X_\tau}(\bq^{t+1})+\psi(\bq^{t+1})\notag \\
&+\< \A^*(\Lambda^t)+\nabla \psi(\bq^{t})-\rho\A^*\( -\nabla \phi(-\A(\bq^{t+1}))+\nabla \phi(\B(\bQ^{t}))  \),\bq^t-\bq^{t+1} \> \notag\\
&\leq \F(\bq^{t})+\delta_{\X_\tau}(\bq^{t})+\psi(\bq^t).
\end{align}
Because $\bq^{t},\bq^{t+1}\in \X_\tau,$ we have that $\delta_{\X_\tau}(\bq^{t})=\delta_{\X_\tau}(\bq^{t+1})=0.$ Rearranging (\ref{Nov_24_14}), we get
\begin{align}\label{Nov_24_15}
&\F(\bq^{t+1})+\psi(\bq^{t+1})-\psi(\bq^{t})-\<\nabla \psi(\bq^t),\bq^{t+1}-\bq^t\>+\< \Lambda^t,(-\A(\bq^{t+1}))-(-\A(\bq^t))\> \notag \\
&-\rho\< \nabla \phi(-\A(\bq^{t+1}))-\nabla \phi(\B(\bQ^{t})) ,(-\A(\bq^t))-(-\A(\bq^{t+1})) \> \leq \F(\bq^{t}). 
\end{align}
After applying the three-point property (\ref{3p}) and the definition of Bregman divergence (\ref{Bdis}) to (\ref{Nov_24_15}), we get (\ref{Nov_24_16}).
\end{proof}

\begin{lem}\label{ostlemQ}
Suppose $\rho>0$ and $\tau\in (0,1/r),$ then the sequence 
$\{(\bq^t,\bQ^t,\Lambda^t)\}$ generated from Algorithm~\ref{alg:BADMMN} satisfies the following inequality for any $t\geq 2$:
\begin{align}\label{ostineqQ}
&\G(\bQ^{t+1})+\<  \Lambda^t,\B(\bQ^t)-\B(\bQ^{t+1})\>-\rho B_\phi(-\A(\bq^{t+1}),\B(\bQ^t))+\rho B_\phi(-\A(\bq^{t+1}),\B(\bQ^{t+1}))\notag \\
&+\rho B_\phi( \B(\bQ^{t+1}),\B(\bQ^t))\notag \\
&\leq\G(\bQ^t)-B_\G(\bQ^t,\bQ^{t+1})+\frac{\rho\| \A(\bq^{t+1})+ \B(\bQ^{t+1}) \|^2\| \B(\bQ^{t+1})-\B(\bQ^t) \|_\infty}{\min\{\(\A(\bq^{t+1}), \B(\bQ^{t+1}), \B(\bQ^t)\)\}^2}\notag \\
&\quad +\frac{\rho\| \B(\bQ^{t+1})-\B(\bQ^t) \|^2 \| \A(\bq^{t+1})+ \B(\bQ^{t+1}) \|_\infty}{\min\{\(\A(\bq^{t+1}), \B(\bQ^{t+1}), \B(\bQ^t)\)\}^2}.
\end{align}
\end{lem}
\begin{proof}
From (\ref{KKTQ}) and Lemma~\ref{afflem}, we have that 
\begin{equation}\label{Nov_24_17}
\PY\( \B^*(\Lambda^t)-\rho\B^*\( \nabla \phi(\B(\bQ^{t+1}))-\nabla \phi(-\A(\bq^{t+1})) \) \)=\PY\(\nabla \G(\bQ^{t+1})\).
\end{equation}
From (\ref{defF}), (\ref{Bdis}) and $\bQ^t-\bQ^{t+1}\in {\rm aff}(\Y-\Y)$, we have that 
\begin{equation}\label{Nov_24_18}
\G(\bQ^{t+1})+\< \PY\(\nabla \G(\bQ^{t+1})\),\bQ^t-\bQ^{t+1} \>=\G(\bQ^t)-B_\G(\bQ^t,\bQ^{t+1}),
\end{equation}
where $B_\G(\bQ^t,\bQ^{t+1})$ is the Bregman divergence with kernel function $\G.$ Substituting (\ref{Nov_24_17}) into (\ref{Nov_24_18}), we get
\begin{align}\label{Nov_24_19}
&\G(\bQ^{t+1})+\<  \B^*(\Lambda^t)-\rho\B^*\( \nabla \phi(\B(\bQ^{t+1}))-\nabla \phi(-\A(\bq^{t+1})) \) ,\bQ^t-\bQ^{t+1} \>\notag \\
&=\G(\bQ^t)-B_\G(\bQ^t,\bQ^{t+1}).
\end{align}
After some simplifications, we get
\begin{align}\label{Nov_24_20}
&\G(\bQ^{t+1})+\<  \Lambda^t,\B(\bQ^t)-\B(\bQ^{t+1})\>-\rho\<\nabla \phi(-\A(\bq^{t+1}))- \nabla \phi(\B(\bQ^{t+1}))  ,\B(\bQ^{t+1})-\B(\bQ^t) \>\notag \\
&=\G(\bQ^t)-B_\G(\bQ^t,\bQ^{t+1}).
\end{align}
From Lemma~\ref{nonsymlem}, we have that
\begin{align}\label{Nov_24_20.1}
&\big|\<\nabla \phi(-\A(\bq^{t+1}))- \nabla \phi(\B(\bQ^{t+1}))  ,\B(\bQ^{t+1})-\B(\bQ^t) \> \\ \notag
&-\<(-\A(\bq^{t+1}))- (\B(\bQ^{t+1}))  ,\nabla \phi\(\B(\bQ^{t+1})\)-\nabla \phi\(\B(\bQ^t)\) \> \big| 
\\[3pt] \notag
&\leq \frac{ \| \A(\bq^{t+1})+ \B(\bQ^{t+1}) \|^2\| \B(\bQ^{t+1})-\B(\bQ^t) \|_\infty+\| \B(\bQ^{t+1})-\B(\bQ^t) \|^2 \| \A(\bq^{t+1})+ \B(\bQ^{t+1}) \|_\infty }{\min\{\(\A(\bq^{t+1}), \B(\bQ^{t+1}), \B(\bQ^t)\)\}^2}.
\end{align}
Applying (\ref{Nov_24_20.1}) in (\ref{Nov_24_20}), we get
\begin{align}\label{Nov_24_20.2}
&\G(\bQ^{t+1})+\<  \Lambda^t,\B(\bQ^t)-\B(\bQ^{t+1})\>-\rho\<(-\A(\bq^{t+1}))- (\B(\bQ^{t+1}))  ,\nabla \phi(\B(\bQ^{t+1}))-\nabla \phi(\B(\bQ^t)) \>\notag \\
&\leq \G(\bQ^t)-B_\G(\bQ^t,\bQ^{t+1})+\frac{\rho\| \A(\bq^{t+1})+ \B(\bQ^{t+1}) \|^2\| \B(\bQ^{t+1})-\B(\bQ^t) \|_\infty}{\min\{\(\A(\bq^{t+1}), \B(\bQ^{t+1}), \B(\bQ^t)\)\}^2}\notag \\
&\quad +\frac{\rho\| \B(\bQ^{t+1})-\B(\bQ^t) \|^2 \| \A(\bq^{t+1})+ \B(\bQ^{t+1}) \|_\infty}{\min\{\(\A(\bq^{t+1}), \B(\bQ^{t+1}), \B(\bQ^t)\)\}^2}.
\end{align}
Applying the three-point property (\ref{3p}) in (\ref{Nov_24_20.2}), we get (\ref{ostineqQ}).
\end{proof}

With all the preparations, we are now ready to prove Theorem~\ref{convthm} in the next subsection.

\subsection{Proof of Theorem~\ref{convthm}}

\begin{proof}

We will prove Theorem~\ref{convthm} in the following three steps.

\medskip

\noindent{\bf Step 1.} Deriving one-step inequality

\medskip

From Lemma~\ref{lemlbQ}, we have that for any $t\in \N^+,$ 
\begin{eqnarray}\label{Dec_2_11}
\| \A(\bq^{t+1})+\B(\bQ^{t+1}) \|
&\leq& \frac{2\sqrt{ \xi_{r,\tau}^2+4|\E|(M+\log r)}}
{\sqrt{\rho}} \leq  \frac{\tilde{\xi}_{r,\tau}}{\sqrt{\rho}}
\\
\label{Dec_2_12}
\| \A(\bq^{t+2})+\B(\bQ^{t+1}) \|
&\leq& \frac{4\hat{\xi}_{r,\tau}}{\sqrt{\rho}},
\end{eqnarray}
where $\xi_{r,\tau}$ and $\hat{\xi}_{r,\tau}$ are defined in (\ref{defidelta}) and (\ref{defideltah}). Applying triangle inequality to (\ref{Dec_2_11}) and (\ref{Dec_2_12}), we get
\begin{equation}\label{Dec_2_13}
\| \B(\bQ^{t+2})-\B(\bQ^{t+1}) \|\leq \frac{4\hat{\xi}_{r,\tau}+2\sqrt{ \xi_{r,\tau}^2+4|\E|(M+\log r)}}{\sqrt{\rho}}=\frac{\tilde{\xi}_{r,\tau}}{\sqrt{\rho}},
\end{equation}
where $\tilde{\xi}_{r,\tau}$ is defined in (\ref{defideltat}). Adding (\ref{Nov_24_16}) in Lemma~\ref{ostlemq} and (\ref{ostineqQ}) in Lemma~\ref{ostlemQ} and using $(iii)$ in Lemma~\ref{lemlbQ} and (\ref{Dec_2_13}), we get for any $t\geq 2,$
\begin{align}\label{ostineq}
&\F(\bq^{t+1})+\G(\bQ^{t+1})-\<\Lambda^t,\A(\bq^{t+1})+\B(\bQ^{t+1})\>+\rho B_\phi\( -\A(\bq^{t+1}), \B(\bQ^{t+1}) \) 
\notag \\
&
+\rho B_\phi\( -\A(\bq^t),-\A(\bq^{t+1}) \)+\rho B_\phi\( \B(\bQ^{t+1}),\B(\bQ^t) \)+B_\G\( \bQ^t,\bQ^{t+1} \) \notag \\
\leq&\;\; \F(\bq^t)+\G(\bQ^t)-\<\Lambda^t,\A(\bq^t)+\B(\bQ^t)\>+\rho B_\phi\( -\A(\bq^t), \B(\bQ^t) \) \notag \\
& +4\sqrt{\rho}\tilde{\xi}_{r,\tau}\| \A(\bq^{t+1})+\B(\bQ^{t+1}) \|^2/\tau^2+4\sqrt{\rho}\tilde{\xi}_{r,\tau}\| \B(\bQ^{t+1})-\B(\bQ^t) \|^2/\tau^2,
\end{align}
where we have ignored the nonnegative term $B_\psi(\bq^{t+1},\bq^t)$
on the left-hand-side of the inequality.
From step 3 in Algorithm~\ref{alg:BADMMN}, we have that
\begin{align}\label{Dec_2_14}
&\<\Lambda^t-\Lambda^{t+1},\A(\bq^{t+1})+\B(\bQ^{t+1})\>
 \notag \\
&
=\rho\< \nabla \phi(\B(\bQ^{t+1}))-\nabla \phi(-\A(\bq^{t+1})),\B(\bQ^{t+1})-(-\A(\bq^{t+1})) \>  \notag \\
&=\rho B_\phi\(\B(\bQ^{t+1}),-\A(\bq^{t+1})\)+\rho B_\phi\( -\A(\bq^{t+1}),\B(\bQ^{t+1}) \).
\end{align}
Combine this and (\ref{boundd}) in Lemma~\ref{boundlm}, we have that for any $t\geq 2,$
\begin{align}\label{Dec_2_15}
&\<\Lambda^t-\Lambda^{t+1},\A(\bq^{t+1})+\B(\bQ^{t+1})\>=2\<\Lambda^t-\Lambda^{t+1},\A(\bq^{t+1})+\B(\bQ^{t+1})\>\notag \\
&-\<\Lambda^t-\Lambda^{t+1},\A(\bq^{t+1})+\B(\bQ^{t+1})\> \notag \\
& \leq \frac{2\| \nabla \G(\bQ^{t+1})-\nabla\G(\bQ^t) \|^2}{\gamma_{r}\rho}-\rho B_\phi\(\B(\bQ^{t+1}),-\A(\bq^{t+1})\)-\rho B_\phi\( -\A(\bq^{t+1}),\B(\bQ^{t+1}) \)\notag \\
&\leq \frac{32r^4\exp(8M)\| \bQ^{t+1}-\bQ^t \|^2}{\gamma_{r}\rho \tau^4} -\rho \|\A(\bq^{t+1})+\B(\bQ^{t+1})\|^2,
\end{align}
where the last inequality comes from $(ii)$ in Lemma~\ref{lemlbQ} and $\B_\phi(x,y)\geq \|x-y\|^2/2.$ From the condition of $\rho$ in (\ref{realranbeta}), we have the following inequalities:
\begin{eqnarray}\label{Dec_2_16}
&&B_\G(\bQ^t,\bQ^{t+1})-\frac{32r^4\exp(8M)\| \bQ^{t+1}-\bQ^t \|^2}{\gamma_{r}\rho \tau^4} \;\geq\; \frac{\| \bQ^{t+1}-\bQ^t \|^2}{4},
\\
\label{Dec_2_17}
&&\rho B_\phi\( \B(\bQ^{t+1}),\B(\bQ^t) \)-\frac{4\sqrt{\rho}\tilde{\xi}_{r,\tau}\| \B(\bQ^{t+1})-\B(\bQ^t) \|^2}{\tau^2} \;\geq\; 0,
\\
&&\rho \|\A(\bq^{t+1})+\B(\bQ^{t+1})\|^2-\frac{4\sqrt{\rho}\tilde{\xi}_{r,\tau}\| \A(\bq^{t+1})+\B(\bQ^{t+1}) \|^2}{\tau^2}
\notag \\
&&
\geq \frac{\rho \|\A(\bq^{t+1})+\B(\bQ^{t+1})\|^2}{2}.
\label{Dec_2_18}
\end{eqnarray}
Adding (\ref{ostineq}), (\ref{Dec_2_15}) and applying (\ref{Dec_2_16}), (\ref{Dec_2_17}) and (\ref{Dec_2_18}), we get
\begin{align}\label{ostineq1}
&\F(\bq^{t+1})+\G(\bQ^{t+1})-\<\Ph(\Lambda^{t+1}),\A(\bq^{t+1})+\B(\bQ^{t+1})\>+\rho B_\phi\( -\A(\bq^{t+1}), \B(\bQ^{t+1}) \) \notag \\
&+\rho B_\phi\( -\A(\bq^t),-\A(\bq^{t+1}) \)+\frac{\rho \|\A(\bq^{t+1})+\B(\bQ^{t+1})\|^2}{2}+\frac{\| \bQ^{t+1}-\bQ^t \|^2}{4} \notag \\
&\leq \F(\bq^t)+\G(\bQ^t)-\<\Ph(\Lambda^t),\A(\bq^t)+\B(\bQ^t)\>+\rho B_\phi\( -\A(\bq^t), \B(\bQ^t) \).
\end{align}

\medskip

\noindent{\bf Step 2.} $(\bq^t,\bQ^t,\Ph(\Lambda^t))$ is bounded and every limit point is a KKT point of (\ref{BVPt}).

\medskip

Define Bregman augmented Lagrangian function $\Phi: \H_1\times \H_2\times \H_3\rightarrow \R$ such that for any $\bq\in \H_1,$ $\bQ\in \H_2$ and $\Lambda\in \H_3$
\begin{align}\label{defiPhi}
\Phi\(\bq,\bQ,\Lambda\):=\F(\bq)+\delta_{\X_\tau}(\bq)+\G(\bQ)+\delta_{\Y}(\bQ)\notag \\
-\<\Lambda,\A(\bq)+\B(\bQ)\>+\rho B_\phi\( -\A(\bq), \B(\bQ) \).
\end{align} 
From the boundedness of $\Ph(\Lambda^t)$ ($(i)$ in Lemma~\ref{lemlbQ}), $\X$ and $\Y,$ we have that $\Phi\(\bq^t,\bQ^t,\Ph(\Lambda^t)\)$ is lower bounded. From (\ref{ostineq1}), we have that for any $t\geq 2,$
\begin{align}\label{ostineq2}
&\Phi\(\bq^{t+1},\bQ^{t+1},\Ph(\Lambda^{t+1})\)+\rho B_\phi\( -\A(\bq^t),-\A(\bq^{t+1}) \)+\frac{\rho \|\A(\bq^{t+1})+\B(\bQ^{t+1})\|^2}{2}\notag \\
&+\frac{\| \bQ^{t+1}-\bQ^t \|^2}{4}\leq \Phi\(\bq^t,\bQ^t,\Ph(\Lambda^t)\),
\end{align}
which, together with the injective property of $\A(\cdot)$ implies that 
\begin{align}\label{limits}
\lim\limits_{t\rightarrow \infty} \| \bq^t-\bq^{t+1} \|=0,\ \lim\limits_{t\rightarrow \infty}\|\A(\bq^{t+1})+\B(\bQ^{t+1})\|=0,\ \lim\limits_{t\rightarrow \infty} \| \bQ^{t+1}-\bQ^t \|=0.
\end{align}
From (\ref{KKTq}), (\ref{conlm}), %%and step 3 in Algorithm~\ref{alg:BADMMN}, 
we have that
\begin{eqnarray}\label{KKTq1}
{\bf 0} &\in & \nabla \F\(\bq^{t+1}\)+\partial\delta_{\X_\tau}(\bq^{t+1})-\A^*(\Lambda^{t+1})+\rho\A^*\( \nabla \phi(\B(\bQ^{t}))-\nabla \phi(\B(\bQ^{t+1}))  \)
\notag \\
&&+\nabla \psi(\bq^{t+1})-\nabla \psi(\bq^t), 
\\
 {\bf 0} & \in & \nabla \G(\bQ^{t+1})+\partial \delta_{ \Y }(\bQ^{t+1})-\B^*(\Lambda^{t+1}).
\label{KKTQ1}
\end{eqnarray}
Recall the definition of affine hull in (\ref{affhull}). Because $\X_\tau\subset \X\subset {\rm aff}(\X),$ we have that 
\begin{equation}\label{Nov_29_12}
\({\rm aff}(\X-\X)\)^\perp=\partial \delta_{{\rm aff}(\X)}(\bq^{t+1})\subset \partial\delta_{\X}(\bq^{t+1})\subset \partial\delta_{\X_\tau}(\bq^{t+1}).
\end{equation}
Therefore, (\ref{KKTq1}) is equivalent to the following equation:
\begin{align}\label{Nov_29_13}
{\bf 0} & \in \nabla \F\(\bq^{t+1}\)+\partial\delta_{\X_\tau}(\bq^{t+1})-\PX\(\A^*(\Lambda^{t+1})\)+\rho\A^*\( \nabla \phi(\B(\bQ^{t}))-\nabla \phi(\B(\bQ^{t+1}))  \)\notag \\
&\quad +\nabla \psi(\bq^{t+1})-\nabla \psi(\bq^t).
\end{align}
From (\ref{defiPh}) and (\ref{Nov_29_14}), we have that $\PX\(\A^*(\Lambda^{t+1})\)=\PX\(\A^*\( \Ph(\Lambda^{t+1})\)\).$ Substituting this into (\ref{Nov_29_13}) and using (\ref{Nov_29_12}) again, we get
\begin{align}\label{Nov_29_15}
{\bf 0}& \in \nabla \F\(\bq^{t+1}\)+\partial\delta_{\X_\tau}(\bq^{t+1})-\A^*(\Ph(\Lambda^{t+1}))+\rho\A^*\( \nabla \phi(\B(\bQ^{t}))-\nabla \phi(\B(\bQ^{t+1}))  \)\notag \\
&\quad +\nabla \psi(\bq^{t+1})-\nabla \psi(\bq^t).
\end{align}
Similarly, from Lemma~\ref{afflem}, we get
\begin{equation}\label{Nov_29_16}
{\bf 0}\in \nabla \G(\bQ^{t+1})+\partial \delta_{ \Y }(\bQ^{t+1})-\B^*(\Ph(\Lambda^{t+1})).
\end{equation}
From (\ref{limits}), (\ref{Nov_29_15}), (\ref{Nov_29_16}) and the boundedness of $\Ph(\Lambda^t),$ it is easy to see that the sequence $\(\bq^t,\bQ^t,\Ph(\Lambda^t)\)$ is bounded with every limit point $(\bq^*,\bQ^*,\Lambda^*)$ satisfying the following KKT conditions:
\begin{align}\label{KKTBVPt}
{\bf 0}&\in \nabla \F\(\bq^*\)+\partial\delta_{\X_\tau}(\bq^*)-\A^*( \Lambda^*),
%\\ \notag
\quad
{\bf 0}\in \nabla \G(\bQ^*)+\partial \delta_{ \Y }(\bQ^*)-\B^*(\Lambda^*),  \notag \\
{\bf 0}&=\A(\bq^*)+\B(\bQ^*),\ \bq^*\in \X_\tau,\ \bQ^*\in \Y.
\end{align}
This completes the proof of Step 2. 

\medskip

\noindent{\bf Step 3.} Preparative work for using the KL-inequality.

\medskip

From Step 3 in Algorithm~\ref{alg:BADMMN} and (ii) in Lemma~\ref{lemlbQ}, we have that $\B(\bQ^{t+1})\geq \frac{\tau}{2}$ and 
\begin{align}\label{Jan_9_1}
&\frac{\rho\| \A(\bq^{t+1})+\B(\bQ^{t+1}) \|^2}{2} \geq \frac{\rho\tau^2}{8} \| \nabla \phi(\B(\bQ^{t+1}))- \nabla \phi(-\A(\bq^{t+1}))\|^2 \notag \\
&\geq \frac{\rho\tau^2}{8} \| \Ph\(\nabla \phi(\B(\bQ^{t+1}))- \nabla \phi(-\A(\bq^{t+1}))\)\|^2=\frac{\tau^2}{8\rho} \| \Ph(\Lambda^{t+1})-\Ph(\Lambda^{t}) \|^2,
\end{align}
where the first inequality comes from the $\frac{2}{\tau}-$Lipschitz continuity of $\log (\cdot)$ in $[\frac{\tau}{2},\infty).$ Substituting (\ref{Jan_9_1}) into (\ref{ostineq2}), we get for any $t\geq 2,$
\begin{align}\label{ostineq3}
&\Phi\(\bq^{t+1},\bQ^{t+1},\Ph(\Lambda^{t+1})\)+\rho B_\phi\( -\A(\bq^t),-\A(\bq^{t+1}) \)+\frac{\tau^2}{8\rho} \| \Ph(\Lambda^{t+1})-\Ph(\Lambda^{t}) \|^2\notag \\
&+\frac{\| \bQ^{t+1}-\bQ^t \|^2}{4}\leq \Phi\(\bq^t,\bQ^t,\Ph(\Lambda^t)\).
\end{align}
From (\ref{Nov_29_15}), we have that 
\begin{align}\label{Jan_9_2}
&\nabla \psi(\bq^t)-\nabla \psi(\bq^{t+1})-\rho \A^*\( \nabla \phi(\B(\bQ^t))-\nabla \phi(\B(\bQ^{t+1})) \)\notag \\
&+\rho \A^*\( -\nabla \phi(-\A(\bq^{t+1}))+\nabla \phi(\B(\bQ^{t+1})) \)\notag \\
&\in \nabla \F(\bq^{t+1})+\partial \delta_{\X_\tau}(\bq^{t+1})-\A^*(\Ph(\Lambda^{t+1}))+\rho \A^*\( -\nabla \phi(-\A(\bq^{t+1}))+\nabla \phi(\B(\bQ^{t+1})) \) \notag \\
&=\partial_{\bq} \Phi\(\bq^{t+1},\bQ^{t+1},\Ph(\Lambda^{t+1})\).
\end{align}
From Step 3 in Algorithm~\ref{alg:BADMMN}, we have that
\begin{align}\label{Jan_9_3}
&\nabla \psi(\bq^t)-\nabla \psi(\bq^{t+1})-\rho \A^*\( \nabla \phi(\B(\bQ^t))-\nabla \phi(\B(\bQ^{t+1})) \)+ \A^*\( \Lambda^t-\Lambda^{t+1} \)\notag \\
& \in \partial_{\bq} \Phi\(\bq^{t+1},\bQ^{t+1},\Ph(\Lambda^{t+1})\).
\end{align}
Similar to the derivation of (\ref{Nov_29_15}), we have that
\begin{align}\label{Jan_9_4}
&\nabla \psi(\bq^t)-\nabla \psi(\bq^{t+1})-\rho \A^*\( \nabla \phi(\B(\bQ^t))-\nabla \phi(\B(\bQ^{t+1})) \)+ \A^*\( \Ph(\Lambda^t)-\Ph(\Lambda^{t+1}) \)\notag \\
&\in\partial_{\bq} \Phi\(\bq^{t+1},\bQ^{t+1},\Ph(\Lambda^{t+1})\).
\end{align}
From (\ref{Nov_29_16}), we have that
\begin{align}\label{Jan_9_5}
&-\rho\cdot\B^*\( (-\A(\bq^{t+1})-\B(\bQ^{t+1}))./\B(\bQ^{t+1}) \)\in \nabla \G(\bQ^{t+1})+\partial \delta_\Y(\bQ^{t+1})-\B^*(\Ph(\Lambda^{t+1})) \notag \\
&-\rho\cdot\B^*\( (-\A(\bq^{t+1})-\B(\bQ^{t+1}))./\B(\bQ^{t+1}) \)\in \partial_{\bQ} \Phi\(\bq^{t+1},\bQ^{t+1},\Ph(\Lambda^{t+1})\),
\end{align}
where ``$./$" means element-wise division. From (\ref{defiPhi}), we have that
\begin{align}\label{Jan_9_6}
-\A(\bq^{t+1})-\B(\bQ^{t+1})=\nabla_\Lambda \Phi\(\bq^{t+1},\bQ^{t+1},\Ph(\Lambda^{t+1})\). 
\end{align}
Reusing the definition of $\hat{\phi}$ in (\ref{Nov_24_23}), because it is strongly convex with parameter 1, we have that for any $x,y\in \B(\ri{\Y})$,
\begin{equation}\label{Jan_10_1}
\|x-y\|\leq \| \hat{\P}(\nabla \phi(x))-\hat{\P}(\nabla \phi(y)) \|.
\end{equation}
Substituting $x=\B(\bQ^{t+1}),$ $y=-\A(\bq^{t+1})$ into (\ref{Jan_10_1}) and use Step 3 of Algorithm~\ref{alg:BADMMN}, we get
\begin{align}\label{Jan_10_2}
&\|\A(\bq^{t+1})+\B(\bQ^{t+1})\|\leq \| \hat{\P}(\nabla \phi(\B(\bQ^{t+1})))-\hat{\P}(\nabla \phi(-\A(\bq^{t+1}))) \|
\notag \\
&
=\frac{1}{\rho}\| \Ph(\Lambda^t)-\Ph(\Lambda^{t+1}) \|.
\end{align}
From (\ref{ostineq3}), we have that there exists $\beta_1>0$ such that for any $t\geq 2,$
\begin{align}\label{ostineq4}
&\Phi\(\bq^{t+1},\bQ^{t+1},\Ph(\Lambda^{t+1})\)+\beta_1 \| \bq^t-\bq^{t+1} \|^2+\beta_1 \| \Ph(\Lambda^{t+1})-\Ph(\Lambda^{t}) \|^2+\beta_1 \| \bQ^{t+1}-\bQ^t \|^2 \notag \\
&\leq \Phi\(\bq^t,\bQ^t,\Ph(\Lambda^t)\).
\end{align}
From (\ref{Jan_9_4}), (\ref{Jan_9_5}), (\ref{Jan_9_6}) and (\ref{Jan_10_2}), we have that there exists $\beta_2>0$ such that 
\begin{align}\label{bdgrad}
&{\rm dist}\({\bf 0}, \partial \Phi\(\bq^{t+1},\bQ^{t+1},\Ph(\Lambda^{t+1})\) \)\notag \\
&\leq \beta_2 \( \|\bq^t-\bq^{t+1}\|+\|\bQ^t-\bQ^{t+1}\|+\|\Ph(\Lambda^{t+1})-\Ph(\Lambda^{t})\| \).
\end{align}
Define 
\begin{equation}\label{defiphist}
\Phi^*:=\lim\limits_{t\rightarrow \infty} \Phi\(\bq^t,\bQ^t,\Ph(\Lambda^t)\). 
\end{equation}
Define the following set
\begin{equation}\label{defiU}
\U:=\left\{ (\bq,\bQ,\Lambda)\in \X_\tau\times \Y\times \H_3:\bQ\geq \frac{\tau^2}{4r^2\exp[4M]},\|\Lambda\|\leq \xi_{r,\tau},\Phi(\bq,\bQ,\Lambda)=\Phi^* \right\},
\end{equation}
where $\xi_{r,\tau}$ is defined in (\ref{defidelta}). We know that $\U$ is compact and every limit point of $(\bq^t,\bQ^t,\Ph(\Lambda^t))$ is inside $\U.$ Because $\Phi$ is constructed from linear functions, logarithm functions and $\X_\tau,$ $\Y$ are convex polyhedra, we know that $\Phi$ is a tame function \cite{wilkie1996model,van1998tame} that satisfies the KL property mentioned in Subsection~\ref{Subsec:nota} (see \cite[Section 4]{attouch2010proximal}). Because $\U$ is compact, from Lemma 6 in \cite{bolte2014proximal}, $\Phi$ satisfies the uniformed KL property, that is, there exists $\epsilon,\eta>0$ and a continuous concave function $\varphi: [0,\eta)\rightarrow \R_+$ such that for all $\(\bar{\bq},\bar{\bQ},\bar{\Lambda}\)\in \U$ 
with $\Phi(\bar{\bq},\bar{\bQ},\bar{\Lambda}) = \Phi^*$
and all $(\bq,\bQ,\Lambda)$ in the following set:
\begin{align}\label{interset}
&\U_{\epsilon,\eta}:=\Big\{(\bq,\bQ,\Lambda)\in \X\times \Y\times \H_3: \notag \\
&\qquad \qquad {\rm dist}(\U,(\bq,\bQ,\Lambda))<\epsilon,\Phi^*<\Phi(\bq,\bQ,\Lambda)<\Phi^*+\eta \Big\},
\end{align}
one has,
\begin{equation}\label{Jan_10_3}
\varphi'(\Phi(\bq,\bQ,\Lambda)-\Phi^*){\rm dist}({\bf 0},\partial \Phi(\bq,\bQ,\Lambda))\geq 1.
\end{equation}

\medskip

\noindent{\bf Step 4.} The sequence $\(\bq^t,\bQ^t,\Ph(\Lambda^t)\)$ converges to a KKT solution of (\ref{BVPt}).

\medskip

If for some $t_1\in \mathbb{N}^+,$ $\Phi\big(\bq^t,\bQ^t,\Ph(\Lambda^t)\big)=\Phi^*,$ from (\ref{ostineq4}), we have that $\big(\bq^t,\bQ^t,\Ph(\Lambda^t)\big)=\big(\bq^{t_1},\bQ^{t_1},\Ph(\Lambda^{t_1})\big)$ for any $t\geq t_1.$ Therefore, from Step 2, we know that the sequence $\big(\bq^t,\bQ^t,\Ph(\Lambda^t)\big)$ converges to a KKT point of (\ref{BVPt}). 

Now, suppose $\Phi\big(\bq^t,\bQ^t,\Ph(\Lambda^t)\big)>\Phi^*$ for any $t\in \mathbb{N}^+.$ Because every limit point of $\big(\bq^t,\bQ^t,\Ph(\Lambda^t)\big)$ is inside $\U,$ from the continuity of $\Phi,$ we know that there exists $N\geq 3$ such that for any $t\geq N,$ $\big(\bq^t,\bQ^t,\Ph(\Lambda^t)\big)\in \U_{\epsilon,\eta}$. From (\ref{bdgrad}) and (\ref{Jan_10_3}), we have that for any $t\geq N,$
\begin{align}\label{Jan_10_4}
&\frac{1}{\varphi'(\Phi(\bq^{t},\bQ^{t},\Ph(\Lambda^{t}))-\Phi^*)}\leq {\rm dist}({\bf 0},\partial \Phi(\bq^{t},\bQ^{t},\Ph(\Lambda^{t})))\notag \\
&\leq \beta_2 \( \|\bq^{t-1}-\bq^{t}\|+\|\bQ^{t-1}-\bQ^{t}\|+\|\Ph(\Lambda^{t})-\Ph(\Lambda^{t-1})\| \).
\end{align}
From (\ref{ostineq4}), (\ref{Jan_10_4}) and the concavity of $\varphi$, we have that for any $t\geq N,$
\begin{align}\label{Jan_10_5}
&\beta_1 \| \bq^t-\bq^{t+1} \|^2+\beta_1 \| \Ph(\Lambda^{t+1})-\Ph(\Lambda^{t}) \|^2+\beta_1 \| \bQ^{t+1}-\bQ^t \|^2\notag \\
&\leq \(\Phi(\bq^{t},\bQ^{t},\Ph(\Lambda^{t}))-\Phi^*\)-\(\Phi(\bq^{t+1},\bQ^{t+1},\Ph(\Lambda^{t+1}))-\Phi^*\)\notag \\
&\leq \frac{\varphi(\Phi(\bq^{t},\bQ^{t},\Ph(\Lambda^{t}))-\Phi^*)-\varphi(\Phi(\bq^{t+1},\bQ^{t+1},\Ph(\Lambda^{t+1}))-\Phi^*)}{\varphi'(\Phi(\bq^{t},\bQ^{t},\Ph(\Lambda^{t}))-\Phi^*)}\notag \\
&\leq \beta_2 \( \|\bq^{t-1}-\bq^{t}\|+\|\bQ^{t-1}-\bQ^{t}\|+\|\Ph(\Lambda^{t})-\Ph(\Lambda^{t-1})\| \)(\varphi_t-\varphi_{t+1}),
\end{align}
where $\varphi_t:=\varphi(\Phi(\bq^{t},\bQ^{t},\Ph(\Lambda^{t}))-\Phi^*)>0.$ From Cauchy-Schwarz inequality, we get
\begin{align}\label{Jan_10_6}
&\frac{1}{\sqrt{3}} \( \| \bq^t-\bq^{t+1} \|+\| \Ph(\Lambda^{t+1})-\Ph(\Lambda^{t}) \|+\| \bQ^{t+1}-\bQ^t \| \)\notag \\
&\leq \sqrt{\| \bq^t-\bq^{t+1} \|^2+\| \Ph(\Lambda^{t+1})-\Ph(\Lambda^{t}) \|^2+\| \bQ^{t+1}-\bQ^t \|^2}.
\end{align}
Applying (\ref{Jan_10_6}) in (\ref{Jan_10_5}), we get
\begin{align}\label{Jan_10_7}
&\| \bq^t-\bq^{t+1} \|+\| \Ph(\Lambda^{t+1})-\Ph(\Lambda^{t}) \|+\| \bQ^{t+1}-\bQ^t \|\notag \\
&\leq \sqrt{ \|\bq^{t-1}-\bq^{t}\|+\|\bQ^{t-1}-\bQ^{t}\|+\|\Ph(\Lambda^{t})-\Ph(\Lambda^{t-1})\| }\cdot \sqrt{\frac{3\beta_2(\varphi_t-\varphi_{t+1})}{\beta_1}}.
\end{align}
Applying the mean value inequality to the right hand side of (\ref{Jan_10_7}), we get 
\begin{align}\label{Jan_10_8}
&\| \bq^t-\bq^{t+1} \|+\| \Ph(\Lambda^{t+1})-\Ph(\Lambda^{t}) \|+\| \bQ^{t+1}-\bQ^t \|\notag \\
&\leq \frac{1}{2} \(\|\bq^{t-1}-\bq^{t}\|+\|\bQ^{t-1}-\bQ^{t}\|+\|\Ph(\Lambda^{t})-\Ph(\Lambda^{t-1})\|\) + \frac{3\beta_2(\varphi_t-\varphi_{t+1})}{2\beta_1}.
\end{align}
Taking summation of the above inequality from $N$ to $N_1>N,$ we have that
\begin{align}\label{Jan_10_9}
&\sum_{t=N}^{N_1}\frac{1}{2}\(\| \bq^t-\bq^{t+1} \|+\| \Ph(\Lambda^{t+1})-\Ph(\Lambda^{t}) \|+\| \bQ^{t+1}-\bQ^t \|\)\notag \\
&\leq \|\bq^{N-1}-\bq^{N}\|+\|\bQ^{N-1}-\bQ^{N}\|+\|\Ph(\Lambda^{N})-\Ph(\Lambda^{N-1})\| + \frac{3\beta_2(\varphi_N-\varphi_{N_1+1})}{2\beta_1}.
\end{align}
Because the right hand side of (\ref{Jan_10_9}) is bounded, we know that $\big(\bq^t,\bQ^t,\Ph(\Lambda^t)\big)$ is a Cauchy sequence. Thus, from Step 2, it converges to a KKT point of (\ref{BVPt}).

\end{proof}

%%%%%%%%%%%%%%%%%%%Computational technique%%%%%%%%%%%%%%%%%%%%%%%%%%%%%%%%%%%%%%%%%%%%%%%%%%%%%%%%%%%%%%%%%%%%%%%%%%%%%%%%%%%%%%%%%%%%%%%%%%%%%%%%%%%%%%%%%%%

\section{Implementation details}\label{Sec:impl}
In this section, we discuss the implementation details of Algorithm~\ref{alg:BADMMN}. While the framework outlined in Algorithm~\ref{alg:BADMMN} and the conditions in Theorem~\ref{convthm} are primarily designed for theoretical convergence analysis, certain practical adjustments are necessary to enhance computational efficiency. These adjustments may not strictly adhere to the framework or conditions specified in the theorem. Although this creates a gap between theory and practice, we believe that strictly following the theoretical framework at the cost of significant efficiency loss would be overly conservative. Instead, our focus is on balancing theoretical rigor with practical performance.

\subsection{Solving the subproblem in Step 1}

The subproblem in Step 1 of Algorithm~\ref{alg:BADMMN} can be written as follows for any $k\in [n]:$
\begin{align}\label{updateq}
&q^{t+1}_k:=\arg\min\Big\{ \<c_k-(d_k-1)\log q_k^t,q_k\>-\sum_{kj\in \E}\<\lambda_{kj}^t,-q_k\>-\sum_{ik\in \E}\< \mu^t_{ik},-q_k \>\\
&+\rho\sum_{kj\in \E}\<\log q_k-\log (Q_{kj}^t{\bf 1}_r),q_k\> +\rho\sum_{ik\in \E} \< \log q_k-\log ({Q_{ik}^t}^\top {\bf 1}_r),q_k \>:\ {\bf 1}^\top_r q_k=1,\ q_k\geq \tau \Big\}. \notag 
\end{align}
After some simplications, we get the following problem, 
\begin{equation}\label{updateq1}
q^{t+1}_k:=\arg\min\Big\{ \<\hat{c}_k^t,q_k\>+\rho\cdot d_k\< q_k,\log q_k \>:\ {\bf 1}^\top_r q_k=1,\ q_k\geq \tau\Big\},
\end{equation}
where $\hat{c}_k^t$ is defined as
\begin{equation}\label{defickt}
\hat{c}_k^t:=c_k-(d_k-1)\log q_k^t+\sum_{kj\in \E}\(\lambda_{kj}^t-\rho\log (Q_{kj}^t{\bf 1}_r)\)+\sum_{ik\in \E} \(\mu^t_{ik}-\rho\log ({Q_{ik}^t}^\top {\bf 1}_r)\).
\end{equation}
Problem (\ref{updateq1}) is a strongly convex optimization problem with simplex-type constraints, and it is independent for each $k \in [n]$. One can apply the semismooth Newton method \cite{qi2010semismooth} to effectively solve its dual problem iteratively.

In Algorithm~\ref{alg:BADMMN}, the constraint $\bq \geq \tau $ is included primarily for convergence analysis, as it prevents the iterates from approaching zero. From Proposition~\ref{eqprop}, this constraint becomes inactive at stationary points when $\tau$ is sufficiently small. In our implementation, rather than choosing a very small $\tau > 0$, we simplify the problem by ignoring the constraint $\bq \geq \tau$ and directly solving:
\begin{equation}\label{updateq2}
q^{t+1}_k:=\arg\min\Big\{ \<\hat{c}_k^t,q_k\>+\rho\cdot d_k\< q_k,\log q_k \>:\ {\bf 1}^\top_r q_k=1,\ q_k\geq 0\Big\},
\end{equation}
which has the following closed-form solution
\begin{equation}\label{updateq3}
q^{t+1}_k:=\exp\left[-\hat{c}^t_k/(\rho d_k)\right]/\( {\bf 1}_r^\top \exp\left[-\hat{c}^t_k/(\rho d_k)\right] \).
\end{equation}
Our numerical experiments indicate that the iterates $(\bq^k)_{k\in [n]}$ of the Bregman ADMM are consistently bounded away from zero, rendering the constraint $\bq \geq \tau$ unnecessary in practice. This simplification significantly enhances computational efficiency without compromising the algorithm's robustness.

\subsection{Solving the subproblem in Step 2}

The subproblem in Step 2 of Algorithm~\ref{alg:BADMMN} can be formulated as follows:
\begin{align}\label{updateQ}
&Q^{t+1}_{ij} := \arg\min_{Q_{ij}\geq 0} \Big\{ \<C_{ij}, Q_{ij}\> + \<\log Q_{ij}, Q_{ij}\> - \<\lambda_{ij}^t, Q_{ij} {\bf 1}_r\> - \<\mu^t_{ij}, Q_{ij}^\top {\bf 1}_r\>  \\ \notag
&+ \rho \<\log (Q_{ij}{\bf 1}_r) - \log q_i^{t+1}, Q_{ij}{\bf 1}_r\> + \rho \<\log (Q_{ij}^\top {\bf 1}_r) - \log q_j^{t+1}, Q_{ij}^\top {\bf 1}_r\> : \<Q_{ij}, {\bf 1}_{r \times r}\> = 1 \Big\}.
\end{align}
After simplifications, this reduces to:
\begin{align}\label{updateQ1}
&Q^{t+1}_{ij} := \arg\min_{Q_{ij}\geq 0} \Big\{ \<\widehat{C}_{ij}^t, Q_{ij}\> + \<\log Q_{ij}, Q_{ij}\> \notag \\
&+ \rho \<\log (Q_{ij}{\bf 1}_r), Q_{ij}{\bf 1}_r\> + \rho \<\log (Q_{ij}^\top {\bf 1}_r), Q_{ij}^\top {\bf 1}_r\> : \<Q_{ij}, {\bf 1}_{r \times r}\> = 1 \Big\},
\end{align}
where $\widehat{C}_{ij}^t$ is defined as:
\begin{equation}\label{defiCijt}
\widehat{C}_{ij}^t := C_{ij} - \(\lambda_{ij}^t + \rho \log q_i^{t+1}\) {\bf 1}_r^\top - {\bf 1}_r \(\mu^t_{ij} + \rho \log q_j^{t+1}\)^\top.
\end{equation}
Problem (\ref{updateQ1}) is a strongly convex optimization problem with an entropy regularizer and a simple affine constraint. It is essentially an entropy optimal transport (OT) problem, where the marginal constraints are replaced by KL-divergence penalties. Both Newton and gradient descent methods can be applied to solve its dual problem efficiently. Notably, in many Bethe variational problems derived from graphical models, $r$ is small, making the computational cost of the Newton method small. Furthermore, since the subproblem (\ref{updateQ1}) is independent for each $ij \in \E$, all such problems can be solved in parallel, unlike the subproblems in CCCP, which are not separable and require block coordinate minimization.

Although Newton or gradient descent methods ensure convergence, they are iterative, resulting in a double-loop structure for the overall algorithm. To create a single-loop variant of Algorithm~\ref{alg:BADMMN}, we add a proximal term to linearize the penalty function in (\ref{updateQ}). Using the strong subadditivity (SSA) of the entropy function \cite{lieb2002some}, we define the following convex function $\varphi : \Y \to \R$ such that for any $\bQ\in \Y,$
\begin{equation}\label{defivarphi}
\varphi(\bQ) := \sum_{ij \in \E} \Big( 2 \<\log Q_{ij}, Q_{ij}\> - \<\log (Q_{ij}{\bf 1}_r), Q_{ij}{\bf 1}_r\> - \<\log (Q_{ij}^\top{\bf 1}_r), Q_{ij}^\top{\bf 1}_r\> \Big).
\end{equation}
By adding the proximal term $\rho B_\varphi(\bQ, \bQ^t)$ in Step 2 of Algorithm~\ref{alg:BADMMN}, the corresponding subproblem becomes:
\begin{equation}\label{updateQ2}
Q^{t+1}_{ij} := \arg\min_{Q_{ij}\geq 0} \left\{ \<\widetilde{C}_{ij}^t, Q_{ij}\> + (1 + 2\rho) \<\log Q_{ij}, Q_{ij}\> : \<Q_{ij}, {\bf 1}_{r \times r}\> = 1 \right\},
\end{equation}
where $\widetilde{C}_{ij}^t$ is defined as:
\begin{equation}\label{defiCijth}
\widetilde{C}_{ij}^t := \widehat{C}_{ij}^t - \rho \Big( 2\log Q_{ij}^t - \log (Q_{ij}^t {\bf 1}_r) {\bf 1}_r^\top - {\bf 1}_r \log ({Q_{ij}^t}^\top {\bf 1}_r)^\top \Big).
\end{equation}
This modified problem (\ref{updateQ2}) admits the following closed-form solution:
\begin{equation}\label{updateQ3}
Q^{t+1}_{ij} := \exp\left[-\widetilde{C}_{ij}^t / (1 + 2\rho)\right]/ \< \exp\left[-\widetilde{C}_{ij}^t / (1 + 2\rho)\right], {\bf 1}_{r \times r}\> .
\end{equation}

\begin{rem}\label{linearrem}
While we have established the convergence of Algorithm~\ref{alg:BADMMN} with the proximal term $B_\psi(\bq, \bq^t)$ in Step 1, the convergence of Algorithm~\ref{alg:BADMMN} when adding the proximal term $\rho B_\varphi(\bQ, \bQ^t)$ in Step 2 remains unproven. This is because the penalty parameter $\rho$ appears in the proximal term, and the resulting error in (\ref{conlm}) cannot be controlled by simply increasing $\rho$. To our knowledge, the convergence guarantee for non-convex ADMM with such a proximal term,
even in the Euclidean case, remains an open problem. Despite this, our numerical experiments show that the linearized Bregman ADMM with the additional
proximal term $\rho B_\varphi(\bQ,\bQ^t)$ performs significantly better than the original version, which requires an iterative solver to solve the second subproblem.
\end{rem}

\subsection{Termination criterion}\label{Subsec:KKT}

In order to terminate the algorithm, we need an measurement of the stationarity. Suppose $(\bq,\bQ,\Lambda)\in \ri{\X}\times \ri{\Y}\times \H_3$ is a primal-dual solution of (\ref{BVP}), the KKT conditions are as follows:
\begin{equation}\label{KKTprimalqQ}
\forall ij\in \E,\ Q_{ij} {\bf 1}_r=q_i,\ Q_{ij}^\top {\bf 1}_r =q_j,
\end{equation}
\begin{equation}\label{KKTdualQ}
\forall ij\in \E,\ \exists \hat{z}_{ij}\in \R,\, {\rm s.t.}\ C_{ij}+\log Q_{ij}-\lambda_{ij}{\bf 1}_r^\top-{\bf 1}_r\mu_{ij}^\top+\hat{z}_{ij} {\bf 1}_{r\times r}={\bf 0}_{r\times r},
\end{equation}
\begin{equation}\label{KKTdualq}
\forall k\in [n],\ \exists z_k\in \R,\ {\rm s.t.}\ c_k-(d_k-1)\log q_k+\sum_{kj\in \E}\lambda_{kj}+\sum_{ik\in \E} \mu_{ik}+z_k {\bf 1}_r={\bf 0}_r,
\end{equation}
where $z_k$ and $\hat{z}_{ij}$ are the Lagrangian multipliers for the normalization constraints ${\bf 1}_r^\top q_k=1$ and $\<Q_{ij},{\bf 1}_{r\times r}\>=1$, respectively. We ignore the normalization constraints and nonnegative constraints in (\ref{KKTprimalqQ}) because they are strictly satisfied in Algorithm~\ref{alg:BADMMN}.  We use the following KL-divergence to measure the violation of primal feasibility (\ref{KKTprimalqQ})
\begin{equation}\label{KKTresp}
{\rm Resp}(\bq,\bQ):=\sum_{ij\in \E} \(\< \log q_i-\log (Q_{ij}{\bf 1}_r),q_i \>+\< \log q_j-\log (Q_{ij}^\top {\bf 1}_r),q_j \>\).
\end{equation}
When $\lambda_{ij},\mu_{ij}$ are given, (\ref{KKTdualQ}) has the following unique solution as a probability density vector:
\begin{equation}\label{solKKTQ}
\widehat{Q}_{ij}:= \exp[-C_{ij}+\lambda_{ij}{\bf 1}_r^\top+{\bf 1}_r\mu_{ij}^\top]/\<\exp[-C_{ij}+\lambda_{ij}{\bf 1}_r^\top+{\bf 1}_r\mu_{ij}^\top],{\bf 1}_{r\times r}\>.
\end{equation}
We use the following KL-divergence to measure the violation of (\ref{KKTdualQ}):
\begin{equation}\label{KKTresdQ}
{\rm Resd}_{\bQ}:=\sum_{ij\in \E}  \< \log Q_{ij}-\log \widehat{Q}_{ij},Q_{ij} \>.
\end{equation}
The verification of (\ref{KKTdualq}) is slightly more complicated. For any $k\in [n],$ if $d_k>1,$ then similar to (\ref{solKKTQ}), (\ref{KKTdualQ}) has the following unique solution as a probability density vector:
\begin{equation}\label{solKKTq}
\hat{q}_k:= \exp\left[\frac{c_k+\sum_{kj\in \E}\lambda_{kj}+\sum_{ik\in \E} \mu_{ik}}{(d_k-1)}\right]/\<\exp\left[\frac{c_k+\sum_{kj\in \E}\lambda_{kj}+\sum_{ik\in \E} \mu_{ik}}{(d_k-1)}\right],{\bf 1}_{r}\>
\end{equation}
We use the KL-divergence $\< \log q_k-\log \hat{q}_k,q_k \>$ to measure the violation of (\ref{KKTdualq}) for $q_k.$ When $d_k=1,$ because $d_k-1=0,$ we don't have the $\log q_k$ term in (\ref{KKTdualq}). In this case, we use the following residue
\begin{equation}\label{solKKTqresd=1}
\Big\|(I_r-{\bf 1}_{r\times r}/r)\Big(c_k+\sum_{kj\in \E}\lambda_{kj}+\sum_{ik\in \E} \lambda_{ik}\mu_{ik}\Big)\Big\|/(1+\|c_k\|).
\end{equation}
Combining these two cases, we use the following term to measure the violation of stationarity of $\bq$: 
\begin{align}\label{KKTresdq}
&{\rm Resd}_{\bq}:=\sum_{k\in [n]\ d_k>1}\< \log q_k-\log \hat{q}_k,q_k \>\notag \\
&\qquad +\sum_{k\in [n],\ d_k=1} \Big\|(I_r-{\bf 1}_{r\times r}/r)\Big(c_k+\sum_{kj\in \E}\lambda_{kj}+\sum_{ik\in \E} \lambda_{ik}\mu_{ik}\Big)\Big\|/(1+\|c_k\|).
\end{align}
With (\ref{KKTresdQ}) and (\ref{KKTresdq}), we define the dual KKT residual as follows:
\begin{equation}\label{KKTresd}
{\rm Resd}(\bq,\bQ):={\rm Resd}_{\bQ}+{\rm Resd}_{\bq}.
\end{equation}
In our implementation, we use the following KKT residual 
\begin{equation}\label{KKTres}
{\rm Res}(\bq,\bQ):=\max\{{\rm Resp}(\bq,\bQ),{\rm Resd}(\bq,\bQ)\},
\end{equation}
to measure the stationarity. When it is smaller than the required tolerance, we terminate our algorithm.

\subsection{Penalty parameter}\label{subsec:updatebeta}

In our theoretical analysis of Algorithm~\ref{alg:BADMMN}, we fix the penalty parameter $\rho$ to be some sufficiently large number. However, in practice, the efficiency of ADMM is sensitive to the penalty parameter and a commonly used strategy is to adaptively update it to balance the primal and dual KKT residuals  \cite{lin2011linearized,tang2024self}. In our implementation, we update $\rho^t>0$ as follows:
\begin{equation}\label{updatebeta}
\rho^{t+1}:=
\begin{cases}
\max\{\rho^t/1.2,10^{-3}\} & \mbox{if ${\rm Resp}(\bq^{t+1},\bQ^{t+1})<{\rm Resd}(\bq^{t+1},\bQ^{t+1})/5$} 
\\
\max\{1.2\cdot\rho^t,10^{3}\} & 
\mbox{if ${\rm Resp}(\bq^{t+1},\bQ^{t+1})>5\cdot {\rm Resd}(\bq^{t+1},\bQ^{t+1})$}
\\
\rho^t
& \mbox{otherwise.}
\end{cases}
\end{equation}
The updating scheme of the penalty parameter is not unique, we find that the strategy in
(\ref{updatebeta}) works pretty well in numerical experiments. In order to avoid updating the penalty parameter too frequently and making our algorithm more stable, we check the KKT residual and update the penalty parameter once in every ten iterations.

\subsection{Numerical version of Algorithm~\ref{alg:BADMMN}}

With all the computational techniques mentioned in this section, we now propose the numerical version of Algorithm~\ref{alg:BADMMN} as in Algorithm~\ref{alg:BADMMNnu}. 

\begin{algorithm}
	\renewcommand{\algorithmicrequire}{\textbf{Input:}}
	\renewcommand{\algorithmicensure}{\textbf{Output:}}
	\caption{Numerical version of Algorithm~\ref{alg:BADMMN}}
	\label{alg:BADMMNnu}
	\begin{algorithmic}
		\STATE {\bf Initialization}: Choose $\rho^1 > 0$, $\bq^1 \in \ri{\X}, \bQ^1 \in \ri{\Y}, \Lambda^1 = {\bf 0}\in \H_3,$ ${\rm tol}>0.$
		\FOR{$t = 1, 2, \dots, {\rm maxiter}$}
		\STATE 1. $\bq^{t+1} \in \arg\min_{\X} \left\{ \F(\bq) - \< \Lambda^t, \A(\bq) \> + \rho^t B_\phi(-\A(\bq), \B(\bQ^t))+B_\psi(\bq,\bq^t) \right\}$
		\STATE 2. $\bQ^{t+1} \in \arg\min_{\Y} \left\{ \G(\bQ) - \< \Lambda^t, \B(\bQ) \> + \rho^t B_\phi(\B(\bQ), -\A(\bq^{t+1}))+\rho^t B_\varphi(\bQ,\bQ^t) \right\}$
		\STATE 3. $\Lambda^{t+1} := \Lambda^t - \rho^t \( \nabla \phi(\B(\bQ^{t+1})) - \nabla \phi(-\A(\bq^{t+1})) \)$
		\IF{{\rm mod}($t$,10)=1}
			\STATE {4.} Compute ${\rm Resp}(\bq^{t+1},\bQ^{t+1})$, ${\rm Resd}(\bq^{t+1},\bQ^{t+1})$ as in (\ref{KKTresp}) and (\ref{KKTresd})
			\IF{ $\max\{ {\rm Resp}(\bq^{t+1},\bQ^{t+1}),{\rm Resd}(\bq^{t+1},\bQ^{t+1}) \}<{\rm tol}$ }
			\STATE {\bf Break}
			\ENDIF
			\STATE {5.} Update $\rho^{t+1}$ as in (\ref{updatebeta})
		\ELSE
		\STATE {6.} $\rho^{t+1}:=\rho^t$ 
		\ENDIF
		\ENDFOR
	\end{algorithmic}
\end{algorithm}

We summarize the computational techniques used in Algorithm~\ref{alg:BADMMNnu} as follows:

\begin{itemize}
\item It drops the constraint $\bq\geq\tau$ used in Algorithm~\ref{alg:BADMMN} so that the subproblem in Step 1 has a closed form solution (\ref{updateq3}).
\item It adds a proximal term $\rho^tB_\varphi(\bQ,\bQ^t)$ in the subproblem of Step 2 so that it has a closed form solution (\ref{updateQ3}).
\item It uses the KKT residual defined in (\ref{KKTres}) to measure the violation of stationarity.
\item It updates the penalty parameter $\rho^t$ adaptively as in (\ref{updatebeta}).
\end{itemize}

%%%%%%%%%%%%%%%%%%%%%%%%%%%%%
\section{Quantum extension}\label{Sec:QT}

\subsection{Quantum Bethe variational problem}
The quantum Bethe variational problem is as follows:
\begin{align}\label{QBVP}
&\min\Big\{\sum_{ij\in \E}\(\<\hC_{ij},Q_{ij}\>+\< Q_{ij}, \logm (Q_{ij}) \>\)+\sum_{k\in [n]} \(\<\hc_{k},q_{k}\>-(d_k-1)\< q_{k}, \logm (q_{k}) \>\): \notag \\
& \forall ij\in \E,\Trl (Q_{ij})=q_j,\Trr (Q_{ij})=q_i,Q_{ij}\in \Hr^{r^2}_+,\forall k\in [n],q_k\in \Hr^r_+,\Tr(q_k)=1
\Big\}. \tag{QBVP}
\end{align}
In (\ref{QBVP}), $\hC_{ij}\in \Hr^{r^2},$ $\hc_k\in \Hr^r$ are hermitian coefficient matrices. The operator $\logm(\cdot)$ refers to the logarithm of a matrix such that for any matrix $A\in \Hr^p_+$ 
(the set of $p\times p$ hermitian positive semidefinite matrix.) with eigenvalue decomposition $P\dd(d)P^\top,$ $\logm(A)=P\dd(\log(d))P^\top.$ Note that $\logm(A)$ requires $A$ to be positive definite. However, by defining $\<A,\logm (A)\>:=\<d,\log (d)\>,$ $A$ can have zero eigenvalues. In this case, the domain of (\ref{QBVP}) is a compact set. Similarly, we use $\expm(\cdot)$ to denote the matrix exponential such that $\expm(A)=P\dd(\exp(d))P^\top.$ The mappings $\Trl: \Hr^{r^2}\rightarrow \Hr^r$, $\Trr:\Hr^{r^2}\rightarrow \Hr^r$ are trace out operators such that for any $B\in \Hr^{r^2}$ and $i,j\in [r],$
\begin{equation}\label{defitraceoutl}
\Trl(B)(i,j):=\sum_{k=1}^rB(i+(k-1)r,j+(k-1)r),
\end{equation}
\begin{equation}\label{defitraceoutr}
\Trr(B)(i,j):=\sum_{k=1}^rB((i-1)r+k,(j-1)r+k).
\end{equation}
Here for simplicity, we only provide the optimization formulation of (\ref{QBVP}). For more bankground knowledge about quantum mechanics, please refer to \cite{leifer2008quantum,zhao2024quantum}. (\ref{QBVP}) is an extension of (\ref{BVP}) with the following changes

\begin{itemize}
\item[(1)] $q_k\in \R^r_+,$ $Q_{ij}\in \R^{r\times r}_+$ are replaced by $q_k\in \Hr^r_+$ and $Q_{ij}\in \Hr^{r^2}_+$ respectively.
\item[(2)] $\log(\cdot)$ and $\exp(\cdot)$ are replaced by $\logm(\cdot)$ and $\expm(\cdot)$ respectively.
\item[(3)] $Q_{ij}{\bf 1}_r$ and $Q_{ij}^\top {\bf 1}_r$ are replaced with $\Trr(Q_{ij})$ and $\Trl(Q_{ij})$ respectively.
\item[(4)] ${\bf 1}^\top q_k=1$ is replaced by $\Tr(q_k)=1.$ 
\end{itemize}
Note that from the constraint $\Trl(Q_{ij})=q_j,$ we have that $\Tr(Q_{ij})=\Tr(\Trl(Q_{ij}))=\Tr(q_j)=1.$ Thus, the trace of each matrix $Q_{ij}$ is also equals to 1. 

\subsection{Quantum belief propagation}

There are two types of quantum belief propagation (QBP) methods. The approach introduced by Leifer and Poulin in \cite{leifer2008quantum} is a direct extension of the classical belief propagation. However, it imposes stringent requirements on the quantum system and may not converge, even when the graph $G$ is a tree. In this paper, we focus on the QBP method proposed by Zhao et al. in \cite{zhao2024quantum}, which is derived from the KKT conditions of (\ref{QBVP}) and is applicable to a broader class of problems.

For any $ij\in \E,$ let $\lambda_{ij},\mu_{ij} \in \Hr^r$ be the Lagrangian multipliers of the marginal constraint $\Trr (Q_{ij})=q_i,\Trl (Q_{ij})=q_j$, respectively. Define $\Psi_k:=\expm(-\hc_k),$ $\Psi_{ij}:=\expm(-\hC_{ij}).$ Let the messages be $M_{ij}^\lambda:=\expm (\lambda_{ij}),$ $M_{ij}^\mu:=\expm (\mu_{ij}).$ For any $ij\in \E$, QBP updates the messages as follows:
\begin{equation}\label{QBP:updateQ}
Q_{ij}\leftarrow \frac{\expm(-\hC_{ij}+\mu_{ij}\otimes I_r +I_r\otimes \lambda_{ij} )}{\Tr\(\expm(-\hC_{ij}+\mu_{ij}\otimes I_r +I_r\otimes \lambda_{ij} )\)},
\end{equation}
\begin{equation}\label{QBP:updateMij}
M_{ij}^{\mu}\leftarrow \psi_j \odot \Pi_{\alpha j\in \E, \alpha\neq i } \(\Trl\(Q_{\alpha j}\)\odot (M_{\alpha j}^{\mu})^{-1}\)\odot \Pi_{j\rho\in \E } \(\Trr\(Q_{j\rho}\)\odot (M_{j\rho}^{\lambda})^{-1}\),
\end{equation}
\begin{equation}\label{QBP:updateMji}
M_{ij}^{\lambda}\leftarrow \psi_i \odot \Pi_{\alpha i\in \E } \(\Trl\(Q_{\alpha i}\)\odot (M_{\alpha i}^{\mu})^{-1}\)\odot \Pi_{i\rho\in \E,\rho\neq j } \(\Trr\(Q_{i\rho}\)\odot 
(M_{i\rho}^{\lambda})^{-1}\).
\end{equation}
%In equations (\ref{QBP:updateMij}) and (\ref{QBP:updateMji}), 
In the above, $\otimes$ denotes the Kronecker product, and
the product $\odot$ is defined as
\begin{equation}\label{defiQprod}
A \odot B = \expm \left( \logm (A) + \logm (B) \right).
\end{equation}
Here, $\Pi$ represents the sequence of $\odot$-products. After each round of message updates, the marginal density operator $q_k$ for any $k \in [n]$ is recovered as follows:
\begin{equation}\label{QBP:recoverq}
q_k:=\begin{cases}
\Big[\Psi_k^{-1}\odot \Big(\Pi_{kj\in \E} M^\lambda_{kj} \Big) \odot \Big(\Pi_{ik\in \E} M^\mu_{ik}\Big)\Big]^{-1/(d_k-1)}/z_k, & d_k>1\\
 \Trl (Q_{ik}) & \exists i\in [n],\ {\rm s.t.}\ ik\in \E \\
 \Trr (Q_{kj}) &\exists j\in [n],\ {\rm s.t.}\ kj\in \E, 
\end{cases}
\end{equation}
where $z_k>0$ is a normalization constant to make sure $\Tr(q_k)=1.$ Although various algorithms have been proposed to solve the classical (\ref{BVP}), the research 
for the quantum case (\ref{QBVP}) is still in an early stage and we have not seen any optimization algorithms proposed to solve (\ref{QBVP}) in the literature. In the next subsection, we will show that our Bregman ADMM can be extended to solve (\ref{QBVP}).

\subsection{Bregman ADMM for (\ref{QBVP})}

In this subsection, we discuss how to extend Algorithm~\ref{alg:BADMMNnu} to solve (\ref{QBVP}). In order to facilitate our discussion, we need some notation and definitions that are similar to those presented in the beginning of 
subsection \ref{Subsec:bdis}.
Define the following linear spaces
\begin{equation}\label{defihH}
\hH_1:=\Pi_{k\in [n]} \Hr^r,\ \hH_2:=\Pi_{ij\in \E} \Hr^{r^2},\ \hH_3:=\Pi_{ij\in \E} \Hr^r\times \Hr^r.
\end{equation}
Let $\hbC:=(\hC_{ij})_{ij\in \E},$ $\hbc:=(\hc_k)_{k\in [n]}.$ Define the following
convex sets
\begin{eqnarray}
\label{defihX}
\hX &:=&\left\{ \bq\in \hH_1:\ \forall k\in [n],\ q_k\in \Hr^r_+,\ \Tr(q_k)=1 \right\},
\\
\label{defihY}
\hY &:=& \left\{ \bQ\in \hH_2:\ \forall ij\in \E,\ Q_{ij}\in \Hr^{r^2}_+,\ \Tr(Q_{ij})=1 \right\},
\end{eqnarray}
and linear maps $\hA:\hH_1\rightarrow \hH_3,$ $\hB:\hH_2\rightarrow \hH_3$ such that for any $\bq\in \hH_1,$ $\bQ\in \hH_2$
\begin{equation}\label{defihAhB}
\hA(\bq):=(-q_i,-q_j)_{ij\in \E},\quad \hB(\bQ):=(\Trr(Q_{ij}),\Trl(Q_{ij}))_{ij\in \E}.
\end{equation}
Define the functions $\hF:\X\rightarrow \R$ and $\hG:\Y\rightarrow \R$ such that for any $\bq\in \hX,$ $\bQ\in \hY$,
\begin{equation}\label{defihFhG}
\hF(\bq):=\<\hbc,\bq\>-\sum_{k\in [n]}(d_k-1)\<q_k,\logm (q_k)\>,\ \hG(\bQ):=\<\hbC,\bQ\>+\sum_{ij\in \E} \<Q_{ij},\logm (Q_{ij})\>.
\end{equation}
Define the function $\hat{\psi}:\hX\rightarrow \R$ such that for any $\bq\in \hX,$ 
\begin{equation}\label{defihpsi}
\hat{\psi}(\bq):=\sum_{k\in [n]}(d_k-1)\<q_k,\logm (q_k)\>.
\end{equation}
Similar to the strong subadditivity (SSA) of Shannon entropy, von Neumann entropy also has SSA \cite{lieb1973proof}. Thus, the following function $\hat{\varphi}:\hY\rightarrow \R$ is convex
\begin{align}\label{defihvarphi}
&\hat{\varphi}(\bQ):=\sum_{ij \in \E} \Big( 2 \<\logm (Q_{ij}), Q_{ij}\>\notag \\
&\qquad\quad - \<\logm (\Trr(Q_{ij})), \Trr (Q_{ij})\> - \<\logm (\Trl(Q_{ij})), \Trl(Q_{ij})\> \Big).
\end{align}
With the above notation, our Bregman ADMM for solving (\ref{QBVP}) is described in Algorithm~\ref{alg:BADMM_QT}.

\begin{algorithm}
	\renewcommand{\algorithmicrequire}{\textbf{Input:}}
	\renewcommand{\algorithmicensure}{\textbf{Output:}}
	\caption{Quantum Bregman ADMM for (\ref{QBVP}).}
	\label{alg:BADMM_QT}
	\begin{algorithmic}
		\STATE {\bf Initialization}: Choose $\rho^1 > 0$, $\bq^1 \in \ri{\hX}, \bQ^1 \in \ri{\hY}, \Lambda^1 = {\bf 0}\in \hH_3$
		\FOR{$t = 1, 2, \dots$}
		\STATE 1. $\bq^{t+1} \in \arg\min_{\hX} \left\{ \hF(\bq) - \< \Lambda^t, \hA(\bq) \> + \rho^t B_{\hat{\phi}}(-\hA(\bq), \hB(\bQ^t))+B_{\hat{\psi}}(\bq,\bq^t) \right\}$
		\STATE 2. $\bQ^{t+1} \in \arg\min_{\hY} \left\{ \hG(\bQ) - \< \Lambda^t, \hB(\bQ) \> + \rho^t B_{\hat{\phi}}(\hB(\bQ), -\hA(\bq^{t+1}))+\rho^t B_{\hat{\varphi}}(\bQ,\bQ^t) \right\}$
		\STATE 3. $\Lambda^{t+1} := \Lambda^t - \rho^t \( \nabla \hat{\phi}(\B(\bQ^{t+1})) - \nabla \hat{\phi}(-\A(\bq^{t+1})) \)$
		\STATE 4. Checking the KKT residual and updating $\rho_{t+1}$ similarly as in Algorithm~\ref{alg:BADMMNnu}.
		\ENDFOR
	\end{algorithmic}
\end{algorithm}
In Algorithm~\ref{alg:BADMM_QT}, we omit the details of calculating the KKT residual and updating the penalty parameter because these calculations are very similar to that in Algorithm~\ref{alg:BADMMNnu}, which has been discussed in detail in Subsection~\ref{Subsec:KKT} and Subsection~\ref{subsec:updatebeta}. Similar to Algorithm~\ref{alg:BADMMNnu}, the subproblems in Algorithm~\ref{alg:BADMM_QT} have closed form solutions. The subproblem in Step 1 can be written as for any $k\in [n],$
\begin{equation}\label{BADMMsub1_QT}
q^{t+1}_k:=\arg\min\Big\{ \<\tilde{c}_k^t,q_k\>+\rho^t\cdot d_k\< q_k,\logm ( q_k) \>:\ \Tr(q_k)=1,\ q_k\in \Hr^r_+\Big\},
\end{equation}
where $\tilde{c}_k^t$ is defined as
\begin{align*}\label{defickt_QT}
&\tilde{c}_k^t:=\hc_k-(d_k-1)\logm (q_k^t)\notag \\
&\qquad +\sum_{kj\in \E}\(\lambda_{kj}^t-\rho^t\logm (\Trr(Q_{kj}^t)\)+\sum_{ik\in \E} \(\mu^t_{ik}-\rho^t\logm (\Trl(Q_{ik}^t))\).
\end{align*}
Problem (\ref{BADMMsub1_QT}) has the following closed form solution.
\begin{equation}\label{BADMMsub1_QTcloed}
q^{t+1}_k:=\expm \(-\tilde{c}_k^t/(\rho^t d_k)\)/\Tr\( \expm \(-\tilde{c}_k^t/(\rho^t d_k)\) \).
\end{equation}
The subproblem in Step 2 can be solved independently for 
each $ij\in \E$ as follows:
\begin{equation}\label{updateQ2_QT}
Q^{t+1}_{ij} := \arg\min \left\{ \<\widetilde{C}_{ij}^t, Q_{ij}\> + (1 + 2\rho^t) \<\logm (Q_{ij}), Q_{ij}\> : \Tr(Q_{ij})=1,\ Q_{ij}\in \Hr^{r^2}_+ \right\},
\end{equation}
where $\widetilde{C}_{ij}^t$ is defined as:
\begin{align*}\label{defiCijth_QT}
&\widetilde{C}_{ij}^t := \hC_{ij} - \(\lambda_{ij}^t + \rho^t \logm (q_i^{t+1})\) \otimes I_r - I_r\otimes \(\mu^t_{ij} + \rho^t \logm (q_j^{t+1})\)\notag \\
& - \rho^t \Big( 2\logm (Q_{ij}^t) - \logm (\Trr(Q_{ij}^t))\otimes I_r - I_r\otimes \logm ( \Trl\({Q_{ij}^t})\) \Big).
\end{align*}
Problem (\ref{updateQ2_QT}) admits the following closed-form solution:
\begin{equation}\label{updateQ2_QT1}
Q^{t+1}_{ij} := \expm\(-\widetilde{C}_{ij}^t / (1 + 2\rho^t)\)/ \Tr\(\expm\(-\widetilde{C}_{ij}^t / (1 + 2\rho^t)\)\).
\end{equation}

\begin{rem}\label{furtherrem}
Unlike Algorithm~\ref{alg:BADMMN}, which have rigorous convergence guarantee, in this paper we only show the applicability of Algorithm~\ref{alg:BADMM_QT} for solving (\ref{QBVP}), without conducting its convergence analysis. It would be interesting to study the convergence behaviour of Algorithm~\ref{alg:BADMM_QT}. Because the 
positivity of the local minima of (\ref{BVP}) in Theorem~\ref{intthm} is critical in the theoretical analysis, it would also be interesting to study whether the local minima of (\ref{QBVP}) are positive definite. We leave these investigations as future research. 
\end{rem}
In the next section, we will use numerical experiments to test the computational behaviour of Algorithm~\ref{alg:BADMMNnu} and Algorithm~\ref{alg:BADMM_QT} compared with various state-of-the-art algorithms for (\ref{BVP}) and (\ref{QBVP}). 

%%%%%%%%%%%%%%%%%%%%%%%%%%%%%%%%%%%%%%%%%%%%%%%%%%%%%%%%%%%%%%%%
%% Numerical experiments
%%%%%%%%%%%%%%%%%%%%%%%%%%%%%%%%%%%%%%%%%%%%%%%%%%%%%%%%%%%%%%%%%%%%

\section{Numerical experiments}\label{Sec:nume}

In this section, we conduct numerical experiments to evaluate the efficiency and robustness of Algorithm~\ref{alg:BADMMNnu} and Algorithm~\ref{alg:BADMM_QT}, referred to as BADMM and QBADMM for simplicity. For (\ref{BVP}), we compare BADMM with the double-loop method (CCCP) and Belief Propagation (BP). For (\ref{QBVP}), we only compare QBADMM with quantum belief propagation (QBP), because we are unaware of any other methods proposed for solving (\ref{QBVP}). All algorithms are implemented in {\sc Matlab} R2021b and executed on a workstation equipped with an Intel(R) Xeon(R) CPU E5-2680 v3 @ 2.50GHz processor and 128GB of RAM.

For CCCP, as recommended in \cite{yuille2002cccp}, we run the BCM subroutine for 5 iterations to solve its subproblem approximately. BP has two variants: Jacobi Belief Propagation (JBP) and Gauss-Seidel Belief Propagation (GSBP). JBP updates all messages simultaneously in parallel during each iteration, akin to the Jacobi method in numerical linear algebra. In contrast, GSBP updates each message sequentially, incorporating the most recent information from previously updated messages within the same iteration. Similarly, QBP has two variants JQBP and GSQBP.

We use the KKT residual (\ref{KKTres}) to measure stationarity, as the Lagrangian multipliers are available in all these algorithms. Note that for BP and QBP, ${\rm Resd}(\bq,\bQ)=0$, since the primal solutions are recovered directly from the equation of dual feasibility. Unless otherwise specified, the tolerance is set to $10^{-6}$, the maximum number of iterations to 10,000, and the maximum running time to 3600 seconds for all algorithms. For CCCP, the number of iterations reported corresponds to the BCM iterations in the subproblem, as each iteration incurs computational costs comparable to those of JBP, GSBP, and BADMM. For (\ref{BVP}), all algorithms are initialized with $\bq=({\bf 1}_r/r)_{k\in [n]}$, $\bQ=({\bf 1}_{r\times r}/r^2)_{ij\in \E}$, and $\Lambda:={\bf 0}\in \H_3$. For (\ref{QBVP}), all algorithms are initialized with $\bq=(I_r/r)_{k\in [n]}$, $\bQ=(I_{r^2}/r^2)_{ij\in \E}$, and $\Lambda:={\bf 0}\in \hH_3.$ For QBADMM, we set $\rho^1=1$.

We summarize the performance of these algorithms in tables, reporting the primal KKT residual (Resp), dual KKT residual (Resd), objective function value (fval), number of iterations (iter), and running time (time). An asterisk (*) denotes algorithms that reached the maximum number of iterations or the maximum running time without achieving the required accuracy. The algorithm that meets the tolerance in the shortest running time is highlighted in {\bf bold font}.

\subsection{Example of (\ref{BVP}): spin-glass models} 

In the first example, we test spin-glass models, which has been tested in \cite{yuille2002cccp,yedidia2001bethe,welling2013belief}. We generate problem data in the same way as in \cite{yuille2002cccp}. In detail, the graph $G$ is generated as an $n_1\times n_1$ 2D grid graph, where $n_1\in \{50,100\}$ and $n_1\times n_1\times n_1$ 3D lattice graph, where $n_1\in \{20,50\}.$ The data $c_k\in \R^2$ and $C_{ij}\in \R^{2\times 2}$ are generated from Gaussian distribution $N(0,\sigma^2)$ where $\sigma\in \{1,2,5\}.$ The results are shown in Table~\ref{spin2d} and Table~\ref{spin3d}. 

\begin{center}
\begin{footnotesize}
\begin{longtable}{|c|c|ccc|cl|}
\caption{Comparison of BADMM, JBP, GSBP and CCCP for 2D spin-glass models}
\label{spin2d}
\\
\hline
problem & algorithm & Resp & Resd & fval& iter & time \\ \hline
\endhead
 $n$=2500 & {\bf JBP} & 5.38e-07 & 0 & -3.4105207e+03 & 17 & 1.11e-01 \\
 $m$=4900 & GSBP & 3.93e-07 & 0 & -3.4105201e+03 & 11 & 2.10e+00 \\
 $\sigma$=1 & BADMM & 1.65e-07 & 4.68e-07& -3.4105249e+03 & 161 & 6.95e-01 \\
  & CCCP & 1.24e-08 & 9.07e-07 & -3.4105192e+03 & 450 & 9.22e+01 \\
 \hline
 $n$=2500 & {\bf JBP} & 8.08e-07 & 0 & -6.1862453e+03 & 49 & 1.59e-01 \\
 $m$=4900 & GSBP & 9.29e-07 & 0 & -6.1862436e+03 & 26 & 4.60e+00 \\
 $\sigma$=2 & BADMM & 3.25e-07 & 9.12e-07& -6.1862513e+03 & 181 & 7.00e-01 \\
  & CCCP & 3.64e-08 & 9.30e-07 & -6.1862423e+03 & 555 & 1.14e+02 \\
 \hline
 $n$=2500 & JBP* & 3.58e+00* & 0 & -1.5038001e+04 & 10000 & 2.75e+01 \\
 $m$=4900 & GSBP* & 5.26e+00* & 0 & -1.5035182e+04 & 10000 & 1.70e+03 \\
 $\sigma$=5 & {\bf BADMM} & 2.54e-07 & 7.19e-07& -1.5034679e+04 & 261 & 1.08e+00 \\
  & CCCP & 4.31e-08 & 9.88e-07 & -1.5034678e+04 & 730 & 1.46e+02 \\
 \hline
 $n$=10000 & {\bf JBP} & 4.80e-07 & 0 & -1.3851321e+04 & 19 & 2.03e-01 \\
 $m$=19800 & GSBP & 2.73e-07 & 0 & -1.3851320e+04 & 12 & 8.72e+00 \\
 $\sigma$=1 & BADMM & 2.90e-07 & 8.19e-07& -1.3851332e+04 & 171 & 2.02e+00 \\
  & CCCP & 1.46e-08 & 8.89e-07 & -1.3851318e+04 & 500 & 4.30e+02 \\
 \hline
 $n$=10000 & {\bf JBP} & 6.25e-07 & 0 & -2.5208488e+04 & 67 & 6.57e-01 \\
 $m$=19800 & GSBP & 9.63e-07 & 0 & -2.5208485e+04 & 31 & 2.24e+01 \\
 $\sigma$=2 & BADMM & 2.35e-07 & 8.27e-07& -2.5208490e+04 & 411 & 4.76e+00 \\
  & CCCP & 2.69e-08 & 9.94e-07 & -2.5208487e+04 & 1360 & 1.17e+03 \\
 \hline
 $n$=10000 & JBP* & 2.28e+01* & 0 & -6.1346285e+04 & 10000 & 9.70e+01 \\
 $m$=19800 & GSBP* & 4.02e+01* & 0 & -6.1344648e+04 & 5030 & 3.60e+03 \\
 $\sigma$=5 & {\bf BADMM} & 3.44e-07 & 9.37e-07& -6.1351748e+04 & 401 & 4.70e+00 \\
  & CCCP & 7.66e-08 & 9.94e-07 & -6.1351745e+04 & 1135 & 9.70e+02 \\
 \hline
 \end{longtable}
\end{footnotesize}
\end{center}

\begin{center}
\begin{footnotesize}
\begin{longtable}{|c|c|ccc|cl|}
\caption{Comparison of BADMM, JBP, GSBP and CCCP for 3D spin-glass models}
\label{spin3d}
\\
\hline
problem & algorithm & Resp & Resd & fval& iter & time \\ \hline
\endhead
 $n$=8000 & {\bf JBP} & 7.90e-07 & 0 & -1.2344487e+04 & 26 & 3.35e-01 \\
 $m$=22800 & GSBP & 8.26e-07 & 0 & -1.2344486e+04 & 15 & 1.35e+01 \\
 $\sigma$= 1 & BADMM & 4.31e-07 & 8.27e-07& -1.2344498e+04 & 321 & 3.81e+00 \\
  & CCCP & 2.09e-07 & 9.98e-07 & -1.2344481e+04 & 915 & 8.93e+02 \\
 \hline
 $n$=8000 & JBP* & 4.69e-01* & 0 & -2.3002998e+04 & 10000 & 1.10e+02 \\
 $m$=22800 & GSBP* & 8.58e-01* & 0 & -2.3003425e+04 & 4237 & 3.60e+03 \\
 $\sigma$= 2 & {\bf BADMM} & 5.25e-07 & 9.98e-07& -2.3003083e+04 & 471 & 5.54e+00 \\
  & CCCP & 2.41e-07 & 9.95e-07 & -2.3003070e+04 & 1385 & 1.33e+03 \\
 \hline
 $n$=8000 & JBP* & 9.47e+01* & 0 & -5.6343253e+04 & 10000 & 1.09e+02 \\
 $m$=22800 & GSBP* & 9.28e+01* & 0 & -5.6437373e+04 & 4273 & 3.60e+03 \\
 $\sigma$= 5 & {\bf BADMM} & 5.03e-07 & 9.07e-07& -5.6441554e+04 & 711 & 8.02e+00 \\
  & CCCP & 3.81e-07 & 9.92e-07 & -5.6441537e+04 & 2570 & 2.46e+03 \\
 \hline
 $n$=125000 & {\bf JBP} & 8.45e-07 & 0 & -1.9678383e+05 & 39 & 5.21e+00 \\
 $m$=367500 & GSBP & 8.67e-07 & 0 & -1.9678383e+05 & 20 & 2.84e+02 \\
 $\sigma$= 1 & BADMM & 4.50e-07 & 8.57e-07& -1.9678385e+05 & 481 & 6.75e+01 \\
 & CCCP* & 4.93e-02* & 3.56e-01* & -1.9670712e+05 & 230 & 3.64e+03 \\
 \hline
 $n$=125000 & JBP* & 1.82e+01* & 0 & -3.6789200e+05 & 10000 & 1.28e+03 \\
 $m$=367500 & GSBP* & 2.36e+01* & 0 & -3.6791424e+05 & 254 & 3.60e+03 \\
 $\sigma$= 2 & {\bf BADMM} & 4.78e-07 & 9.09e-07& -3.6792269e+05 & 1761 & 2.44e+02 \\
  & CCCP* & 2.05e-01* & 9.45e-01* & -3.6772230e+05 & 225 & 3.60e+03 \\
 \hline
 $n$=125000 & JBP* & 1.18e+03* & 0 & -9.0248431e+05 & 10000 & 1.28e+03 \\
 $m$=367500 & GSBP* & 1.33e+03* & 0 & -9.0367298e+05 & 255 & 3.61e+03 \\
 $\sigma$= 5 & {\bf BADMM} & 5.26e-07 & 9.85e-07& -9.0393910e+05 & 3681 & 5.04e+02 \\
   & CCCP* & 3.95e-01* & 1.11e+00* & -9.0366947e+05 & 230 & 3.68e+03 \\
 \hline
 \end{longtable}
\end{footnotesize}
\end{center}

From Table~\ref{spin2d} and Table~\ref{spin3d}, we observe that the behavior of JBP and GSBP is unstable. Both methods either converge after only a few iterations or diverge. While GSBP requires fewer iterations than JBP, JBP is significantly faster. This speed advantage arises because JBP updates the messages in parallel, whereas GSBP updates them sequentially. In our implementation of JBP in {\sc Matlab}, we store all messages in an $r \times r \times m$ 3D array, allowing efficient tensor operations. In contrast, for GSBP, we use a cell array to store messages and rely on a for loop to update them. The use of for loops is known to be much slower than performing operations directly on tensors because {\sc Matlab} leverages multi-core parallel computing for matrix and tensor computations. 

The situation for BADMM and CCCP is similar. BADMM requires around 3 times fewer iterations than CCCP, and its running time is at least ten times, and sometimes even more than one hundred times, faster than that of CCCP. This demonstrates the critical importance of the parallelizability of an algorithm.

When comparing BADMM and BP, we observe that BP is indeed more efficient than BADMM when the former converges. However, the convergence of BP is not guaranteed
and its convergence is unpredictable. In contrast, 
BADMM is much more robust as it successfully solves all instances, whereas JBP and GSBP solve only half of the instances. Inspired by the efficient yet unstable behavior of BP, a practical strategy is to initially run BP for a few iterations to check for convergence. If BP does not converge, switching to BADMM can be a viable alternative. Although BADMM is not as efficient as BP, it offers much greater stability.

One aspect worth mentioning is that we tested Algorithm~\ref{alg:BADMMNnu} with Step 3 replaced by the linear dual update from Algorithm~\ref{alg:BADMM}. However, we observed that this modified algorithm failed to converge for any of the tested instances. Consequently, we do not present these results in the paper. This observation highlights that the nonlinear dual update in Algorithm~\ref{alg:BADMMN} is not only critical for theoretical analysis but also essential in practice.

\subsection{Example of (\ref{BVP}): sensor network localization}

In the second example, we test the Sensor Network Localization (SNL) problem \cite{mao2007wireless}, which aims to estimate the locations of sensors based on noisy observations of pairwise distances. Traditional approaches treat SNL as a nonlinear regression problem and employ algorithms such as gradient descent \cite{tang2024riemannian,liang2004gradient} and semidefinite programming \cite{biswas2004semidefinite,chaudhury2015large} to solve it. In \cite{ihler2004nonparametric}, Ihler et al. reformulated SNL as a graphical model and proposed the Nonparametric Belief Propagation (NBP) method as a 
solver.

Specifically, consider $n$ sensors $x_1, x_2, \ldots, x_n$ located in a planar region $[0,1]^2$. For any distinct $i, j \in [n]$, the noisy distance between sensors is observed with probability $P_o(x_i, x_j)$, given by:
\begin{equation}\label{SNLprob}
d_{ij} = \|x_i - x_j\| + \nu_{ij},\ \nu_{ij} \sim N(0, \sigma^2),\ P_o(x_i, x_j) := \exp\(-\frac{\|x_i - x_j\|^2}{2R^2}\),
\end{equation}
where $N(0, \sigma^2)$ is a normal distribution with standard deviation $\sigma > 0$, controlling the noise level, and $R > 0$ is a parameter determining the likelihood of an observation. Additionally, we assume the presence of $a$ anchors $y_1, y_2, \ldots, y_a$, whose locations are known. The noisy distance between a sensor $x_k$ and an anchor $y_\ell$ is observed with probability $P_o(x_k, y_\ell)$, given by:
\begin{equation}\label{SNLSA}
d_{k\ell} = \|x_k - y_\ell\| + \nu_{k\ell},\ \nu_{k\ell} \sim N(0, \sigma^2),\ P_o(x_k, y_\ell) := \exp\(-\frac{\|x_k - y_\ell\|^2}{2R^2}\).
\end{equation}
We construct a graph $G = ([n], \E)$, where $ij \in \E$ if and only if the distance between $x_i$ and $x_j$ is observed. Similarly, let $\E_1 \subset [n] \times [a]$, where $k\ell \in \E_1$ if and only if the distance between $x_k$ and $y_\ell$ is observed. With these definitions, we consider the graphical model in (\ref{MPDF}) with potential functions defined as:
\begin{equation}\label{defiPotenij}
\forall ij \in \mathcal{E},\ \Psi_{ij}(x_i, x_j) := P_o(x_i, x_j) \frac{\exp[-(d_{ij} - \|x_i - x_j\|)^2 / (2\sigma^2) ]}{\sqrt{2\pi\sigma^2}},
\end{equation}
\begin{equation}\label{defiPotenkl}
\forall k \in [n],\ \Psi_k(x_k) := \prod_{k\ell \in \mathcal{E}_1} P_o(x_k, y_\ell) \frac{\exp[-(d_{k\ell} - \|x_k - y_\ell\|)^2 / (2\sigma^2)]}{\sqrt{2\pi\sigma^2}}.
\end{equation}
Unlike previous graphical models with finite states, the potential functions in (\ref{defiPotenij}) and (\ref{defiPotenkl}) are continuous functions. To apply our algorithm, we discretize $[0,1]^2$ into a grid $\{0, 1/t, 2/t, \ldots, (t-1)/t, 1\}^2$ and restrict the domain of the potential functions to this discrete set. The resulting number of states is $r = (t+1)^2$. Note that this discretization introduces significant model error if $t$ is too small and dramatically increases problem size if $t$ is too large. Ihler et al. \cite{ihler2004nonparametric} addressed this issue by using NBP with stochastic approximation to solve the continuous model directly, avoiding discretization. Here, we use discretization solely to test the behavior of our algorithm. Extending Bregman ADMM to the continuous setting is an interesting direction for future research.

Akin to the experiments in \cite{ihler2004nonparametric}, we test SNL problems with and without outlier noise. The first experiment doesn't contain outlier noise. We set $n = 100$, $a = 4$, $t = 10$, $\sigma \in \{0.02, 0.05, 0.1\}$, and $R \in \{0.1, 0.2\}$. All sensors and anchors are uniformly distributed in $[0,1]^2$. We use the \texttt{rng('default')} command in {\sc Matlab} to ensure reproducibility of results. Since the discretization already introduces model error, it is unnecessary to solve the problems to high accuracy. We choose the tolerance to be $10^{-4}$ for all the methods. The results are shown in Table~\ref{SNL}.

\begin{center}
\begin{footnotesize}
\begin{longtable}{|c|c|ccc|cl|}
\caption{Comparison of BADMM, JBP, GSBP and CCCP for Sensor network localization without outlier noise}
\label{SNL}
\\
\hline
problem & algorithm & Resp & Resd & fval& iter & time \\ \hline
\endhead
 $n$=100 & JBP* & 3.98e+01* & 0 & 2.3809698e+03 & 10000 & 5.08e+02 \\ 
 $m$=266 & GSBP* & 5.50e+00* & 0 & 2.2703467e+03 & 10000 & 1.58e+03 \\ 
 $\sigma$=0.02 & {\bf BADMM} & 7.33e-05 & 9.68e-05& 2.2396171e+03 & 2181 & 1.39e+02 \\ 
  $R$=0.1& CCCP & 7.01e-08 & 9.96e-05 & 2.2412749e+03 & 7500 & 1.36e+03 \\ 
 \hline 
 $n$=100 & JBP* & 2.70e-01* & 0 & 2.1984481e+03 & 10000 & 5.73e+02 \\ 
 $m$=266 & {\bf GSBP} & 8.20e-05 & 0 & 2.1987160e+03 & 73 & 1.24e+01 \\ 
 $\sigma$=0.05 & BADMM & 3.76e-06 & 1.78e-05& 2.1987218e+03 & 781 & 5.58e+01 \\ 
  $R$=0.1& CCCP & 1.50e-07 & 1.00e-04 & 2.2004457e+03 & 4670 & 8.90e+02 \\ 
 \hline 
 $n$=100 & {\bf JBP} & 9.88e-05 & 0 & 2.1918324e+03 & 96 & 5.57e+00 \\ 
 $m$=266 & GSBP & 9.61e-05 & 0 & 2.1918182e+03 & 49 & 8.30e+00 \\ 
 $\sigma$=0.1 & BADMM & 2.50e-05 & 3.34e-05& 2.1918158e+03 & 1191 & 8.65e+01 \\ 
  $R$=0.1& CCCP & 1.02e-06 & 9.99e-05 & 2.1918183e+03 & 2465 & 4.63e+02 \\ 
 \hline 
 $n$=100 & JBP & 9.30e-07 & 0 & 7.8468525e+03 & 35 & 8.31e+00 \\ 
 $m$=860 & {\bf GSBP} & 2.01e-06 & 0 & 7.8708632e+03 &  6 & 3.07e+00 \\ 
 $\sigma$=0.02 & BADMM & 7.55e-05 & 2.37e-05& 7.3861696e+03 & 301 & 8.90e+01 \\ 
  $R$=0.2& CCCP* & 1.68e+40* & 3.68e+37* & 2.1335418e+08 & 7025 & 3.60e+03 \\ 
 \hline 
 $n$=100 & {\bf JBP} & 4.71e-05 & 0 & 7.0430711e+03 & 22 & 6.06e+00 \\ 
 $m$=860 & GSBP & 8.80e-05 & 0 & 7.0493708e+03 & 19 & 9.92e+00 \\ 
 $\sigma$=0.05 & BADMM & 7.83e-05 & 9.67e-05& 7.0457827e+03 & 1891 & 6.51e+02 \\ 
  $R$=0.2& CCCP* & 2.66e+37* & 3.37e+35* & 4.4979669e+07 & 6595 & 3.60e+03 \\ 
 \hline 
 $n$=100 & {\bf JBP} & 9.49e-05 & 0 & 7.1567571e+03 & 38 & 1.10e+01 \\ 
 $m$=860 & GSBP & 5.70e-05 & 0 & 7.1611669e+03 & 16 & 8.56e+00 \\ 
 $\sigma$=0.1 & BADMM & 8.01e-05 & 9.63e-05& 7.1611816e+03 & 731 & 2.65e+02 \\ 
  $R$=0.2& CCCP* & 3.46e+36* & 5.16e+33* & 1.7696345e+07 & 6575 & 3.60e+03 \\ 
 \hline
\end{longtable}
\end{footnotesize}
\end{center}

From Table~\ref{SNL}, we observe that the performance of JBP and GSBP is quite impressive, being nearly one hundred times faster than BADMM in the fifth example. However, they fail to converge in the first example, whereas BADMM successfully converges for all instances. In the fourth example, although both JBP and GSBP find nearly stationary points, their function values are higher than those obtained by BADMM. Additionally, Figure~\ref{figSNL} illustrates that the predicted locations by BADMM are indeed more accurate than those produced by JBP and GSBP. Regarding CCCP, this algorithm exhibits significant numerical issues in the last three instances, resulting in outputs that are not even probability vectors.

\begin{figure}[htbp]
\centerline{\includegraphics[height=8cm,width=10cm]{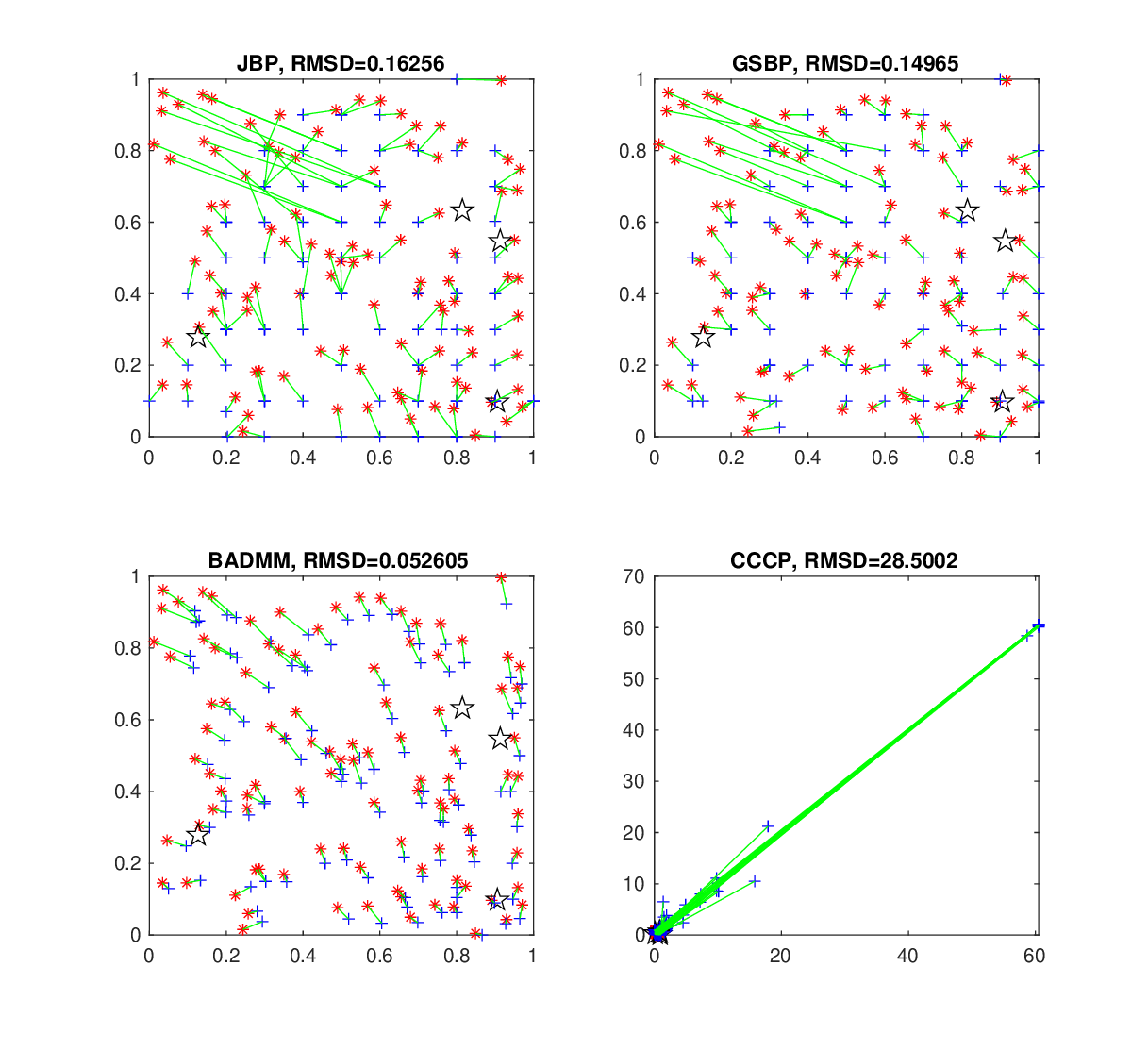}}
\caption{Results for SNL without outlier noise. $\sigma=0.02,R=0.2,m=100.$ \blue{+}'s are predicted sensors, \red{\ding{83}}'s are real sensors, \ding{73}'s are anchors. RMSD$=\sqrt{\sum_{k\in [n]}\|x_k-\hat{x}_k\|^2/n},$ where $\{\hat{x}_k:\ k\in [n]\}$ are predicted locations of sensors.} 
\label{figSNL}
\end{figure}

In the second experiment, we add outlier noise by randomly selecting $5\%$ of the distance measurement and replacing them by uniformly distributed random numbers in $[0,1].$ To avoid overfitting the outlier noise, we modify the potential as follows:
\begin{equation}\label{defiPotenijOL}
\forall ij \in \mathcal{E},\ \Psi_{ij}(x_i, x_j) := P_o(x_i, x_j) \frac{\exp[-|d_{ij} - \|x_i - x_j\|| / (2\sigma) ]}{\sqrt{2\pi\sigma}},
\end{equation}
\begin{equation}\label{defiPotenklOL}
\forall k \in [n],\ \Psi_k(x_k) := \prod_{k\ell \in \mathcal{E}_1} P_o(x_k, y_\ell) \frac{\exp[-|d_{k\ell} - \|x_k - y_\ell\|| / (2\sigma)]}{\sqrt{2\pi\sigma}},
\end{equation}
which are obtained by removing squares from (\ref{defiPotenij}) and (\ref{defiPotenklOL}). We set $n = 100$, $a = 4$, $t = 10$, $\sigma \in \{0.005, 0.01, 0.02\}$, and $R \in \{0.1, 0.2\}.$ The results are shown in Table~\ref{SNLOL}

\begin{center}
\begin{footnotesize}
\begin{longtable}{|c|c|ccc|cl|}
\caption{Comparison of BADMM, JBP, GSBP and CCCP for Sensor network localization with $5 \% $ outlier noise}
\label{SNLOL}
\\
\hline
problem & algorithm & Resp & Resd & fval& iter & time \\ \hline
\endhead
 $n$=100 & JBP* & 5.57e+00* & 0 & 2.3301846e+03 & 10000 & 5.91e+02 \\
 $m$=266 & GSBP* & 1.83e+01* & 0 & 2.2766568e+03 & 10000 & 1.69e+03 \\
 $\sigma$=5.00e-03 & {\bf BADMM} & 3.13e-05 & 9.78e-05& 2.2652299e+03 & 801 & 6.11e+01 \\
  $R$=1.00e-01& CCCP & 2.45e-07 & 9.94e-05 & 2.2652519e+03 & 4140 & 8.11e+02 \\
 \hline
 $n$=100 & JBP* & 6.69e+00* & 0 & 2.2389592e+03 & 10000 & 5.72e+02 \\
 $m$=266 & GSBP* & 2.37e+01* & 0 & 2.1908650e+03 & 10000 & 1.68e+03 \\
 $\sigma$=1.00e-02 & {\bf BADMM} & 3.75e-05 & 9.29e-05& 2.1954595e+03 & 481 & 3.64e+01 \\
  $R$=1.00e-01& CCCP & 2.74e-07 & 9.98e-05 & 2.1916139e+03 & 6425 & 1.24e+03 \\
 \hline
 $n$=100 & JBP* & 1.40e+00* & 0 & 2.1913097e+03 & 10000 & 5.73e+02 \\
 $m$=266 & GSBP* & 4.97e-01* & 0 & 2.1782500e+03 & 10000 & 1.67e+03 \\
 $\sigma$=2.00e-02 & {\bf BADMM} & 2.68e-05 & 9.33e-05& 2.1786968e+03 & 1041 & 7.62e+01 \\
  $R$=1.00e-01& CCCP & 7.39e-07 & 9.97e-05 & 2.1786950e+03 & 2605 & 5.09e+02 \\
 \hline
 $n$=100 & {\bf JBP} & 9.27e-05 & 0 & 8.3100066e+03 & 140 & 3.75e+01 \\
 $m$=860 & GSBP & 3.46e-05 & 0 & 8.3391879e+03 & 17 & 9.30e+00 \\
 $\sigma$=5.00e-03 & BADMM & 9.94e-05 & 6.34e-05& 8.0383921e+03 & 321 & 1.09e+02 \\
  $R$=2.00e-01& CCCP* & 5.39e+35* & 4.54e+33* & 3.2474887e+07 & 6290 & 3.60e+03 \\
 \hline
 $n$=100 & JBP* & 6.36e-01* & 0 & 7.2462420e+03 & 10000 & 2.92e+03 \\
 $m$=860 & GSBP* & 2.57e+00* & 0 & 7.2584189e+03 & 6487 & 3.60e+03 \\
 $\sigma$=1.00e-02 & {\bf BADMM} & 8.27e-05 & 4.97e-05& 7.2323807e+03 & 321 & 1.12e+02 \\
  $R$=2.00e-01& CCCP* & 3.69e+35* & 1.93e+33* & 1.9004665e+07 & 6310 & 3.60e+03 \\
 \hline
 $n$=100 & {\bf JBP} & 9.11e-05 & 0 & 6.9682013e+03 & 23 & 6.07e+00 \\
 $m$=860 & GSBP & 2.31e-05 & 0 & 6.9681874e+03 & 19 & 1.04e+01 \\
 $\sigma$=2.00e-02 & BADMM & 8.40e-05 & 5.03e-05& 6.9791008e+03 & 741 & 2.58e+02 \\
  $R$=2.00e-01& CCCP* & 2.54e+35* & 5.61e+32* & 1.3099666e+07 & 6315 & 3.60e+03 \\
 \hline
 \end{longtable}
\end{footnotesize}
\end{center}

From Table~\ref{SNLOL}, we observe that JBP and GSBP perform poorly for SNL problems with outlier noise, solving only 2 out of the 6 instances. In contrast, BADMM demonstrates high robustness, successfully solving all the problems. For CCCP, significant numerical issues arise when $R = 0.2$, as observed in the previous case. In the fourth example where JBP and GSBP are faster than BADMM, the function value obtained by BADMM is smaller. Furthermore, as shown in Figure~\ref{figSNLOL}, the predictions produced by BADMM are significantly more accurate than those of JBP and GSBP, underscoring the importance of finding a stationary point with a small function value.

\begin{figure}[htbp]
\centerline{\includegraphics[height=8cm,width=10cm]{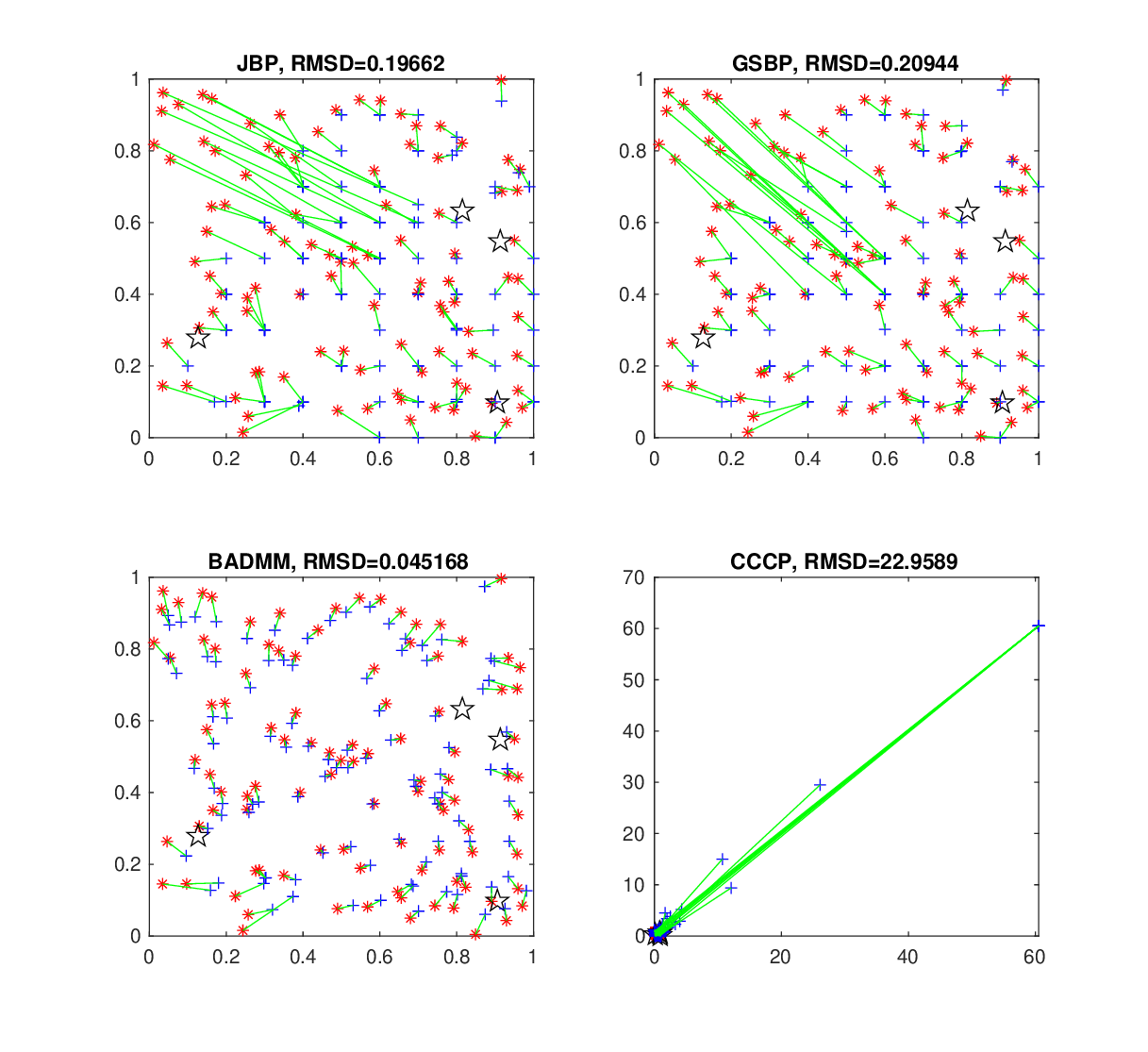}}
\caption{Results for SNL with $5\%$ outlier noise. $\sigma=0.005,R=0.2,m=100.$ \blue{+}'s are predicted sensors, \red{\ding{83}}'s are real sensors, \ding{73}'s are anchors. RMSD$=\sqrt{\sum_{k\in [n]}\|x_k-\hat{x}_k\|^2/n},$ where $\{\hat{x}_k:\ k\in [n]\}$ are predicted locations of sensors.} 
\label{figSNLOL}
\end{figure}

\subsection{Example of (\ref{QBVP}): 2D Ising models}

In the third example, we test the 2D Ising models in (\ref{QBVP}), which has been tested in \cite{zhao2024quantum}. We generate $G$ as an $n_1\times n_1$ 2D grid graph. The coefficient matrices are generated as 
\begin{align}\label{dataIsing}
&\forall k\in [n],\ \hc_k:=\frac{1}{T}\(h_x \begin{bmatrix}0&1\\1&0\end{bmatrix}-h_z\begin{bmatrix}1&0\\0&-1\end{bmatrix}\),\ \notag \\
&\forall ij\in \E,\ \hC_{ij}:=-\frac{J}{T}\cdot \begin{bmatrix}1&0\\ 0 & -1\end{bmatrix}\otimes \begin{bmatrix}1&0\\ 0 & -1\end{bmatrix},
\end{align}
where $h_x,h_z,J,T$ are parameters. We choose $n_1\in \{10,30\}$ and $[h_x,h_z,J]=[1.05,0.5,1]$ or $[2.5,0,-1]$ as suggested in \cite{zhao2024quantum}.

\begin{center}
\begin{footnotesize}
\begin{longtable}{|c|c|ccc|cl|}
\caption{Comparison of QBADMM, QJBP, QGSBP for Ising model, $h_x=1.05,h_z=0.5,J=1.$}
\label{Ising model 1}
\\
\hline
problem & algorithm & Resp & Resd & fval& iter & time \\ \hline
\endhead
 $n$=100 & {\bf QJBP} & 2.60e-07 & 0 & -2.4386471e+02 &  6 & 3.92e-01 \\
 $m$=180 & QGSBP & 7.08e-07 & 0 & -2.4386511e+02 & 24 & 1.28e+00 \\
 $T$=1 & QBADMM & 2.54e-07 & 4.15e-07& -2.4386494e+02 & 51 & 1.73e+00 \\
 \hline
 $n$=100 & QJBP* & - & - & - & -&-  \\
 $m$=180 & QGSBP & 9.88e-07 & 0 & -2.4416552e+03 & 436 & 2.11e+01 \\
 $T$=0.1 & {\bf QBADMM} & 8.53e-07 & 2.97e-07& -2.4416553e+03 & 171 & 5.31e+00 \\
 \hline
 $n$=100 & QJBP* & - & - & - & -&-  \\
 $m$=180 & QGSBP & 9.99e-07 & 0 & -2.4478861e+04 & 3194 & 1.53e+02 \\
 $T$=0.01 & {\bf QBADMM} & 9.71e-07 & 2.19e-07& -2.4478820e+04 & 591 & 1.81e+01 \\
 \hline
 $n$=900 & {\bf QJBP} & 4.50e-07 & 0 & -2.3050728e+03 &  6 & 2.31e+00 \\
 $m$=1740 & QGSBP & 3.63e-07 & 0 & -2.3050734e+03 & 25 & 1.09e+01 \\
 $T$=1 & QBADMM & 1.27e-07 & 1.96e-07& -2.3050722e+03 & 61 & 1.80e+01 \\
 \hline
 $n$=900 & QJBP* & - & - & - & -&-  \\
 $m$=1740 & QGSBP & 9.99e-07 & 0 & -2.3076988e+04 & 2415 & 1.16e+03 \\
 $T$=0.1 & {\bf QBADMM} & 9.28e-07 & 2.80e-07& -2.3076992e+04 & 661 & 1.99e+02 \\
 \hline
 $n$=900 & QJBP* & - & - & - & -&-  \\
 $m$=1740 & QGSBP* & 1.21e-04* & 0 & -2.3130566e+05 & 7355 & 3.60e+03 \\
 $T$=0.01 & {\bf QBADMM} & 9.84e-07 & 2.31e-07& -2.3130470e+05 & 2211 & 6.62e+02 \\
 \hline
\end{longtable}
\end{footnotesize}
\end{center}

\begin{center}
\begin{footnotesize}
\begin{longtable}{|c|c|ccc|cl|}
\caption{Comparison of QBADMM, QJBP, QGSBP for Ising model, $h_x=2.5,h_z=0,J=-1.$}
\label{Ising model 2}
\\
\hline
problem & algorithm & Resp & Resd & fval& iter & time \\ \hline
\endhead
 $n$=100 & QJBP* & 5.81e+01* & 0 & -1.0040858e+02 & 10000 & 4.12e+02 \\
 $m$=180 & {\bf QGSBP} & 7.06e-07 & 0 & -2.7606344e+02 & 32 & 1.54e+00 \\
 $T$=1 & QBADMM & 1.01e-08 & 3.36e-08& -2.7606705e+02 & 51 & 1.70e+00 \\
 \hline
 $n$=100 & {\bf QJBP} & 3.50e-07 & 0 & -2.5000487e+03* & 13 & 4.80e-01 \\
 $m$=180 & QGSBP & 9.93e-07 & 0 & -3.0046956e+03 & 902 & 4.25e+01 \\
 $T$=0.1 & QBADMM & 7.54e-07 & 2.58e-07& -3.0046879e+03 & 341 & 1.10e+01 \\
 \hline
 $n$=100 & {\bf QJBP} & 1.72e-07 & 0 & -2.5000342e+04* & 13 & 4.68e-01 \\
 $m$=180 & QGSBP* & 1.68e-05* & 0 & -3.0726971e+04 & 10000 & 4.63e+02 \\
 $T$=0.01 & QBADMM & 9.78e-07 & 2.05e-07& -3.0725593e+04 & 2741 & 8.48e+01 \\
 \hline
 $n$=900 & JQBP* & 5.14e+02* & 0 & -5.2905916e+02 & 9295 & 3.60e+03 \\
 $m$=1740 & QGSBP & 6.18e-07 & 0 & -2.5067091e+03 & 35 & 1.57e+01 \\
 $T$=1 & {\bf QBADMM} & 1.33e-07 & 4.38e-07& -2.5067207e+03 & 51 & 1.49e+01 \\
 \hline
 $n$=900 & {\bf QJBP} & 8.25e-07 & 0 & -2.2500146e+04* & 13 & 4.19e+00 \\
 $m$=1740 & QGSBP & 9.99e-07 & 0 & -2.7634677e+04 & 3083 & 1.36e+03 \\
 $T$=0.1 & QBADMM & 9.70e-07 & 2.14e-07& -2.7634590e+04 & 1111 & 3.25e+02 \\
 \hline
 $n$=900 & {\bf QJBP} & 4.02e-07 & 0 & -2.2500103e+05* & 13 & 4.20e+00 \\
 $m$=1740 & QGSBP* & 1.47e-03* & 0 & -2.8361552e+05 & 7824 & 3.60e+03 \\
 $T$=0.01 & QBADMM* & 2.55e-06* & 5.25e-07& -2.8357454e+05 & 10000 & 2.92e+03 \\
 \hline
 \end{longtable}
\end{footnotesize}
\end{center}

From Table~\ref{Ising model 1}, we observe that QBADMM is the only algorithm capable of solving all the instances. QJBP encounters numerical issues, and its output becomes \texttt{NaN} in {\sc Matlab}, so we do not include its results in the table. This numerical instability has also been reported in the experiments conducted by Zhao et al. in \cite{zhao2024quantum}. Interestingly, QGSBP does not exhibit such numerical issues and is often able to find solutions with acceptable accuracy. 

It is worth noting that when solving (\ref{QBVP}), the algorithms perform $\logm(\cdot)$ and $\expm(\cdot)$ operations, which cannot be directly executed on 3D tensors in {\sc Matlab}. As a result, we use a for loop to conduct these operations sequentially across all the algorithms. This approach means that we have not yet exploited the parallelizability of QBADMM. We believe its efficiency could be significantly improved by implementing it on GPUs using parallel computing.

From Table~\ref{Ising model 2}, we observe that QJBP solves 4 of the problems with the shortest runtime, without encountering the previous numerical issues. However, its function value is significantly higher than those of GSBP and QBADMM, suggesting that QJBP may still face numerical challenges that cause it to converge to suboptimal local minimizers. Comparing QGSBP and QBADMM, we note that QBADMM generally outperforms QGSBP. In the final example, where both methods terminate before achieving the required solution accuracy, the output of QBADMM is more accurate than that of QGSBP. Furthermore, QBADMM does not reach the maximum running time in this case; if we continue running QBADMM, it is expected to find a solution with a KKT residual smaller than $10^{-6}$.

%%%%%%%%%%%%%%%%%%%Conclusion%%%%%%%%%%%%%%%%%%%%%%%%%%%%%%%%%%%%%%%%%%%%%%%%%%%%%%%%%%%%%%%%%%%%%%%%%%%%%%%%%%%%%%%%%%%%%%%%%%%%%%%%%%%%%%%%%%%%
\section{Conclusion}\label{Sec:conc}

In this work, we proposed a Bregman ADMM to solve the Bethe variational problem and its quantum counterpart. Despite the non-smooth and non-convex nature of this optimization problem, we established the convergence properties of
our Bregman ADMM by leveraging the problem's special structure. A key step in our analysis involves proving that the entries of local minima are bounded below by a positive constant, which prevents the iterates of the Bregman ADMM from approaching the boundary and enables the exploitation of the objective function's smoothness. Extensive numerical experiments demonstrate that our Bregman ADMM is highly robust compared to the well-known belief propagation method, which does not have convergence guarantee for a general graph. 
For future work, an intriguing direction would be to extend the theoretical results to the quantum Bethe variational problem. Another promising avenue is to apply our algorithm to tackle more challenging problems in scientific computing.

\section*{Acknowledgments.}
% Enter the text of acknowledgments here
The authors would like to thank Dr. Nachuan Xiao for many helpful discussions on the convergence analysis of Bregman ADMM.

\bibliographystyle{abbrv}
\bibliography{BVP}

\end{document}